\documentclass[12pt,a4paper]{amsart}

\usepackage{amssymb,amsfonts,amsmath}
\usepackage[mathscr]{eucal}
\usepackage[hmargin=3cm,vmargin=3cm]{geometry}

\makeindex

\newtheorem{thm}{{{Theorem}}}[section]
\newtheorem{prop}[thm]{{Proposition}}
\newtheorem{lem}[thm]{{Lemma}}
\newtheorem{cor}[thm]{{Corollary}}

\newtheorem{cond}[thm]{{Condition}}
\newtheorem{exa}[thm]{{Example}}
\numberwithin{equation}{section}

\def\N{\mathbb{N}}
\def\Z{\mathbb{Z}}
\def\Q{\mathbb{Q}}
\def\R{\mathbb{R}}
\def\C{\mathbb{C}}
\def\A{\mathbb{A}}

\def\GL{{\mathop{\mathrm{GL}}}}

\def\SL{{\mathop{\mathrm{SL}}}}

\def\SO{{\mathop{\mathrm{SO}}}}
\def\U{{\mathop{\textnormal{U}}}}
\def\SU{{\mathop{\textnormal{SU}}}}
\def\Sp{{\mathop{\mathrm{Sp}}}}

\def\sl{{\mathfrak{sl}}}

\def\bK{{\mathbf{K}}}

\def\Re{{\mathop{\mathrm{Re}}}}

\def\Tr{{\mathop{\mathrm{Tr}}}}

\def\diag{{\mathop{\mathrm{diag}}}}
\def\sgn{{\mathop{\mathrm{sgn}}}}
\def\Ad{{\mathop{\mathrm{Ad}}}} 
\def\vol{{\mathop{\mathrm{vol}}}}

\def\d{{\mathrm{d}}}

\def\bsl{\backslash}
\def\inf{\infty}

\def\fo{{\mathfrak{o}}}

\def\cO{{\mathcal{O}}}
\def\cS{{\mathcal{S}}}
\def\cC{{\mathcal{C}}}
\def\cH{{\mathcal{H}}}
\def\cI{{\mathcal{I}}}

\def\fa{{\mathfrak{a}}}
\def\cL{{\mathcal{L}}}
\def\cP{{\mathcal{P}}}
\def\cF{{\mathcal{F}}}
\def\cU{{\mathcal{U}}}

\def\fg{{\mathfrak{g}}}
\def\ft{{\mathfrak{t}}}
\def\fh{{\mathfrak{h}}}
\def\fk{{\mathfrak{k}}}
\def\fm{{\mathfrak{m}}}
\def\fn{{\mathfrak{n}}}

\def\fu{{\mathfrak{u}}}

\def\fsp{{\mathfrak{sp}}}

\def\fH{{\mathfrak{H}}}

\def\fin{{\mathrm{fin}}}
\def\reg{{\mathrm{reg}}}

\def\fs{{f_\sigma}}
\def\tfs{{\tilde{f}_\sigma}}

\def\unip{{\mathrm{unip}}}

\def\bm{\mathbf{m}}
\def\bC{\mathbf{C}}

\def\bb{\boldsymbol{\beta}}
\def\bz{\mathbf{z}}
\def\Shin{\mathrm{Shin}}

\def\bk{\boldsymbol{k}}

\title{The dimensions of spaces of Siegel cusp forms of general degree}

\author{Satoshi Wakatsuki}
\address{Faculty of Mathematics and Physics, Institute of Science and Engineering, Kanazawa University, Kakumamachi, Kanazawa, Ishikawa, 920-1192, Japan}
\email{wakatsuk@staff.kanazawa-u.ac.jp}


\begin{document}

\maketitle

\begin{abstract}

In this paper, we give a dimension formula for spaces of Siegel cusp forms of general degree with respect to neat arithmetic subgroups.
The formula was conjectured before by several researchers (cf. \cite{Ib1,IS2,IS3,Tsushima3}).
The dimensions are expressed by special values of Shintani zeta functions for spaces of symmetric matrices at non-positive integers.
This formula was given by Shintani for only a small part of the geometric side of the trace formula (see \cite{Shintani}).
To be precise, it is the contribution of unipotent elements corresponding to the partitions $(2^j,1^{2n-2j})$, where $n$ denotes the degree and $0\leq j \leq n$.
Hence, our work is to show that all the other contributions vanish.
In addition, one finds that Shintani's formula means the dimension itself.
Combining our formula and an explicit formula of the Shintani zeta functions, which was discovered by Ibukiyama and Saito (cf. \cite{IS1,IK,Saito2}), we can derive an explicit dimension formula for the principal congruence subgroups of level greater than two (cf. \cite{IS2,IS3}).
In this explicit dimension formula, the dimensions are described by degree $n$, weight $k$, level $N$, and the Bernoulli numbers $B_m$.
\end{abstract}

\tableofcontents

\section{Introduction}\label{s1}

The purpose of this paper is to give a dimension formula for spaces of Siegel cusp forms of general degree with respect to neat arithmetic subgroups.
Our formula expresses the dimensions by special values of Shintani zeta functions for spaces of symmetric matrices at non-positive integers.
Morita was the first to make a connection between the dimensions for degree two and the Shintani zeta function of the space of $2\times 2$ symmetric matrices (cf. \cite{Morita}).
After that, Shintani found a clearer relation for general degree by using the functional equations of the Shintani zeta functions for spaces of symmetric matrices.
His formula expresses the contributions of unipotent elements corresponding to the partitions $(2^j,1^{2n-2j})$ by the special values, where $n$ denotes the degree and $0\leq j \leq n$ (see \cite[Section 3]{Shintani}).
Furthermore, based on known dimension formulas \cite{Christian,Morita,Yamazaki,Tsushima1} for degrees two and three, several researchers conjectured that all the other contributions vanish (cf. \cite{Ib1,IS2,IS3,Tsushima3}).
In this paper, we solve this conjecture in the affirmative.
Therefore, one finds that Shintani's formula means the dimension itself.

Special values of the Shintani zeta functions can be studied by an explicit formula of Ibukiyama and Saito (cf. \cite{IS1,IK,Saito2}).
They discovered that the zeta functions are expressed by the Riemann zeta function and the Mellin transform of some Eisenstein series of half integral weights.
Furthermore, they presented a conjectural explicit dimension formula for the principal congruence subgroups of level greater than two, which was derived from their explicit formula and the above mentioned conjecture (see \cite{IS2,IS3} and Section \ref{sef}).
Therefore, the explicit dimension formula is proved by their work and our result.
In the formula, the dimensions are described by degree $n$, weight $k$, level $N$, and the Bernoulli numbers $B_m$, that is, they can be calculated numerically.

In the proof, we begin by applying a pseudo-coefficient of a holomorphic discrete series to Arthur's invariant trace formula.
Then, using some arguments similar to \cite{Arthur4}, we can show that the dimension is expressed by unipotent weighted orbital integrals and global coefficients.
It is probably difficult to calculate all them directly at the present time.
In order to overcome such difficulties, we change the pseudo-coefficient into the spherical trace function (a sum of matrix coefficients) of the discrete series via the trace Paley-Winner theorem.
This is a key step of our proof.
By Godement's explicit construction, we finds an interesting property (cf. Lemma \ref{vl1}) of the spherical trace function, which is one of the main reasons for the vanishing.
Using both the pseudo-coefficient and the spherical trace function, we prove the vanishing of unipotent contributions which do not correspond to any partition of the form $(2^j,1^{2n-2j})$.
Furthermore, by using the work of Finis-Lapid \cite{FL} for an estimate of the geometric side, we can connect the remaining terms with truncated zeta integrals of the spherical trace function.
Therefore, the explicit calculation is reduced to Shintani's formula.
In order to solve some convergence problems related to them, it is required to assume that weights are sufficiently large.
However, we can extend the weight range of the dimension formula by Ibukiyama's argument \cite{Ibukiyama}, because the dimensions are expressed by a polynomial with respect to weight $k$.

It is well-known that the dimension of a space of holomorphic cusp forms is equal to the multiplicity of the corresponding holomorphic discrete series in the discrete spectrum of the $L^2$-space on an arithmetic quotient of a semisimple Lie group.
In the case that arithmetic quotients are compact, Langlands obtained a general multiplicity formula for integrable discrete series (cf. \cite{Langlands1,Langlands2}).
For the case $\R$-rank one, multiplicity formulas were obtained by Selberg \cite{Selberg} ($\SL(2)$) and Osborne-Warner \cite{OW}.
Hence, our main interest is about non-compact arithmetic quotients and higher rank groups.
In such cases, there are various general multiplicity formulas for discrete series, see, e.g., Arthur \cite{Arthur4}, Ferrari \cite{Fe}, Goresky-Kottwitz-MacPherson \cite{GKM}, Gross-Pollack \cite{GP}, Gross \cite{Gross}, and Stern \cite{Stern1,Stern}.
For the multiplicity of a single discrete series, a fine formula is not established yet, but it will be probably provided by the stable trace formula in the future (cf. \cite{Spallone}).
As a remarkable point, although the stable trace formula will make such a formula, our dimension formula is much simpler than it.
Perhaps, this is a characteristic phenomenon of holomorphic discrete series.

\subsection{Dimension formula}\label{smt}

We shall give a summary for our results.
Let $\Sp(n)$ denote the split symplectic group of rank $n$ over the rational number field $\Q$, i.e.,
\begin{equation}\label{eqsp}
\Sp(n)=\left\{g\in\GL(2n) \mid g\begin{pmatrix}O_n&-I_n \\ I_n&O_n \end{pmatrix} {}^t\!g = \begin{pmatrix}O_n&-I_n \\ I_n&O_n \end{pmatrix}  \right\}
\end{equation}
where $I_n$ (resp. $O_n$) denotes the unit (resp. zero) matrix of degree $n$ and ${}^t\!g$ denotes the transpose of $g$.
Throughout this section, we set $G=\Sp(n)$.
We write $\A$ (resp. $\A_\fin$) for the adele ring (resp. finite adele ring) of $\Q$.
Fix an open compact subgroup $K_0$ of $G(\A_\fin)$.
Then, the intersection $\Gamma=G(\Q)\cap(G(\R)K_0)$ is an arithmetic subgroup  of $G(\Q)$ via the diagonal embedding $G(\Q)\subset G(\A)=G(\R)\times G(\A_\fin)$. 
Let $S_k(\Gamma)$ denote the space of holomorphic Siegel cusp forms of weight $k$ with respect to $\Gamma$ (see Section \ref{s3.2} for the definition).

We define Shintani zeta functions by $G$ and $K_0$.
For each finite place $v$ of $\Q$, we write $\Q_v$ for the completion of $\Q$ at $v$ and denote by $\Z_v$ the ring of integers of $\Q_v$.
Let $\d k$ denote the Haar measure on $\bK_\fin=\prod_{v<\inf} G(\Z_v)$ normalized by $\int_{\bK_\fin}\d k=1$ and let $h_0$ denote the characteristic function of $K_0$ on $G(\A_\fin)$.
For $1\leq r\leq n$, we denote by $V_r$ the vector space of symmetric matrices of degree $r$ over $\Q$.
Then, a compactly supported smooth function $\phi_{0,r}^*$ on $V_r(\A_\fin)$ is defined by
\begin{equation}\label{sme1}
\phi_{0,r}^*(x)=\int_{\bK_\fin}h_0(k^{-1}\begin{pmatrix}I_n& \begin{matrix}-2x&0\\ 0&0 \end{matrix} \\ O_n&I_n \end{pmatrix}k)\, \d k \qquad (x\in V_r(\A_\fin)).
\end{equation}
A Fourier transform of $\phi_{0,r}^*$ is required for our formula.
For each place $v<\inf$, we choose a non-trivial additive character $\tilde\psi_v$ on $\Q_v$ which is defined in Section \ref{s3.4}.
Then, we set $\psi_v(x_v)=\tilde\psi_v(\Tr(x_v))$ $(x_v\in V_r(\Q_v))$ and we get a non-trivial additive character $\psi_\fin=\prod_{v<\inf}\psi_v$ on $V_r(\A_\fin)$.
A compactly supported smooth function $\phi_{0,r}$ on $V_r(\A_\fin)$ is given by the Fourier transform
\begin{equation}\label{sme2}
\phi_{0,r}(y)=2^{-\frac{r(r-1)}{2}} \int_{V_r(\A_\fin)}\phi_{0,r}^*(x)\, \psi_\fin(-xy)\, \d x
\end{equation}
where $\d x$ denotes the Haar measure on $V_r(\A_\fin)$ normalized so that the volume of $\prod_{v<\inf}V_r(\Z_v)$ is $1$.
The group $\GL(r)$ acts on $V_r$ as $x\cdot g={}^t\!g xg$.
Let $U_r$ denote the support of $\phi_{0,r}$ in $V_r(\A_\fin)$.
Set $\Omega_r=\{x\in V_r(\R) \mid x>0\}$ and $\Gamma_r=\SL(r,\Z)$ where $\Z$ denotes the ring of integers.
By the definition one sees that $L_r=V_r(\Q)\cap (V_r(\R)U_r)$ is $\Gamma_r$-invariant.
For $1\leq r\leq n$ and $s\in\C$, a zeta function $\zeta_\Shin(\phi_{0,r},s)$ is defined by
\begin{equation}\label{sme3}
\zeta_\Shin(\phi_{0,r},s)=\frac{1}{2}\sum_{x\in L_r\cap\Omega_r/\Gamma_r}\frac{\phi_{0,r}(x)}{\varepsilon_r(x)\, \det(x)^s} 
\end{equation}
where we put $\varepsilon_r(x)=|\{\gamma\in\Gamma_r\mid x\cdot \gamma=x\}|$ for $x\in L_r\cap\Omega_r$.
We call $\zeta_\Shin(\phi_{0,r},s)$ a Shintani zeta function for $(\GL(r),V_r)$.
It is known that $\zeta_\Shin(\phi_{0,r},s)$ absolutely converges for $\mathrm{Re}(s)>\frac{r+1}{2}$ and is meromorphically continued to the whole $s$-plane (see \cite{Shintani,Yukie}).
In particular, $\zeta_\Shin(\phi_{0,r},s)$ is holomorphic except for possible simple poles at $s=1,\frac{3}{2},\cdots,\frac{r}{2},\frac{r+1}{2}$.
For $r=0$, we set $\zeta_\Shin(\phi_{0,r},s)=1$.

The following theorem is deduced from \eqref{e38} in Section \ref{s3.4} and Theorem \ref{t37} (main theorem) in Section \ref{s3.5}.
\begin{thm}\label{t11}
For every $\gamma$ in $\Gamma$, we assume that $\gamma$ is unipotent if there exists a natural number $l$ such that $\gamma^l$ is unipotent.
(Note that this assumption holds if $\Gamma$ is neat in the sense of Borel \cite{Borel1}.)
If $k>n+1$, then we have
\begin{align*}
\dim S_k(\Gamma)=& \frac{[\bK_\fin:K_0\cap\bK_\fin]}{[K_0:K_0\cap\bK_\fin]} \sum_{r=0}^n \zeta_\Shin(\phi_{0,r},r-n) \\
&\times \prod_{j=1}^{n-r}\frac{(-1)^j(j-1)!}{(2j-1)!}\zeta(1-2j) \times 2^{-2n+r}\prod_{t=1}^{n-r}\prod_{u=t+r}^n(2k-t-u)
\end{align*}
where $\zeta(s)$ denotes the Riemann zeta function.
\end{thm}

Let $N$ be a natural number.
For each finite place $v$ of $\Q$, we set $K_v(N)=\{\gamma\in G(\Z_v)\mid \gamma\equiv I_{2n} \mod N\Z_v\}$ if $v$ divides $N$, and $K_v(N)=G(\Z_v)$ if $v$ does not divide $N$.
Then the group $K(N)=\prod_{v<\inf}K_v(N)$ is open and compact in $G(\A_\fin)$ and the principal congruence subgroup $\Gamma_n(N)$ of level $N$ is obtained as
\begin{equation}\label{sme4}
\Gamma_n(N)=G(\Q)\cap(G(\R)K(N)).
\end{equation}
It is well-known that $\Gamma_n(N)$ is neat when $N$ is greater than $2$.
The group $\Gamma_n(N)$ is a typical example of neat arithmetic subgroups.
For $1\leq r\leq n$, we denote by $L_r^*$ the lattice of half-integral symmetric matrices in $V_r(\Q)$, i.e.,
\[
L_r^*=\{(x_{jl})\in V_r(\Q)\mid x_{jj}\in\Z, \;\; x_{jl}\in\frac{1}{2}\Z\;\; (j<l)\},
\]
and we set
\begin{equation}\label{sme5}
\zeta_\Shin(L_r^*,s)=\frac{1}{2}\sum_{x\in L^*_r\cap\Omega_r/\Gamma_r}\frac{1}{\varepsilon_r(x) \, \det(x)^s}.
\end{equation}
We formally set $\zeta_\Shin(L_0^*,s)=1$ $(r=0)$.
If $K_0=K(N)$, then
\[
\zeta_\Shin(\phi_{0,r},s)=2^{r-sr}N^{-\frac{r(r+1)}{2}+sr}\times \zeta_\Shin(L_r^*,s) \qquad (0\leq r\leq n).
\]
Applying $K_0=K(N)$ to Theorem \ref{t11}, one can get the following formula.
\begin{cor}\label{c12}
If $k>n+1$ and $N>2$, then we have
\begin{align*}
\dim S_k(\Gamma_n(N))= & [\Gamma_n(1):\Gamma_n(N)]\sum_{r=0}^n \zeta_\Shin(L_r^*,r-n) \times 2^{r-r^2+rn}N^{\frac{r(r-1)}{2}-rn} \\
&\times \prod_{j=1}^{n-r}\frac{(-1)^j(j-1)!}{(2j-1)!}\zeta(1-2j)  \times 2^{-2n+r}\prod_{t=1}^{n-r}\prod_{u=t+r}^n(2k-t-u),
\end{align*}
where $[\Gamma_n(1):\Gamma_n(N)]=N^{n(2n+1)}\prod_{p:\mathrm{prime},\; p|N}\prod_{l=1}^n(1-p^{-2l})$.
\end{cor}

We should mention Shintani's study for unipotent terms.
Let $\sigma$ denote the holomorphic discrete series of $G(\R)$ related to the space $S_k(\Gamma)$ (see Sections \ref{s3.1} and \ref{secap}).
We fix a Haar measure $\d g$ on $G(\R)$ and denote by $d_\sigma$ the formal degree of $\sigma$.
By Harish-Chandra's formula, $d_\sigma$ equals $\prod_{1\leq t\leq u\leq n}(2k-t-u)$ up to a constant multiple.
Let $\fs$ denote the spherical trace function of $\sigma$, which satisfies $\fs(I_{2n})=d_\sigma$ (see Sections \ref{s25} and \ref{s3.1}).
In \cite[Expos\'e 10]{Godement}, Godement showed the equality
\[
\dim S_k(\Gamma)= \int_{\Gamma\bsl G(\R)} \sum_{\gamma\in \Gamma} \fs (g^{-1}\gamma g) \, \d g \qquad (k>2n),
\]
which we call Godement's formula.
Let $\Pi_r$ denote the subset of $\Gamma_n(N)$ consisting of elements $\gamma$ such that $\gamma$ is $G(\Z)$-conjugate to a matrix of the form $\begin{pmatrix}I_n & X \\ O_n & I_n \end{pmatrix}$ where $X$ is an integral symmetric matrix of degree $n$ and rank $r$.
Set
\[
\mathfrak{I}_n(\Pi_r,N,k)= \int_{\Gamma_n(N)\bsl G(\R)} \sum_{\gamma\in \Pi_r} \fs (x^{-1}\gamma x) \, \d x.
\]
In \cite[Section 3]{Shintani}, he showed that the integral $\int_{\Gamma_n(N)\bsl G(\R)}\Big| \sum_{\gamma\in \Pi_r} \fs (x^{-1}\gamma x) \Big| \,  \d x$ is convergent and $\mathfrak{I}_n(\Pi_r,N,k)$ equals the term for $r$ in the right hand side of the equality of Corollary \ref{c12}.
Namely, Corollary \ref{c12} means that the total $\mathfrak{I}_n(\Pi,N,k)=\sum_{r=0}^n \mathfrak{I}_n(\Pi_r,N,k)$ coincides with $\dim S_k(\Gamma_n(N))$ if $N>2$.
He called $\mathfrak{I}_n(\Pi,N,k)$ the contribution of purely parabolic conjugacy classes to the dimension formula.
In the papers \cite{IS2,IS3}, $\mathfrak{I}_n(\Pi,N,k)$ was named the contribution of central unipotent elements to the dimension formula.

We fix a finite set $S$ of finite places of $\Q$ and choose an open compact subgroup $K_S$ in $G(\Q_S)$.
Assume that $N$ moves over natural numbers prime to $S$.
An arithmetic subgroup $\Gamma_n(S,N)$ is defined by
\[
\Gamma_n(S,N)=G(\Q)\cap (G(\R)K_S\prod_{v\not\in S\sqcup\{\inf\}}K_v(N)).
\]
If $S$ is the empty set, then $\Gamma_n(S,N)=\Gamma_n(N)$.
The limit multiplicity formula states that $\vol(\Gamma_n(S,N)\bsl G(\R))^{-1}\times \dim S_k(\Gamma_n(S,N))$ tends to $d_\sigma$ as $N\to\inf$ (cf. \cite{Savin}).
As a consequence of Theorem \ref{t11} and its proof, one can obtain an estimate of the remainder term.
\begin{cor}\label{cor2}
There exists a sufficiently large $N'$ such that, for $k>n+1$ and $N>N'$ ($N$ is prime to $S$), we have
\[
\dim S_k(\Gamma_n(S,N))=\vol(\Gamma_n(S,N)\bsl G(\R))\, d_\sigma + O(N^{2n^2}k^{n(n-1)/2})
\]
if $n$ is even, and
\[
\dim S_k(\Gamma_n(S,N))=\vol(\Gamma_n(S,N)\bsl G(\R))\, d_\sigma + O(N^{2n^2-n+1}k^{(n-1)(n-2)/2}) 
\]
if $n$ is odd.
Note that $\vol(\Gamma_n(S,N)\bsl G(\R))=O(N^{n(2n+1)})$ and $d_\sigma=O(k^{n(n+1)/2})$.
In particular, for $k>n+1$ and $N>2$, we have
\[
\dim S_k(\Gamma_n(N))=\vol(\Gamma_n(N)\bsl G(\R))\, d_\sigma + O(N^{2n^2}k^{n(n-1)/2})
\]
if $n$ is even, and
\[
\dim S_k(\Gamma_n(N))=\vol(\Gamma_n(N)\bsl G(\R))\, d_\sigma + O(N^{2n^2-n+1}k^{(n-1)(n-2)/2}) 
\]
if $n$ is odd.
\end{cor}
\begin{proof}
The assumption of Theorem \ref{t11} for $\Gamma$ is required for Proposition \ref{p1}.
It can be replaced by the condition that $N$ is sufficiently large (cf. \cite[Lemma 5 and its proof]{Clozel}).
Furthermore, the main term obviously comes from the term of $r=0$ in Theorem \ref{t11}.
For the term of $r=1$ in Theorem \ref{t11}, when $n$ is odd and $n>1$, we deduce $\zeta_\Shin(\phi_{0,1},1-n)=0$ from the trivial zeros of the Riemann zeta function.
Hence, the second term becomes the term of $r=1$ if $n$ is even or $n=1$, and the term of $r=2$ if $n$ is odd and $n>1$.
Thus, the first assertion follows from Theorem \ref{t11}.
The second assertion is obviously deduced from Corollary \ref{c12}.
\end{proof}
It would be interesting to find an asymptotic behavior of Hecke eigenvalues as a generalization of Corollary \ref{cor2}.

\subsection{Explicit formula}\label{sef}

In order to get numerical numbers of dimensions by Corollary \ref{c12}, one needs to calculate the special values $\zeta_\Shin(L_r^*,r-n)$ explicitly.
If $r=1$, then it is obvious that
\[
\zeta_\Shin(L_1^*,1-n)=\frac{1}{2}\zeta(1-n)=-\frac{B_n}{2n}
\]
where $B_n$ denotes the $n$-th Bernoulli number, i.e., they are defined by
\[
\frac{t e^t}{e^t-1}=\sum_{m=0}^\inf \frac{B_m}{m!}t^m .
\]
For the case $r=2$, the spacial values $\zeta_\Shin(L_2^*,2-n)$ were given by Siegel in \cite{Siegel} (Shintani also gave an alternative proof in \cite{Shintani}).
\begin{thm}[Siegel]\label{t13}
If $r=2$, then $\zeta_\Shin(L_2^*,0)=1/96$ and
\[
\zeta_\Shin(L_2^*,2-n)= -\frac{(-1)^nB_{2n-2}}{2^{2n}(n-1)} \qquad \text{if $n>2$}.
\]
\end{thm}

For the case $r\geq 3$, the special values $\zeta_\Shin(L_r^*,r-n)$ were calculated by an explicit formula \cite[Theorems 1.2 and 1.3]{IS1} of Ibukiyama and Saito for the Shintani zeta functions (see \cite{IK,Saito2} for alternative proofs).
When $r$ is odd, $\zeta_\Shin(L_r^*,s)$ is described by the Riemann zeta function (it is a sum of two Euler products).
When $r$ is even, $\zeta_\Shin(L_r^*,s)$ is described by the Riemann zeta function and the Mellin transform of some Eisenstein series of half integral weights.
Note that $\zeta_\Shin(L_2^*,s)$ $(r=2)$ coincides with the Mellin transform of a Eisenstein series of weight $3/2$.
From the explicit formula, they derived the following formula for the special values (see \cite[Theorems 3 and 4]{IS2}).
An interesting point is that the terms involving Eisenstein series vanish at $s=r-n$ by the trivial zeros of the Riemann zeta function for even integers $r\geq 4$.
\begin{thm}[Ibukiyama and Saito]\label{t14}
If $r\geq 3$, $r$ is odd, and we put ${\displaystyle R=\frac{r-1}{2}}$, then we have
\begin{align*}
\zeta_\Shin(L_r^*,r-n)=&  (-1)^{R+1}2^{-R(2n-2r+1)-r} \times \prod_{j=1}^R|B_{2j}| \\
&\times B_{n-R}\prod_{l=1}^R B_{2n-2r+2l}  \times \frac{1}{R!}  \prod_{t=1}^{R+1} \frac{1}{n-r+t}.
\end{align*}
If $r\geq 4$, $r$ is even, and we put $R'=r/2$, then we have
\begin{align*}
\zeta_\Shin(L_r^*,r-n)=& (-1)^{\left[R'/2\right]+R'(n-r+1)} 2^{-r(n-r+1)-R'}  \times |B_{R'}|\prod_{j=1}^{R'-1}|B_{2j}| \\
& \times \prod_{l=1}^{R'} B_{2n-2r+2l} \times \frac{1}{R'!}  \prod_{t=1}^{R'}\frac{1}{n-r+t}.
\end{align*}
\end{thm}

It is clear that one can obtain an explicit dimension formula combining Corollary \ref{c12} with Theorems \ref{t13} and \ref{t14}.
The following theorem is our explicit dimension formula, that was presented as a conjecture in \cite{IS2,IS3}.
We slightly changed its description.
\begin{thm}\label{t15}
If $k>n+1$ and $N>2$, then we have
\begin{align*}
\dim S_k(\Gamma_n(N))=& [\Gamma_n(1):\Gamma_n(N)]\, \sum_{r=0}^n  N^{\frac{r(r-1)}{2}-rn} \times \prod_{t=1}^{n-r}\prod_{u=t+r}^n(2k-t-u) \\
& \qquad \qquad \qquad \times \prod_{j=1}^{n-r}\frac{(j-1)!\; |B_{2j}|}{(2j-1)!\; j} \times  I(n,r) 
\end{align*}
where
\[
[\Gamma_n(1):\Gamma_n(N)]=N^{n(2n+1)}\prod_{p:\mathrm{prime},\; p|N}\prod_{l=1}^n(1-p^{-2l}),
\]
if $r$ is odd we set 
\[
I(n,r)=(-1)^{[\frac{n+1}{2}]}2^{-2n+R+1}\times \prod_{j=1}^R|B_{2j}| \times |B_{n-R}|\prod_{l=1}^R |B_{2n-2r+2l}| \times \frac{1}{R!}\prod_{m=1}^{R+1}\frac{1}{n-r+m}
\]
($R=\frac{r-1}{2}$), and if $r$ is even we set
\[
I(n,r)=(-1)^{R'(1+\delta_{nr})}2^{-3n+3R'}\times |B_{R'}|\prod_{j=1}^{R'-1}|B_{2j}| \times \prod_{l=1}^{R'}|B_{2n-2r+2l}| \times \frac{1}{R'!}\prod_{m=1}^{R'}\frac{1}{n-r+m}
\]
($R'=\frac{r}{2}$ and $\delta_{nr}$ is the Kronecker symbol).
When $r$ is odd, one sees $I(n,r)=0$ if $n-R$ is odd and $n-R>1$.
When $r$ is even, one also sees $I(n,r)=0$ if $R'$ is odd and $R'>1$.
\end{thm}

For low degrees $n=1$, $2$ and $3$, there are some known results for explicit dimension formulas of $S_k(\Gamma_n(N))$ under the condition $k>n+1$ and $N>2$.
As numerical examples of Theorem \ref{t15}, we shall write the known cases $n=1$, $2$, $3$ and the new cases $n=4$, $5$, $6$ as below.

\vspace{1mm}
\noindent
(1) $n=1$. For $k>2$ and $N>2$, it is well-known that
\[
\dim S_k(\Gamma_1(N))=[\Gamma_1(1):\Gamma_1(N)] \times \left\{ \frac{2k-2}{2^4\cdot 3}-\frac{1}{2^2\cdot N} \right\} .
\]

\vspace{1mm}
\noindent
Numerical examples of $\dim S_k(\Gamma_1(N))$. \\
\begin{tabular}{|c@{\, \vrule width1.5pt \,\,\,}ccccccccccccccccc|} \hline
$N$ \begin{picture}(4,4)(0,0) \put(4,-1){ \line(-3,2){6} } \end{picture} $k$ & 3 & 4 & 5 & 6 & 7 & 8 & 9 & 10 & 11 & 12 & 13 & 14 & 15 & 16 & 17 & 18 & 19    \\ \noalign{\hrule height 1.5pt}
3 & 0 & 1 & 2 & 3 & 4 & 5 & 6 & 7 & 8 & 9 & 10 & 11 & 12 & 13 & 14 & 15 & 16   \\ \hline
4 & 1 & 3 & 5 & 7 & 9 & 11 & 13 & 15 & 17 & 19 & 21 & 23 & 25 & 27 & 29 & 31 & 33  \\ \hline
5 & 4 & 9 & 14 & 19 & 24 & 29 & 34 & 39 & 44 & 49 & 54 & 59 & 64 & 69 & 74 & 79 & 84  \\ \hline
\end{tabular}

\vspace{4mm}
\noindent
(2) $n=2$. Christian \cite{Christian}, Morita \cite{Morita}, Yamazaki \cite{Yamazaki}. For $k>3$ and $N>2$, it is known that
\begin{multline*}
\dim S_k(\Gamma_2(N))=[\Gamma_2(1):\Gamma_2(N)] \\
\times \left\{ \frac{(2k-2)(2k-3)(2k-4)}{2^{10}\cdot 3^3\cdot 5}-\frac{2k-3}{2^6\cdot 3^2\cdot N^2}+\frac{1}{2^5\cdot 3\cdot N^3} \right\}. 
\end{multline*}

\vspace{1mm}
\noindent
Numerical examples of $\dim S_k(\Gamma_2(N))$. \\
\begin{tabular}{|c@{\, \vrule width1.5pt \,\,\,}cccccccc|} \hline
$N$ \begin{picture}(4,4)(0,0) \put(4,-1){ \line(-3,2){6} } \end{picture} $k$ &   4 & 5 & 6 & 7 & 8 & 9 & 10 & 11   \\ \noalign{\hrule height 1.5pt}
3 & 15 & 76 & 200 & 405 & 709 & 1130 & 1686 & 2395  \\ \hline
4 & 360 & 1352 & 3240 & 6280 & 10728 & 16840 & 24872 & 35080  \\ \hline
5 & 5655 & 18980 & 43680 & 83005 & 140205 & 218530 & 321230 & 451555  \\ \hline
\end{tabular}

\vspace{4mm}
\noindent
(3) $n=3$. Tsushima \cite{Tsushima1}. For $k>4$ and $N>2$, it is known that
\begin{multline*}
\dim S_k(\Gamma_3(N))=[\Gamma_3(1):\Gamma_3(N)] \times \\
\left\{ \frac{(2k-2)(2k-3)(2k-4)^2(2k-5)(2k-6)}{2^{16}\cdot 3^6\cdot 5^2\cdot 7}  -\frac{2k-4}{2^{10}\cdot 3^2\cdot 5\cdot N^5}+\frac{1}{2^8\cdot 3^3\cdot N^6}  \right\}.
\end{multline*}

\vspace{1mm}
\noindent
Numerical examples of $\dim S_k(\Gamma_3(N))$. \\
\begin{tabular}{|c@{\, \vrule width1.5pt \,\,\,}cccc|} \hline
$N$ \begin{picture}(4,4)(0,0) \put(4,-1){ \line(-3,2){6} } \end{picture} $k$ &   5 & 6 & 7 & 8   \\ \noalign{\hrule height 1.5pt}
3 & 41132 & 260624 & 1036100 & 3154151  \\ \hline
4 & 14400512 & 87671808 & 345492480 & 1048957952   \\ \hline
5 & 2189096000 & 13202280000 & 51921714000 & 157545444875  \\ \hline
\end{tabular}

\vspace{4mm}
\noindent
(4) $n=4$. For $k>5$ and $N>2$, it follows from Theorem \ref{t15} that
\begin{multline*}
\dim S_k(\Gamma_4(N))=[\Gamma_4(1):\Gamma_4(N)] \times \\
\left\{ \frac{(2k-2)(2k-3)(2k-4)^2(2k-5)^2(2k-6)^2(2k-7)(2k-8)}{2^{25}\cdot 3^8\cdot 5^4\cdot 7^2}  \right. \\
+ \frac{(2k-3)(2k-4)(2k-5)^2(2k-6)(2k-7)}{2^{17}\cdot 3^7\cdot 5^3\cdot 7\cdot N^4} \\ 
\left. - \frac{(2k-4)(2k-5)(2k-6)}{2^{15}\cdot 3^5\cdot 5\cdot 7\cdot N^7} +\frac{1}{2^{12}\cdot 3^4\cdot 5\cdot N^{10}}  \right\}.
\end{multline*}

\vspace{1mm}
\noindent
Numerical examples of $\dim S_k(\Gamma_4(N))$. \\
\begin{tabular}{|c@{\, \vrule width1.5pt \,\,\,}ccc|} \hline
$N$ \begin{picture}(4,4)(0,0) \put(4,-1){ \line(-3,2){6} } \end{picture} $k$ &   6 & 7 & 8   \\ \noalign{\hrule height 1.5pt}
3 & 4579839810 & 59162254866 & 456282921627   \\ \hline
4 & 103260267479040 & 1412646545915904 & 11110964624621568   \\ \hline
5 & 429562396640081250 & 5989030815121331250 & 47380818119506096875 \\ \hline
\end{tabular}

\vspace{4mm}
\noindent
(5) $n=5$. For $k>6$ and $N>2$, it follows from Theorem \ref{t15} that
\begin{multline*}
\dim S_k(\Gamma_5(N))=[\Gamma_5(1):\Gamma_5(N)] \times \\
\left\{ \frac{(2k-2)(2k-3)(2k-4)^2(2k-5)^2(2k-6)^3(2k-7)^2(2k-8)^2(2k-9)(2k-10)}{2^{33}\cdot 3^{12}\cdot 5^5\cdot 7^3\cdot 11}  \right. \\
- \frac{(2k-4)(2k-5)(2k-6)^2(2k-7)(2k-8)}{2^{23}\cdot 3^7\cdot 5^3\cdot 7\cdot N^9} \\ 
\left. - \frac{(2k-5)(2k-6)(2k-7)}{2^{17}\cdot 3^7\cdot 5^2\cdot 7\cdot N^{12}} +\frac{2k-6}{2^{16}\cdot 3^6\cdot 5\cdot 7\cdot N^{14}}  \right\}.
\end{multline*}

\vspace{1mm}
\noindent
Numerical examples of $\dim S_k(\Gamma_5(N))$. \\
\begin{tabular}{|c@{\, \vrule width1.5pt \,\,\,}cc|} \hline
$N$ \begin{picture}(4,4)(0,0) \put(4,-1){ \line(-3,2){6} } \end{picture} $k$ &   7 & 8   \\ \noalign{\hrule height 1.5pt}
3 & 54749238798613788 & 1961103357322399719   \\ \hline
4 & 320755407836707217735680 & 11468658955208332371034112  \\ \hline
5 & 95447256764961220187148437500 & 3412305106826559796929248046875 \\ \hline
\end{tabular}

\vspace{4mm}
\noindent
(6) $n=6$. For $k>7$ and $N>2$, it follows from Theorem \ref{t15} that
\begin{multline*}
\dim S_k(\Gamma_6(N))=[\Gamma_6(1):\Gamma_6(N)] \times \\
\left\{ \frac{691}{2^{43}\cdot 3^{13}\cdot 5^7\cdot 7^5\cdot 11^2\cdot 13} (2k-2)(2k-3)(2k-4)^2(2k-5)^2(2k-6)^3(2k-7)^3  \right. \\
\times (2k-8)^3(2k-9)^2(2k-10)^2(2k-11)(2k-12)  \\
- \frac{(2k-3)(2k-4)(2k-5)^2(2k-6)^2(2k-7)^3(2k-8)^2(2k-9)^2(2k-10)(2k-11)}{2^{31}\cdot 3^{14}\cdot 5^5\cdot 7^4\cdot 11\cdot N^6}   \\
- \frac{(2k-4)(2k-5)(2k-6)^2(2k-7)^2(2k-8)^2(2k-9)(2k-10)}{2^{30}\cdot 3^9\cdot 5^4\cdot 7^2\cdot 11\cdot N^{11}} \\ 
\left. + \frac{(2k-6)(2k-7)(2k-8)}{2^{23}\cdot 3^8\cdot 5^2\cdot 7\cdot N^{18}} -\frac{2k-7}{2^{19}\cdot 3^7\cdot 5^3\cdot 7\cdot N^{20}}  \right\}.
\end{multline*}

\vspace{1mm}
\noindent
Numerical examples of $\dim S_k(\Gamma_6(N))$. \\
\begin{tabular}{|c@{\, \vrule width1.5pt \,\,\,}c|} \hline
$N$ \begin{picture}(4,4)(0,0) \put(4,-1){ \line(-3,2){6} } \end{picture} $k$ &    8   \\ \noalign{\hrule height 1.5pt}
3 & 14338236964403459094697389537 \\ \hline
4 & 62657675456744807193941531065954861056 \\ \hline
5 & 3159011529622615201202592700939984097900390625 \\ \hline
\end{tabular}

\vspace{4mm}
Our dimension formula (Theorem \ref{t37}) includes the case of vector valued Siegel cusp forms (cf. Section \ref{s3.5}).
In the vector valued case, weight factors are defined by irreducible polynomial representations of $\GL(n,\C)$.
As for degree two $(n=2)$, the dimension formula is known (see \cite{Tsushima2,Wakatsuki}).
According to our generalization of Shintani's formula to the vector valued case, the dimension formula involves Gelfand-Tsetlin multiplicities and transposition matrices between Schur functions and zonal polynomials.
In the scalar valued case, such factors become trivial.
They are computable if we fix a weight, but they are not explicit for general weights.

Finally we shall mention some literatures concerning explicit dimension formulas for spaces of scalar and vector valued Siegel modular forms with respect to $\Gamma_n(1)$, $\Gamma_n(2)$, or $k=1$.
For $\Gamma=\Gamma_2(1)$, $\Gamma_2(2)$ ($n=2$), we refer to \cite{Igusa}, \cite{Hashimoto}, \cite{Petersen}, \cite{Tsushima4}, \cite{Tsushima2}, and \cite{Wakatsuki}.
For $\Gamma=\Gamma_3(1)$, $\Gamma_3(2)$ ($n=3$), we refer to \cite{Tsuyumine}, \cite{Runge}, and \cite{Taibi}.
For the weight one ($k=1$), we refer to \cite{Li} and \cite{Gunji}.
Note that we did not touch literatures on numerical computations for dimensions, even though many such works exist.

\subsection{Some remarks on the proof}

We will show that almost all unipotent weighted orbital integrals $J_M^G(u,\fs)$ over $\R$ vanish for proper Levi subgroups $M$ and spherical trace functions $\fs$ of holomorphic discrete series (cf. Proof of Proposition \ref{p33}).
This means that the special values $\zeta_\Shin(\phi_{0,r},r-n)$ in Theorem \ref{t11} are closely related to global coefficients $a^G(S,u)$ of the geometric side of the Arthur trace formula.
There are some works for global coefficients (see \cite{CL,Chaudouard,Chaudouard2,Matz1,Matz2,HW}),  but the works are still not sufficient to calculate all of them explicitly.

In Shintani's formula, the explicit calculations for local orbital integrals over $\R$ are reduced to the Fourier transforms of spherical trace functions along the centers of unipotent radicals of maximal parabolic subgroups.
In this method, we do not calculate each of unipotent orbital integrals.
Explicit calculations for local orbital integrals should be provided by limit formulas (cf. \cite{Bozicevic,Rossmann} and \cite[Appendix]{Arthur4}).
However, the formulas are not enough to compute all unipotent orbital integrals.

When we apply a pseudo-coefficient of a single discrete series to Arthur's invariant trace formula, we can see which unitary representations remain in the spectral side using Hiraga's study \cite{Hiraga}.
In our case, only $\dim S_k(\Gamma)$ appears in the spectral side when weight $k$ is greater than $n+1$.

For each geometric unipotent conjugacy class $C$ of $G$ over $\Q$, the contribution of $C(\Q)$ is expressed by a linear sum of unipotent weighted orbital integrals associated with $C$ (cf. Propositions \ref{cz1} and \ref{t1}).
When $C$ does not correspond to any partition of the form $(2^j,1^{2n-2j})$, we will prove the vanishing of the related unipotent weighted orbital integrals (cf. Proposition \ref{p33}).
There are three main reasons for the vanishing.
First, the character formula for holomorphic discrete series of $\Sp(n,\R)$ is similar to that of irreducible finite dimensional representations of $\GL(n,\C)$ (cf. \cite{Hecht,Martens} and Lemma \ref{nr1}).
Second, the spherical trace function of a holomorphic discrete series has a special vanishing property (cf. Lemma \ref{vl1}).
The third reason is explained by the Jacobson-Morozov theory.
For each unipotent element $u$, we have a grading $\fg=\oplus_{m\in\Z}\fg_m$ on $\fg=\mathrm{Lie}(G)$.
Then, we have $\fg_m=0$ for all $m>2$ if and only if $u$ corresponds to a partition of the form $(2^j,1^{2n-2j})$ (cf. Lemma \ref{lu3}).
The vanishing of $J_G^G(u,\fs)$ follows from an integration over $\exp(\oplus_{m>2}\fg_m(\R))$.

Our method can be applied to other holomorphic cusp forms for neat arithmetic subgroups.
If one applies the method to the group $R_{F/\Q}(\SL(2))$ ($F$ is a totally real field over $\Q$), then one can get a dimension formula for spaces of Hilbert cusp forms.
Such a dimension formula was already obtained by Shimizu in \cite{Shimizu}.
It is also possible to get known dimension formulas for $\Q$-rank one groups.
For $\SU(1,n)$ (Picard cusp forms), we refer to \cite{Kato} and \cite{Kojima}.
For non-split $\Q$-forms of $\Sp(2)$, we refer to \cite{Arakawa} and \cite{Wakatsuki}.
For higher rank groups $\mathrm{Spin}(2,n)$, $\SU(n,n)$ and so on, it should be possible to use our method to calculate dimensions of spaces of their holomorphic cusp forms.

\section{Preliminaries}\label{s22}

\subsection{Notation}\label{notation}

Let $\A$ denote the adele ring of the rational number field $\Q$ and $\A_\fin$ the finite adele ring of $\Q$.
Then, we have $\A=\R \times \A_\fin$ where $\R$ denotes the real number field.
Let $\C$ denote the complex number field.
For each place $v$ of $\Q$, we denote by $\Q_v$ the completion of $\Q$ at $v$.
We write $\inf$ for the infinite place of $\Q$, i.e., $\Q_\inf=\R$.

Let $G$ be a connected reductive linear algebraic group over a field $F$ of characteristic $0$.
For a Levi subgroup $M$ of $G$, we denote by $\cL(M)$ (resp.\ $\cF(M)$) the set of Levi subgroups (resp. parabolic subgroups) of $G$ over $F$ that contain $M$.
For each $P\in\cF(M)$, a Levi decomposition of $P$ over $F$ is given by $P=M_PN_P$ where $M_P\in\cL(M)$ and $N_P$ denotes the unipotent radical of $P$.
Set $\cP(M)=\{P\in\cF(M) \, | \, M_P=M\}$.
We write $A_M$ for the split part of the center of $M$ over $F$.
Let $X(M)_F$ denote the additive group of rational characters of $M$ over $F$.
Then, we have the finite dimensional real vector spaces $\fa_M=\mathrm{Hom}_\Z(X(M)_F,\R)$ and $\fa_M^*=X(M)_F\otimes_\Z \R$.

From now on, we set $F=\Q$ and consider a connected reductive linear algebraic group $G$ over $\Q$.
Fix a minimal Levi subgroup $M_0$ of $G$ over $\Q$.
Set $\cL=\cL(M_0)$, $\cF=\cF(M_0)$, and $\cP=\cP(M_0)$.
We also fix a suitable maximal compact subgroup $\bK=\prod_v \bK_v$ of $G(\A)$, i.e., $\bK$ is admissible relative to $M_0$  (cf. \cite[Section 1]{Arthur2}).
In particular,  the maximal compact group $\bK_v$ of $G(\Q_v)$ is admissible relative to $M_0$ for each $v$, and the Iwasawa decompositions $G(\A)=P(\A)\bK$ and $G(\Q_v)=P(\Q_v)\bK_v$ hold for any $P\in\cF$ and any place $v$ of $\Q$.
For each $M\in \cL$, a mapping $H_M:M(\A)\to \fa_M$ is defined by
\[
\langle H_M(x),\chi\rangle=\log|\chi(x)| \qquad (x\in M(\A), \;\; \chi\in X(M)_\Q).
\]
For each $P\in\cF$, we set
\[
H_P(nmk)=H_{M_P}(m)  \qquad (n\in N(\A), \;\; m\in M_P(\A), \;\; k\in\bK).
\]
A subgroup $M(\A)^1$ of $M(\A)$ is defined by
\[
M(\A)^1=\{m\in M(\A)\mid H_M(m)=0\}
\]
and let $A_M(\R)^0$ denote the identity component of $A_M(\R)$.
Then we have the isomorphisms $A_M(\R)^0\cong\fa_M$ and $M(\A)\cong A_M(\R)^0\times M(\A)^1$ via the mapping $H_M$.

For each $M\in\cL$, a bijection $\fa_M^*\to\fa_{A_M}^*$ is defined by the restriction $X(M)_\Q\to X(A_M)_\Q$.
We assume that Levi subgroups $M_1$ and $M_2$ in $\cL$ satisfy $M_1\subset M_2$.
It obviously follows that $A_{M_2}\subset A_{M_1}\subset M_1 \subset  M_2$ over $\Q$.
Since the restriction $X(M_2)_F\to X(M_1)_F$ is injective, we obtain a linear injection $\fa_{M_2}^*\to\fa_{M_1}^*$ and a linear surjection $\fa_{M_1}\to \fa_{M_2}$.
Since the restriction $X(A_{M_1})_\Q\to X(A_{M_2})_\Q$ is surjective, we have a linear surjection $\fa_{M_1}^*\to \fa_{M_2}^*$ and a linear injection $\fa_{M_2}\to \fa_{M_1}$.
Set 
\[
\fa_{M_1}^{M_2}=\{ a_1 \in \fa_{M_1} \mid \langle a_1,a^*_2\rangle=0 \quad ( \forall a^*_2 \in \fa_{M_2}^*) \},
\]
\[
(\fa_{M_1}^{M_2})^*= \{ a_1^* \in \fa_{M_1}^* \mid \langle a_2,a^*_1\rangle=0 \quad ( \forall a_2 \in \fa_{M_2}) \}  .
\]
Then, we have
\[
\fa_{M_1}=\fa_{M_2}\oplus \fa_{M_1}^{M_2}  \quad \text{and} \quad \fa_{M_1}^*=\fa_{M_2}^*\oplus (\fa_{M_1}^{M_2})^*.
\]
For each $P\in\cF$, we set $\fa_P^*=\fa_{M_P}^*$, $\fa_{P_1}^{P_2}=\fa_{M_{P_1}}^{M_{P_2}}$, $(\fa_{P_1}^{P_2})^*=(\fa_{M_{P_1}}^{M_{P_2}})^*$, and $A_P=A_{M_P}$.
We denote by $\Phi_P\subset X(A_P)_\Q$ the set of roots of $A_P$ in the Lie algebra $\mathfrak n_P$ of $N_P$.

Fix a minimal parabolic subgroup $P_0$ in $\cP$.
Set $\fa_0=\fa_{M_0}$, $\fa_0^*=\fa_{M_0}^*$, $\fa_0^P=\fa_0^{M_P}=\fa_{M_0}^{M_P}$, $(\fa_0^P)^*=(\fa_0^{M_P})^*=(\fa_{M_0}^{M_P})^*$, and $\Phi_0=\Phi_{P_0}$.
The set $\Phi_0\sqcup (-\Phi_0)$ is a root system in $(\fa_0^G)^*$ and $\Phi_0$ is a system of positive roots.
Let $W_0^G$ denote the Weyl group of the root system $\Phi_0\sqcup (-\Phi_0)$ in $(\fa_0^G)^*$.
We set $\Delta_0=\Delta_{P_0}$, i.e., $\Delta_0$ is the set of simple roots attached to $\Phi_0$.
Let $\widehat{\Delta}_0$ denote the set of simple weights corresponding to $\Delta_0$.
Let $P\in \cF$ and $P\supset P_0$.
A subset $\Delta_0^P$ of $\Delta_0$ is defined by $\fa_P=\{  a\in\fa_0 \; | \; \langle a ,\alpha \rangle=0,\; \, \forall \alpha\in\Delta_0^P  \}$.
We set
\[
\Delta_P=\{\alpha|_{\fa_P}\in (\fa_P^G)^* \mid \alpha\in \Delta_0-\Delta_0^P \}, \quad \widehat{\Delta}_P=\{ \varpi_{\alpha} \in (\fa_P^G)^* \; | \; \alpha\in\Delta_0-\Delta_0^P  \},
\]
where $\varpi_{\alpha}$ denotes the simple weight corresponding to $\alpha$.

\subsection{Haar measures}\label{s221}

Let $G$ be a connected reductive linear algebraic group over $\Q$.
Let $C_c^\inf(G(\R))$ denote the space of compactly supported smooth functions on $G(\R)$ and let $C_c^\inf(G(\A_\fin))$ denote the space of compactly supported locally constant functions on $G(\A_\fin)$.
Set $C_c^\inf(G(\A))=C_c^\inf(G(\R))\otimes C_c^\inf(G(\A_\fin))$.
For each finite place $v$ of $\Q$, we write $C_c^\inf(G(\Q_v))$ for the space of compactly supported locally constant functions on $G(\Q_v)$.
For a finite set $S$ of places of $\Q$, we set $\Q_S=\prod_{v\in S}\Q_v$, $\bK_S=\prod_{v\in S}\bK_v$, and $C_c^\inf(G(\Q_S))=\otimes_{v\in S} C_c^\inf(G(\Q_v))$.

Let $P\in\cF$.
An element $\rho_P$ in $\fa_P^*$ is defined by $\rho_P=\frac{1}{2}\sum_{\alpha\in\Phi_P} (\dim \fn_{\alpha}) \alpha$, where $\fn_{\alpha}=\{X\in\fn_P\mid \mathrm{Ad}(a)X=\alpha(a) X \;\; (a\in A_P)\}$ $(\mathrm{Ad}(a)X=aXa^{-1})$.
The function $\delta_P(x)=e^{2\rho_P(H_P(x))}$ $(x\in P(\A))$ is called the modular character of $P(\A)$.
Fix a Haar measure $\d g$ on $G(\A)$ and a Haar measure $\d H$ on $\fa_P$.
We denote by $\d n$ the Haar measure on $N_P(\A)$ normalized by $\int_{N_P(\Q)\bsl N_P(\A)}\d n=1$.
Let $\d k$ denote the Haar measure on $\bK$ normalized by $\int_\bK \d k=1$ and let $\d a$ denote the Haar measure on $A_P(\R)^0$ induced from $\d H$.
Then, there exists a unique Haar measure $\d^1 m$ on $M_P(\A)^1$ such that
\[
\int_{G(\A)}f(g)\d g=\int_{N_P(\A)}\int_{M_P(\A)^1}\int_{A_P(\R)^0}\int_{\bK} f(nmak)e^{-2\rho_P(H_P(a))} \, \d k\, \d^1 m \,  \d a \, \d n
\]
for any $f\in C_c^\inf(G(\A))$.
A Haar measure $\d m$ on $M_P(\A)$ is determined by $\d m=\d^1m \, \d a$.
For each place $v$ of $\Q$, let $\d g_v$ denote a Haar measure on $G(\Q_v)$ and let $\d n_v$, $\d m_v$, and $\d k_v$ denote Haar measures on $N_P(\Q_v)$, $M_P(\Q_v)$ and $\bK_v$ respectively.
We may assume that $\d g=\prod_v\d g_v$, $\d k=\prod_v \d k_v$, $\d n=\prod_v \d n_v$, and $\d m=\prod_v \d m_v$.
Furthermore, we may also assume $\int_{\bK_v}\d k_v=1$ for any $v$, $\int_{N_P(\Q_v)\cap \bK_v}\d n_v=1$ for any $v<\inf$, and
\[
 \int_{G(\Q_v)}f_v(g_v)\d g_v=\int_{N_P(\Q_v)}\int_{M_P(\Q_v)}\int_{\bK_v}f_v(n_v m_v k_v)e^{-2\rho_P(H_P(m_v))} \, \d k_v \, \d m_v  \, \d n_v 
\]
for any $v$ and any $f_v\in C_c^\inf(G(\Q_v))$.
We have the Haar measures $\d n_S=\prod_{v\in S}\d n_v$, $\d m_S=\prod_{v\in S}\d m_v$, and $\d k_S=\prod_{v\in S}\d k_v$ on $N(\Q_S)$, $M(\Q_S)$, and $\bK_S$ respectively.
The above Haar measures were chosen as in \cite[Section 1]{Arthur8}.
For a test function $f$ on $G(\Q_S)$, we set
\[
f_P(m)=\delta_P(m)^{\frac{1}{2}}\int_{\bK_S}\int_{N_P(\Q_S)}f(k^{-1}mnk)\, \d n  \d k \qquad (m\in M_P(\Q_S)).
\]

\section{Multiplicities and truncated zeta integrals}\label{s2}

Throughout this section, we assume that $G$ is a connected semisimple linear algebraic group over $\Q$.
Hence, we may suppose that $G$ is a closed subgroup of $\GL(N)$ over $\Q$ for a fixed natural number $N$.
For each element in $G(\R)$, its eigenvalues are defined in the usual way by the inclusion $G(\R) \subset \GL(N,\C)$.
Eigenvalues of elements in $G(\Q)$ are also defined by regarding $G(\Q)$ as a subgroup of $G(\R)$.
Moreover, we assume that $G$ is simply connected, $G'(\R)$ is not compact for each $\Q$-simple factor $G'$ of $G$, and $G(\R)$ admits discrete series representations (i.e., square-integrable unitary representations).
Especially, $G$ has the strong approximation property with respect to $\inf$ (cf. \cite[Theorem 7.12]{PR}), that is, $G(\Q)$ is dense in $G(\A_\fin)$ via the diagonal embedding.

We fix an open compact subgroup $K_0$ of $G(\A_\fin)$.
An arithmetic subgroup $\Gamma$ of $G(\Q)$ is defined by
\[
\Gamma=G(\Q)\cap (G(\R)K_0).
\]
We write $L^2_{\mathrm{disc}}(\Gamma \bsl G(\R))$ for the discrete spectrum of $L^2(\Gamma\bsl G(\R))$ and $L^2_{\mathrm{cont}}(\Gamma \bsl G(\R))$ for the continuous spectrum of $L^2(\Gamma\bsl G(\R))$.
Then, the orthogonal direct sum
\[
L^2(\Gamma \bsl G(\R))=L^2_{\mathrm{disc}}(\Gamma \bsl G(\R))\oplus L^2_{\mathrm{cont}}(\Gamma \bsl G(\R))
\]
holds.
The discrete spectrum $L^2_{\mathrm{disc}}(\Gamma \bsl G(\R))$ decomposes into an orthogonal direct sum
\[
L^2_{\mathrm{disc}}(\Gamma \bsl G(\R)) \cong \bigoplus_{\pi \in \Pi_{\mathrm{unit}}(G(\R))} m(\pi,\Gamma) \cdot H_\pi
\]
where $\Pi_{\mathrm{unit}}(G(\R))$ is the unitary dual of $G(\R)$, $H_\pi$ is the Hilbert space of $\sigma$, and $m(\pi,\Gamma)$ is a non-negative integer.
The number $m(\pi,\Gamma)$ is called the multiplicity of $\pi$ in $L^2_{\mathrm{disc}}(\Gamma \bsl G(\R))$. 
For a discrete series $\sigma$ of $G(\R)$, we will relate $m(\sigma,\Gamma)$ to unipotent weighted orbital integrals and truncated zeta integrals of spherical trace functions of $\sigma$ under some conditions (cf. Propositions \ref{cz1} \and \ref{t1}).

\subsection{Arthur's invariant trace formula}\label{s23}

There exists a finite set $S_0$ of finite places of $\Q$ such that
\[
K_0=K_{S_0}\prod_{v\not\in S_0 ,\; v<\inf}\bK_v
\]
for a certain open compact subgroup $K_{S_0}$ in $G(\Q_{S_0})$.
We fix such a finite set $S_0$ and set
\[
S_1=S_0\sqcup\{\inf\}.
\]
Since $G$ has the strong approximation property with respect to $\inf$, it follows that
\[
G(\Q)\bsl G(\A)/K_0\cong \Gamma \bsl G(\R).
\]
Hence we have a $G(\R)$-isomorphism
\begin{equation}
L^2(G(\Q)\bsl G(\A)/K_0)\cong L^2(\Gamma\bsl G(\R)).
\end{equation}
We will assume the following condition for $\Gamma$.
\begin{cond}\label{c21}
Let $\gamma$ be any element of $\Gamma$.
If there exists a natural number $l$ such that $\gamma^l$ is unipotent, then $\gamma$ is unipotent.
\end{cond}
Let $\gamma$ be an element of $G(\Q)$.
Supposing that we have $a=1$ if $a$ is a root of $1$ for any eigenvalue $a$ of $\gamma$, the element $\gamma$ is said to be neat.
If all elements of $\Gamma$ are neat, then $\Gamma$ is said to be neat in the sense of \cite[Section 17]{Borel1} and $\Gamma$ clearly satisfies Condition \ref{c21}.
It is known that a principal congruence subgroup of $\GL(N,\Z)$ is neat if its level is greater than two (see, e.g.,  \cite[p.118]{Borel1}, \cite[Lemma 2]{Morita}).

Let $S$ be a finite set of places of $\Q$.
For each $L\in \cL$, we write  $\cH(L(\Q_S))$ for the Hecke algebra on $L(\Q_S)$ which consists of $\bK_S\cap L(\Q_S)$-finite functions in $C_c^\inf(L(\Q_S))$.
Here, a function $f$ on $L(\Q_S)$ is called $\bK_S\cap L(\Q_S)$-finite if the space of functions on $L(\Q_S)$ spanned by left and right $\bK_S\cap L(\Q_S)$-translates of $f$ is finite dimensional.
Note that, if $S$ does not include $\inf$, then we have $\cH(G(\Q_S))=C_c^\inf(G(\Q_S))$ as vector space over $\C$.
We also note that any function in $C_c^\inf(G(\A_\fin))$ is $\bK_\fin$-finite, where we set $\bK_\fin=\prod_{v<\inf}\bK_v$.

We write $h_0$ for the characteristic function of $K_0$.
Set
\[
h=\vol(K_0)^{-1} h_0 . 
\]
The function $h$ belongs to $C_c^\inf(G(\A_\fin))$.
For any finite set $S'(\supset S_0)$ of finite places of $\Q$, there exists a function $h_{S'}$ in $\cH(G(\Q_{S'}))$ such that $h$ is the product of it with the characteristic function of $\prod_{v\not\in S'\sqcup\{\inf\}}\bK_v$.
We will identify $h$ with $h_{S'}$ in the geometric side of trace formula.

Fix a discrete series representation $\sigma$ of $G(\R)$.
Let $\tfs$ denote a pseudo-coefficient of $\sigma$ (cf. \cite{CD}).
The function $\tfs$ is in $\cH(G(\R))$ and satisfies $\Tr(\sigma(\tfs))=1$.
\begin{cond}\label{c2}
For any $\pi\in \Pi_{\mathrm{unit}}(G(\R))$, we have $\Tr(\pi(\tfs))=0$ unless $\pi\cong\sigma$.
\end{cond}
One can see the Harish-Chandra parameters such that $\sigma$ satisfies Condition \ref{c2} by the result of \cite{Hiraga} (Condition \ref{c2} holds if the set $\mathrm{Wall}(\sigma)$ is empty in the paper \cite{Hiraga}).

Let $H$ denote a connected reductive linear algebraic group over $\Q$.
For $\gamma\in H(\Q)$, we write $H_{\gamma,+}$ for the centralizer of $\gamma$ in $H$, and $H_\gamma$ for the connected component of $1$ in $H_{\gamma,+}$.
Let $\cU_H$ denote the Zariski closure in $H$ of the set of unipotent elements in $H(\Q)$.
Note that $\cU_H$ is a closed algebraic subvariety of $H$ over $\Q$.
We also see that $\cU_H(\Q)$ is the set of unipotent elements of $H(\Q)$.

Two elements of $G(\Q)$ are called $\cO$-equivalent if their semisimple parts are $G(\Q)$-conjugate.
We write $\cO$ for the set of $\cO$-equivalence classes in $G(\Q)$.
The set $\fo_\unip=\cU_G(\Q)$ clearly belongs to $\cO$.

For each $M\in\cL$, $L\in\cL(M)$, and $\gamma\in M(\Q_S)$, the local distributions $J_M^L(\gamma)$ and $I_M^L(\gamma)$ on $\cH(L(\Q_S))$ are defined in \cite{Arthur3} and \cite{Arthur8} (we set $F=\Q$ in his papers).
Especially $J_M^M(\gamma)=I_M^M(\gamma)$ means the ordinary orbital integral on $M(\Q_S)$.
We omit detailed explanations for them, but we will later explain only unipotent weighted orbital integrals $J_M^L(u)$ $(u\in\cU_M(\R))$ in Section \ref{secu1}.

\begin{prop}\label{p1}
We choose a parabolic subgroup $Q\in\cP(M)$ for each $M\in\cL$.
If we assume that $\Gamma$ satisfies Condition \ref{c21} and $\sigma$ satisfies Condition \ref{c2}, then we have
\begin{equation}\label{p1e}
m(\sigma,\Gamma)=\sum_{M\in\cL}\frac{|W^M_0|}{|W^G_0|}\sum_{u\in(\cU_M(\Q))_{M,S_1}} a^M(S_1,u) \, I_M^G(u,\tfs) \, J_M^M(u,h_Q),
\end{equation}
where $(\cU_M(\Q))_{M,S_1}$ denotes the finite set of $M(\Q_{S_1})$-conjugacy classes in $\cU_M(\Q)\subset \cU_M(\Q_{S_1})$.
\end{prop}
\begin{proof}

First, we apply the function $\tfs h$ in $\cH(G(\Q_{S_1}))$ to Arthur's invariant trace formula.
The test function $\tfs$ is cuspidal and satisfies Condition \ref{c2}.
Hence, using the arguments in \cite[Section 3]{Arthur4}, we have
\[
m(\sigma,\Gamma)=\sum_{\fo}I_\fo(\tfs h),
\]
\[
I_\fo(\tfs h)=\sum_{M\in\cL}\frac{|W^M_0|}{|W^G_0|} \sum_{\varsigma}\sum_{u} a^M(S,\varsigma u)\, I_M^G(\varsigma u,\tfs) \, J_M^M(\varsigma u,h_Q)
\]
for any sufficiently large $S\supset S_1$, where $\fo$ moves over $\cO$-equivalence classes in $G(\Q)$, $\varsigma$ runs over all $M(\Q)$-conjugacy classes of semisimple elements in $\fo\cap M(\Q)$ which are $\Q$-elliptic in $M$, and $u$ runs over all $M_{\varsigma,+}(\Q)M_\varsigma(\Q_S)$-conjugacy classes in $\cU_{M_\varsigma}(\Q)$.
By \cite[Section 3]{Arthur10} the terms $I_\fo(\tfs h)$ vanish except for finitely many $\cO$-equivalence classes $\fo$.
Hence, one may assume that $\fo$ moves over a finite subset of $\cO$ in the above sum.
In particular, the subset is independent of choice of $S$.
It also follows that $\varsigma$ ranges over a finite set independent of $S$ (cf. \cite[Section 2]{Arthur9}).
Note that the range for $u$ depends on $S$.
For each conjugacy class $\varsigma$, we choose a representative element and it is denoted by the same notation $\varsigma$.
By \cite[Proposition 3.1]{Arthur10} we can choose $S=S_1$ for $I_{\fo_\unip}(\tfs h)$.
Hence, $I_{\fo_\unip}(\tfs h)$ equals the right hand side of \eqref{p1e}.
Therefore, it is sufficient to prove $I_\fo(\tfs h)=0$ unless $\fo=\fo_\unip$.

We say that $\varsigma$ is $\R$-elliptic in $M$ if there exists a maximal torus $T$ in $M$ over $\R$ such that $T(\R)/A_M(\R)$ is compact and $\varsigma$ belongs to $T(\R)$.
Here, we note that $A_M$ is defined over $\Q$.

\begin{lem}\label{p1l}
If $\varsigma$ is not $\R$-elliptic in $M$, then $I_M^G(\varsigma u,\tfs)$ vanishes.
\end{lem}
\begin{proof}
This lemma follows from the arguments in \cite[p.277--278]{Arthur4}.
Here we give a summary for it.
We write $M_\inf$ for the Levi subgroup defined by $M_\inf=M$ over $\R$.
Then, $\fa_{M_\inf}$ is obtained from $M_\inf$ over $\R$.
When $\fa_M\neq \fa_{M_\inf}$, $\varsigma$ is always non-$\R$-elliptic in $M$ and we easily get $I_M(\varsigma u)=0$ by the descent property \cite[Corollary 8.2]{Arthur8} for $\cL(M_\inf)$ and the cuspidality of $\tfs$.
Hence, we may assume $\fa_M=\fa_{M_\inf}$ and $\varsigma$ is not $\R$-elliptic in $M$.
Set $\gamma=\varsigma u$.
The Jordan semisimple part of $\gamma$ is $\varsigma$.
If $M_\gamma\neq G_\gamma$, then
\[
I_M^G(\gamma,\tfs)=\lim_{a\to 1}\sum_{L\in\cL(M)}r_M^L(\gamma,a)\, I_L^G(a\gamma,\tfs)
\]
by \cite[(2.2)]{Arthur8}, where $a$ ranges over small generic points on $A_M(\R)$.
Applying the descent property \cite[Corollary 8.3]{Arthur8} to $I^G_L(a\gamma,\tfs)$, we have
\[
I_L^G(a\gamma,\tfs)=\sum_{L_1\in\cL(M)}d_{M}^G(L,L_1) \, \widehat{I}_M^{L_1}(a\gamma,(\tfs)_{L_1}).
\]
Since $\tfs$ is cuspidal, we obtain
\[
I_M^G(\gamma,\tfs)=\lim_{a\to 1}  I_M^G(a\gamma,\tfs).
\]
From this we may assume $M_\gamma=G_\gamma$.
Now $M$ has a proper Levi subgroup $M'(\supset M_\gamma)$ defined over $\R$, because $\fa_M=\fa_{M_\inf}$ and $\varsigma$ is not $\R$-elliptic in $M$.
Using the descent property \cite[Corollary 8.2]{Arthur8}, we obtain
\[
I_M^G(\gamma,\tfs)=\sum_{L_2\in\cL(M')}d_{M'}^G(M,L_2) \, \widehat{I}_{M'}^{L_2}(\gamma,(\tfs)_{L_2}) = d_{M'}^G(M,G) \, \widehat{I}_{M'}^{G}(\gamma,(\tfs)_G)=0.
\]
In the above, we note that $\cL(M')$ and $\widehat{I}_{M'}^{L_2}(\gamma)$ are defined over $\R$.
Therefore, the proof is completed.
\end{proof}

Fix an $\cO$-equivalence class $\fo\; (\neq \fo_\unip)$.
To prove $I_\fo(\tfs h)=0$, we will relate Condition \ref{c21} to $J_M^M(\varsigma u,h_Q)$ using the strong approximation for $G$.
Now, $\varsigma$ ranges a finite set of $\fo\cap M(\Q)$ (independent of $S$) in the above equality for $I_\fo(\tfs h)$.
Hence, there exists a large enough $S$ such that any such element $\varsigma$ belongs to $\bK_v$ for all $v\not\in S$.
By Lemma \ref{p1l}, we may also assume that $\varsigma$ is $\R$-elliptic in $M$.
From now, we fix such a subset $S$.

For each $M_{\varsigma,+}(\Q)M_\varsigma(\Q_S)$-conjugacy classes in $\cU_{M_\varsigma}(\Q)$, we can choose a representative element $u$ such that $u$ belongs to $\bK_v$ for any $v\not\in S$.
This can be proved as follows.
There exists a parabolic subgroup $R$ of $M$ over $\Q$ such that $\varsigma$ belongs to $M_R(\Q)$ and $u$ belongs to $N_R(\Q)$ (cf. \cite[Proof of Lemma 5.5]{HW}).
Considering the action of $A_R(\Q)$ on $N_R(\Q)$ and the root system for $(M,A_R)$, one can find a suitable element $\chi\in\mathrm{Hom}(\GL(1),A_R)_\Q$ and a large enough natural number $m$ such that $\Ad(\chi(m))\, u$ satisfies the required condition.

Assume that $J_M^M(\varsigma u,h_Q)$ does not vanish.
Note that we can not extend $J_M^M(\varsigma u,h_Q)$ to the orbital integral over $M(\A_\fin)$ for non-trivial unipotent elements $u$ in general.
Set
\[
\gamma=\varsigma u \; (\in M(\Q))\quad \text{and} \quad  S'=S-\{\inf\}.
\]
It follows from $J_M^M(\gamma,h_Q)\neq 0$ that there exist $m_1\in M(\Q_{S'})$, $n_1'\in N_Q(\Q_{S'})$, and $k\in\bK_{S'}$ such that $h_0(k_1^{-1}m_1^{-1}\gamma m_1n_1' k_1)=1$.
If we set $n_1=m_1n_1'm_1^{-1}$, then $n_1$ is in $N_Q(\Q_{S'})$ and we have $h_0(k_1^{-1}m_1^{-1}\gamma n_1 m_1 k_1)=1$.
Since $h_0$ is smooth, there exist a neighborhood $U_1$ at $n_1$ in $N_Q(\Q_{S'})$ and a neighborhood $U_2$ at $m_1k_1$ in $G(\Q_{S'})$ such that
\[
x^{-1}\gamma ux\in K_{S_0}\prod_{v\in S'-S_0}\bK_v \quad (u\in U_1,\;\; x\in U_2).
\]
Therefore, by the strong approximation and the above mentioned conditions for $\varsigma$ and $u$ $(\gamma=\varsigma u)$, there exist
\[
n\in N_Q(\Q)\cap\big( U_1\prod_{v\not\in S}N_Q(\Q_v)\cap\bK_v\big)\quad\text{and}\quad \delta\in G(\Q)\cap \big(U_2\prod_{v\not\in S}\bK_v\big)
\]
such that
\[
\delta^{-1}\gamma n\delta \in K_0.
\]
Namely, it follows that $\delta^{-1}\gamma n\delta$ belongs to $\Gamma$.
Set
\[
\Gamma'=\delta\Gamma\delta^{-1}.
\]
Then $\Gamma'$ satisfies Condition \ref{c21} and $\gamma n$ belongs to $\Gamma'$.
Note that the elements $\gamma n$, $\gamma$, and $\varsigma$ have the same eigenvalues, because $Q(\Q)\cap\fo=(M(\Q)\cap\fo)N_Q(\Q)$ holds (cf. \cite[p.923]{Arthur12}).

We have a decomposition $\varsigma=\varsigma_1\varsigma_2$ such that the eigenvalues of $\varsigma_1$ (resp. $\varsigma_2$) belong to $\R_{>0}$ (resp. $\C^1$) and $\varsigma_1$, $\varsigma_2$ are semisimple (see, e.g., \cite[Section 2]{Borel3}).
Since $\varsigma$ is $\R$-elliptic in $M$, it follows that $\varsigma_1$ belongs to $A_M(\R)^0$.

Since $\gamma n\in \delta K_0\delta^{-1}\cap Q(\A_\fin)$ and $\gamma n\in\Gamma'\cap Q(\Q)$, we get $H_Q(\gamma n)=0$ as an element $\gamma n$ in $\Gamma'\subset G(\R)$.
It is clear that $H_Q(\varsigma_2)=0$ over $Q(\R)$.
Hence, we have $H_Q(\varsigma_1)=H_Q(\varsigma)=H_Q(\gamma)=H_Q(\gamma n)=0$ over $Q(\R)$.
Thus, we obtain $\varsigma_1=1$, because $\varsigma_1$ is in $A_M(\R)^0$.
So, the eigenvalues of $\gamma n$ are in $\C^1$.

Let $\xi=\gamma n \in \Gamma'$ and $\xi=\xi_s\xi_u$ where $\xi_s$ (resp. $\xi_u$) denotes the semisimple (resp. unipotent) Jordan component of $\xi$.
There exists a parabolic subgroup $P_1$ over $\Q$ such that $\xi_s$ (resp. $\xi_u$) belongs to $M_{P_1}(\Q)$ (resp. $N_{P_1}(\Q)$).
Since the projection $\eta:P_1\to P_1/N_{P_1}\cong M_{P_1}$ is a homomorphism over $\Q$, the image $\eta(\Gamma'\cap P_1(\Q))$ is arithmetic in $M_{P_1}(\Q)$ (see \cite[Chapter 4]{PR}).
There exists a compact torus $T_1(\R)$ in $M_{P_1}(\R)$ such that $\xi_s\in T_1(\R)$ because the eigenvalues of $\xi_s$ are in $\C^1$.
Hence, using $\eta(\xi_s)=\eta(\xi)$, one finds that $\eta(\xi_s)$ is in $\eta(\Gamma_1)\cap \eta(T_1(\R))$.
Thus, there exists a natural number $m_1$ such that $\eta(\xi_s^{m_1})=1$.
This means $\xi_s^{m_1}=1$.
By Condition \ref{c21} we get $\xi_s=1$ and this contradicts $\fo\neq \fo_\unip$.
Hence, we obtain $J_M^M(\varsigma u,h_Q)=0$ and we have proved this proposition.
\end{proof}

\subsection{Integrable discrete series and spherical trace functions}\label{s25}

From now on, we assume that a discrete series representation $\sigma$ is integrable, that is, all $\bK_\inf$-finite matrix coefficients of $\sigma$ are integrable over $G(\R)$.
In \cite[Theorem]{HS}, one can see a necessary and sufficient condition of the Harish-Chandra parameter for the integrability of $\sigma$.
Let $H_\sigma$ denote a representation space of $\sigma$ and let $(\tau,H_{\sigma}(\tau))$ denote the minimal $\bK_\inf$-type of $(\sigma,H_\sigma)$.
We write $\psi_{\sigma,\tau}$ for the spherical trace function of $\sigma$ with $\tau$, i.e.,
\[
\psi_{\sigma,\tau}(x)=\Tr(\mathrm{pr}_\tau\circ\sigma(x)\circ \mathrm{pr}_\tau) \quad (x\in G(\R))
\]
where $\mathrm{pr}_\tau$ is the projection $H_\sigma\to H_\sigma(\tau)$ (cf. \cite[Chapter 6]{Warner}).
Set
\[
\fs(x)= d_\tau^{-1} d_\sigma \, \overline{\psi_{\sigma,\tau}(x)} \quad (x\in G(\R))
\]
where $d_\sigma$ is the formal degree of $\sigma$ and we set $d_\tau=\dim H_\sigma(\tau)$.
By this definition, it is obvious that $\fs$ is $\bK_\inf$-finite, $\fs$ satisfies $\Tr(\sigma(\fs))=1$, and
\[
\fs(k^{-1}xk)=\fs(x) \quad (x\in G(\R) , \; \; k\in\bK_\inf).
\]
It follows from the results of \cite{Hiraga} and \cite{HS} that the integrable discrete series $\sigma$ always satisfies Condition \ref{c2}.

Let $p$ be a real number and suppose $0<p\leq 2$.
We write $\cC^p(G(\R))$ for the Schwartz space as defined in \cite[p.236]{Hoffmann}.
The space $\cC^p(G(\R))$ is the same as the space $\cC_\gamma(G(\R))$ $(\gamma=\frac{2}{p}-1)$ defined in \cite[p.67]{Milicic}.
For the case $p=2$, we set $\cC(G(\R))=\cC^2(G(\R))$ and the space $\cC(G(\R))$ is just the Harish-Chandra Schwartz space on $G(\R)$.
Note that $\cC^p(G(\R))$ is included in $L^p(G(\R))$ (see \cite[p.242]{Hoffmann}) and $C_c^\inf(G(\R))\subset \cC^{p_1}(G(\R))\subset \cC^{p_2}(G(\R))$ $(0<p_1<p_2)$.
Especially, the inclusion mapping of $C_c^\inf(G(\R))$ onto $\cC^p(G(\R))$ is continuous and its image is dense in $\cC^p(G(\R))$ (cf. \cite[Theorem 1 (iii)]{Milicic}).
It is well-known that $\fs$ belongs to $\cC(G(\R))$ (cf. \cite[p.450]{Knapp}).
If $\sigma$ is integrable, then $\fs$ belongs to $\cC^1(G(\R))$.
For any $p$, we can get a sufficient condition of the Harish-Chandra parameter such that $\fs$ belongs to $\cC^p(G(\R))$ by \cite[Theorem]{Milicic} (see also \cite{TV} and \cite{HS}).
Note that the subspace of $\bK_\inf$-finite vectors in $H_\sigma$ is stable for $\sigma(\fg)$ (cf. \cite[Proposition 8.5]{Knapp}).

The following lemma was stated in the paper \cite{Langlands2}.
\begin{lem}[Langlands]\label{l1}
If $(\pi,H_\pi)$ is a unitary representation of $G(\R)$ which does not contain the integrable discrete series $\sigma$, then $\pi(\fs)=0$ holds.
In particular, we have $\Tr(\pi(\fs))=0$.
\end{lem}
\begin{proof}
We explain a proof of this lemma using \cite[Lemma 2.2]{Takase}, because the proof was not written in \cite{Langlands2}.
It is sufficient to show $\pi(\overline{\psi_{\sigma,\tau}})u=0$ for any $u\in H_\pi$.
Let $H_\pi(\sigma)$ denote the $\sigma$-isotypic component of $H_\pi$ and let $H_\pi(\sigma,\tau)$ denote the $\tau$-isotypic component of $H_\pi(\sigma)$.
The result \cite[Lemma 2.2]{Takase} states that
\[
H_\pi(\sigma,\tau)=\{ v \in H_\pi \mid \pi(\overline{\psi_{\sigma,\tau}})v=d_\sigma^{-1}v\}.
\]
Note that we are assuming $H_\pi(\sigma,\tau)=0$.
For $\varphi_1$, $\varphi_2\in L^1(G(\R))$, a convolution $\varphi_1*\varphi_2$ is defined by
\[
\varphi_1*\varphi_2(x)=\int_{G(\R)}\varphi_1(xy^{-1}) \, \varphi_2(y)\, \d y.
\]
It is obvious that $\overline{\psi_{\sigma,\tau}}*\overline{\psi_{\sigma,\tau}}=d_\sigma^{-1}\overline{\psi_{\sigma,\tau}}$.
Hence, for any $u\in H_\pi$, the vector $\pi(\overline{\psi_{\sigma,\tau}})u$ belongs to $H_\pi(\sigma,\tau)$ by
\[
\pi(\overline{\psi_{\sigma,\tau}})\{ \pi(\overline{\psi_{\sigma,\tau}})u \}=\pi(\overline{\psi_{\sigma,\tau}}*\overline{\psi_{\sigma,\tau}})u=d_\sigma^{-1}\pi(\overline{\psi_{\sigma,\tau}})u.
\]
Thus, $\pi(\overline{\psi_{\sigma,\tau}})u=0$ follows.
\end{proof}

\begin{lem}\label{l0}
There exists a sequence $\{f_i\}_{i\in\N}$ such that $f_i\in\cH(G(\R))$ $(\forall i\in\N)$, $|f_i(x)|\leq |\fs(x)|$ $(\forall x\in G(\R))$, and $\{f_i\}_{i\in\N}$ converges to $\fs$ in the topology of $\cC(G(\R))$.
\end{lem}
\begin{proof}
We prove this lemma by using an argument similar to \cite[Proof of Proposition 12.16 (a)]{Knapp}.
Let $t\in\R$ and $\chi_t$ be the characteristic function of $\{x\in G(\R)\mid \|x\|\leq t\}$, where $\|x\|$ is the same notation as in \cite[p.188]{Knapp}.
Fix a function $h\in C_c^\inf(G(\R))$ such that $\int_{G(\R)}h(g)\, \d g=1$ and $h(g)\geq 0$ $(g\in G(\R))$.
Set
\[
\tilde h(g)=\int_{\bK_\inf\times\bK_\inf}h(k_1gk_2)\, \d k_1 \, \d k_2\in C_c^\inf(G(\R)).
\]
Furthermore, we set
\[
h_t(x)=\tilde h*\chi_t*\tilde h(x)=\int_{G(\R)\times G(\R)}\tilde h(g_1)\chi_t(g_1^{-1}xg_2^{-1})\tilde h(g_2)\, \d g_1 \, \d g_2.
\]
It is clear that $h_t(k_1gk_2)=h_t(g)$ $(k_1,k_2\in \bK_\inf$, $g\in G(\R))$ and $h_t$ is in $C_c^\inf(G(\R))$.
Since $|\tilde h(g_1)\chi_t(g_1^{-1}x g_2^{-1})\tilde h(g_2)|\leq |\tilde h(g_1)\tilde h(g_2)|$, we have $0\leq h_t(x)\leq 1$ $(x\in G(\R))$.
Set $f_i=h_i\fs$ for each $i\in\N$.
It follows from \cite[Proof of Proposition 12.16 (a)]{Knapp} that $\{f_i\}_{i\in\N}$ converges to $\fs$ in the topology of $\cC(G(\R))$.
Thus, the sequence $\{f_i\}_{i\in\N}$ satisfies the above mentioned conditions.
\end{proof}

For each unipotent orbit $u$ in $M(\R)$, a unipotent weighted orbital integral $J_M^L(u)$ over $\R$ is defined by the integral \eqref{euw} in Section \ref{secu1}.
For every $f\in C_c^\inf(G(\R))$, the integral \eqref{euw} is compatible with the usual definition \cite[(3*) in p.224]{Arthur3}.
We will assume the following condition.
\begin{cond}\label{cuni}
Let $M$, $L\in\cL$, $M\subset L$, and $Q\in\cP(L)$.
For any unipotent orbit $u$ in $M(\R)$, the integral $J_M^L(u,(\fs)_Q)$ absolutely converges with respect to the integral \eqref{euw}.
\end{cond}
By Proposition \ref{lac}, there exists a small $p>0$ such that Condition \ref{cuni} holds if $\fs$ belongs to $\cC^p(G(\R))$.

\begin{prop}\label{p2}
Assume that $\sigma$ is integrable and Condition \ref{cuni} holds.
For each unipotent orbit $u$ in $M(\R)$, we have
\[
I_M^G(u,\tfs)=J_M^G(u,\fs).
\]
\end{prop}
\begin{proof}
Let $\Pi_{\mathrm{temp}}(G(\R))$ denote the set of equivalence classes of irreducible tempered unitary representations of $G(\R)$.
We write $\cI(G(\R))$ for the space defined in \cite[Section 11]{Arthur11}, which consists of certain functions on $\Pi_{\mathrm{temp}}(G(\R))$.
For each $\C$-valued function $f$ on $G(\R)$, a $\C$-valued function $f_G$ on $\Pi_{\mathrm{temp}}(G(\R))$ is formally defined by $f_G(\pi)=\Tr(\pi(f))$ $(\pi\in\Pi_{\mathrm{temp}}(G(\R)))$.
It is known that $f_G$ belongs to $\cI(G(\R))$ if $f\in\cH(G(\R))$.
It was proved in \cite{Arthur10} that $I_M^G(u)$ is supported on characters, i.e., $I_M^G(u,f)=0$ if $f_G=0$ $(f\in\cH(G(\R)))$ (cf. \cite[Theorem 5.1]{Arthur10}).
Hence, $\widehat{I}^G_M(u)$ is a continuous linear form  on $\cI(G(\R))$ and we have the Fourier transform
\[
I_M^G(u,\tfs)=\widehat{I}_M^G(u,\tilde f_{\sigma,G}).
\]
By Lemma \ref{l1} $f_{\sigma,G}$ is in $\cI(G(\R))$ and we get
\[
\widehat{I}_M^G(u,\tilde f_{\sigma,G})=\widehat{I}_M^G(u,f_{\sigma,G}).
\]

Let $\cI_1(G(\R))$ denote the space defined in \cite[p.173]{Arthur6}, which includes $\cI(G(\R))$.
A mapping from $\cC(G(\R))$ to $\cI_1(G(\R))$ is defined by $f\mapsto f_G$.
In particular, the mapping is continuous and surjective by \cite[Theorem]{Arthur6}.
Note that the inclusion mapping from $\cI(G(\R))$ into $\cI_1(G(\R))$ is also continuous.
Take a sequence $\{ f_i \}_{i\in\N}\subset \cH(G(\R))$ as in Lemma \ref{l0}.
Since $f_{i,G}$ and $f_{\sigma,G}$ are in $\cI(G(\R))$, the sequence $\{ f_{i,G} \}_{i\in \N}$ converges to $f_{\sigma,G}$ in the topology of $\cI(G(\R))$.
Hence,
\[
\widehat{I}_M^G(u,f_{\sigma,G})=\lim_{i\to\inf}\widehat{I}_M^G(u,f_{i,G})=\lim_{i\to\inf}I_M^G(u,f_i).
\]
Now, $I_M^G(u,f)$ is defined by
\[
I_M^G(u,f)=J_M^G(u,f)-\sum_{L\in\cL(M),\, L\neq G }\widehat{I}^L_M(u,\phi_L(f)),
\]
where $\phi_L$ is a mapping from $\cC(G(\R))$ to $\cI_1(L(\R))$ (see \cite[Section 7]{Arthur2} or \cite[p.178--179]{Arthur7} for its definition).
Since $\phi_L$ is continuous on $\cC(G(\R))$ (cf. \cite[Corollary 9.2]{Arthur2} and \cite[p.179]{Arthur7}) and we have $\phi_L(\fs)=0$ for any $L\neq G$ by Lemma \ref{l1} and its definition, we find
\[
\lim_{i\to\inf}\phi_L(f_i)=\phi_L(\fs)=0 .
\]
Using the facts that $\widehat{I}^L_M(u)$ is a continuous linear form on $\cI(L(\R))$, $\phi_L(f_i)$ belongs to $\cI(L(\R))$ (cf. \cite[Proof of Theorem 12.1]{Arthur11}), and $\{\phi_L(f_i)\}_{i\in \N}$ converges into $0$ in $\cI(L(\R))$, we get
\[
\lim_{i\to\inf} \widehat{I}^L_M(u,\phi_L(f_i))= \widehat{I}^L_M(u,\lim_{i\to\inf} \phi_L(f_i))= \widehat{I}^L_M(u,0)=0.
\]
Therefore, by Condition \ref{cuni} and Lebesgue's convergence theorem, we have
\[
\widehat{I}_M^G(u,f_{\sigma,G})=\lim_{i\to\inf} J_M^G(u,f_i)=J_M^G(u,\fs).
\]
Thus, the proof is completed.
\end{proof}

\begin{lem}\label{l2}
Assume that Condition \ref{cuni} holds.
For each $u\in (\cU_M(\Q))_{M,S}$ we have
\[
J_M^G(u,\fs)J^M_M(u,h_Q)=J_M^G(u,\fs h).
\]
\end{lem}
\begin{proof}
By \cite[(18.7) and a comment in p.109]{Arthur5} (cf. \cite[Proofs of Theorem 8.1 and Proposition 9.1]{Arthur8}), we know the splitting formula
\[
J_M^G(u,\fs h)=\sum_{L_1,L_2\in\cL(M)}d_M^G(L_1,L_2)\, J_M^{L_1}(u,(\fs)_{Q_1})\, J_M^{L_2}(u,h_{Q_2})
\]
where $Q_j\in\cP(L_j)$ $(j=1,2)$ and the correspondence $(L_1,L_2)\longrightarrow (Q_1,Q_2)$ is explained in \cite[p.101]{Arthur5}.
Note that $J_M^G(u,\fs h)$ is convergent by this equality and Condition \ref{cuni}.
Since $\fs$ is a cusp form in $\cC(G(\R))$, i.e.,
\[
\int_{N_P(\R)}\fs(x^{-1}ny)\, \d n=0 \qquad (x,y\in G(\R))
\]
for any proper parabolic subgroup $P$ of $G$ over $\R$ (cf. \cite[p.233]{Wallach1}), this lemma follows from the splitting formula.
\end{proof}

\begin{prop}\label{p21}
Let $\sigma$ be an integrable discrete series of $G(\R)$.
If Conditions \ref{c21} and \ref{cuni} hold, then we have
\[
m(\sigma,\Gamma)=\sum_{M\in\cL}\frac{|W^M_0|}{|W^G_0|}\sum_{u\in (\cU_M(\Q))_{M,S} } a^M(S,u) \, J_M^G(u,\fs h).
\]
\end{prop}
\begin{proof}
This lemma follows from Propositions \ref{p1} and \ref{p2} and Lemma \ref{l2}.
\end{proof}

\subsection{Truncated zeta integrals}\label{s26}

Fix a minimal parabolic subgroup $P_0$ in $\cP$.
For each $P\in\cF$ such that $P\supset P_0$, let $\widehat\tau_P$ denote the characteristic function of the subset $\{H\in\fa_P \mid \varpi(H)>0 \;\; (\forall \varpi\in\widehat\Delta_P)\}$ of $\fa_P$ and let $\tau_P$ denote the characteristic function of the subset $\{ H\in\fa_P \mid \alpha(H)>0 \; \; (\forall \alpha\in\Delta_P)\}$ of $\fa_P$.
Let $\mathfrak{S}_G$ denote a Siegel set in $G(\A)$ with respect to $G(\Q)\bsl G(\A)$ and $P_0$, i.e., there exist a compact set $\omega$ in $P_0(\A)^1$ and a point $T_1$ in $\fa_0$ such that 
\[
\mathfrak{S}_G=\{pak\mid k\in\bK,\;\; p\in\omega, \;\; a\in A_{P_0}(\R)^0,\;\; \tau_{P_0}(H_{P_0}(a)-T_1)=1 \}.
\]
For a point $T$ in $\fa_0$, we denote by $F^G(g,T)$ the characteristic function of the projection of the set $\{ g\in \mathfrak{S}_G \mid \widehat\tau_{P_0}(T-H_0(g))=1 \}$ to $G(\Q)\bsl G(\A)$.

Let $C$ denote a geometric unipotent conjugacy class in $\cU_G$ containing a $\Q$-rational point.
Fix an element $u\in C(\Q)$.
Let $\fg$ denote the Lie algebra of $G$ over $\Q$ and set $X=\log(u)\in\fg(\Q)$.
There exist a semisimple element $H$ in $\fg(\Q)$ and a nilpotent element $Y$ in $\fg(\Q)$ such that $\{H,X,Y\}$ is a standard triple, i.e., $[H,X]=2X$, $[H,Y]=-2Y$, $[X,Y]=H$ (cf. \cite[Chapter VIII, Section 11]{Bourbaki}).
We set
\[
\fg_j=\{Z\in\fg \mid [H,Z]=jZ\},\quad \fu=\bigoplus_{j>0}\fg_j ,\quad \fu_{>k}=\bigoplus_{j>k}\fg_j.
\]
Let $L$ denote the centralizer of $H$ in $G$ and we put $U=\exp(\fu)$ and $U_{>k}=\exp(\fu_{>k})$.
Then $Q=LU$ is a parabolic subgroup of $G$ defined over $\Q$ and it is called the canonical parabolic subgroup of $u$.
In particular, $L$ is a Levi subgroup of $Q$ and $U$ is the unipotent radical of $Q$.
Let $\Ad(g)X=gXg^{-1}$ $(g\in G$, $X\in \fg)$.
Now the triple $(L,\Ad,\fg_2)$ becomes a regular prehomogeneous vector space over $\Q$ (cf. \cite{Gyoja,Hoffmann1}).
We write $\fg_2^\reg$ for the set of regular points in $\fg_2$.
For each test function $f$ on $G(\A)$, a truncated zeta integral $Z_C^T(f)$ is defined by
\[
Z_C^T(f)=\int_{Q(\Q)\bsl Q(\A)}F^G(q,T) \sum_{\mu\in\fg^\reg_2(\Q)}\sum_{\nu\in U_{>2}(\Q)} f_\bK(q^{-1}\exp(\mu)\nu q)\, \d q
\]
where $T$ is in $\fa_0$, $\d q$ is the left Haar measure on $Q(\A)$ normalized by the same manner as in Section \ref{s221} and we set
\[
f_\bK(x) = \int_{\bK} f(k^{-1}x k)  \, \d k.
\]
The integral $Z_C^T(f)$ is an analogue of the zeta integrals of prehomogeneous vector spaces.

Any element $u$ in $\cU_M(\Q)$ is contained in a unique geometric unipotent conjugacy class $C_u$ in $M$.
For a fixed geometric unipotent conjugacy class $C'$ in $\cU_M$, we have a bijection from the set of $u$ in $(\cU_M(\Q))_{M,S}$ with $C_u=C'$ onto the set of $M(\Q_S)$-orbits in $C'(\Q_S)$ (see \cite[Lemma 7.1 and p.1268]{Arthur1}).
For each $u$ in $\cU_M(\Q)$, we denote by $C_u^G$ the induced unipotent conjugacy class of $G$ associated with $C_u$ and $G$ (cf. \cite{LS,Hoffmann2}).
\begin{prop}\label{cz1}
Assume that $\sigma$ is integrable.
Let $C$ be a geometric unipotent conjugacy class in $\cU_G$ containing a $\Q$-rational point.
Then, the zeta integral $Z_C^T(\fs h)$ is absolutely convergent.
If we also assume that Condition \ref{cuni} holds, then we have
\begin{equation}\label{eq31}
\lim_{T\to\inf} Z_C^T(\fs h)=\sum_{M\in\cL}\frac{|W^M_0|}{|W^G_0|}\sum_{u\in (\cU_M(\Q))_{M,S}, \; \; C_u^G=C } a^M(S,u) \, J_M^G(u,\fs h) 
\end{equation}
where $T\to\inf$ means that $\min_{\alpha\in\Delta_0}\alpha(T)\to\inf$.
\end{prop}
\begin{proof}
For each geometric unipotent conjugacy class $C$ of $G$, we write $J_C^T(f)$ for the distribution on $C_c^\inf(G(\A))$ uniquely determined by \cite[Theorem 4.2]{Arthur1}, which is a polynomial in $T$.
By the result of \cite{FL} (see also \cite{Chaudouard2}), the distribution $J_C^T(f)$ can be defined as the integral
\[
J_C^T(f)=\int_{G(\Q)\bsl G(\A)} \sum_{P\supset P_0} (-1)^{\dim\fa_P}\sum_{\delta \in P(\Q)\bsl G(\Q)}K_{P,C}(\delta x,f)\,\widehat\tau_P(H_P(\delta x)-T)\, \d x
\]
where
\[
K_{P,C}(x,f)=\sum_{u\in\cU_{M_P}(\Q),\;\; C_u^G=C}\int_{N_P(\A)}f(x^{-1}u nx)\, \d n.
\]
We take a sequence $\{f_i\}_{i\in\N}\subset\cH(G(\R))$ as in Lemma \ref{l0}.
Let $T_0$ denote the point in $\fa_{M_0}$ defined in \cite[Lemma 1.1]{Arthur2}.
By \cite[Corollary 8.4]{Arthur1} and Condition \ref{cuni}, we have
\[
\text{(RHS of \eqref{eq31})}=\lim_{i\to \inf}J_C^{T_0}(f_ih).
\]
Furthermore, it follows from \cite[Theorem 7.1]{FL} that
\[
\lim_{i\to \inf}J_C^{T_0}(f_ih)=J_C^{T_0}(\fs h)
\]
because $\sigma$ is integrable.
Since $\fs$ is a cusp form in $\cC(G(\R))$, if $P\neq G$ we get
\[
\int_{N_P(\A)}(\fs h)(x^{-1}u n x)\, \d n=0.
\]
Hence, we obtain
\[
J_C^T(\fs h)=\int_{G(\Q)\bsl G(\A)}K_{G,C}(x,\fs h)\, \d x .
\]
This means that $J^{T_0}_C(\fs h)=J^T_C(\fs h)$ holds for any $T$ and we have
\[
J_C^{T_0}(\fs h)=\lim_{T\to\inf}J^T_C(\fs h).
\]
For a test function $f$ on $G(\A)$, we set
\[
\tilde J_C^T(f)=\int_{G(\Q)\bsl G(\A)} F^G(x,T)\sum_{u\in C(\Q)}f(x^{-1}ux)\; \d x.
\]
By \cite[Theorem 7.1]{FL} one finds
\[
\lim_{T\to\inf}J^T_C(\fs h)=\lim_{T\to\inf}\tilde J_C^T(\fs h) .
\]
Thus, we have
\begin{equation}\label{eq32}
\text{(RHS of \eqref{eq31})}=\lim_{T\to\inf}\tilde J_C^T(\fs h).
\end{equation}

Since the support of $F^G(x,T)$ on a Siegel set is compact, it is not difficult to show
\begin{equation}\label{2tl1}
 \int_{G(\Q)\bsl G(\A)}F^G(x,T)  \sum_{u\in\cU_G(\Q)} |(\fs h)(x^{-1}ux)| \d x <\inf
\end{equation}
(see, e.g., \cite[Theorem 3.1]{FL}).
By \eqref{2tl1} we can change the infinite sum and the integration of $\tilde J_C^T(\fs h)$.
Using \cite[Theorem 2]{Hoffmann2} or \cite[Theorem 5]{Hoffmann1} one can easily find
\begin{equation}\label{ceq}
C(\Q)=\bigsqcup_{\delta\in Q(\Q)\bsl G(\Q)}\delta^{-1}\, \exp(\fg_2^\reg(\Q))\, U_{>2}(\Q) \, \delta.
\end{equation}
Therefore, by \eqref{2tl1} and \eqref{ceq}, we get the absolute convergence of $Z_C^T(\fs h)$.
We also obtain the equality
\[
\tilde J_C^T(\fs h)= Z_C^T(\fs h).
\]
Thus, \eqref{eq31} is derived from \eqref{eq32} and this equality.
\end{proof}

\begin{prop}\label{t1}
Assume that $\sigma$ is integrable and Conditions \ref{c21} and \ref{cuni} hold.
Then we have
\[
m(\sigma,\Gamma)=\sum_{C} \lim_{T\to\inf}Z_C^T(\fs h) = \vol(\Gamma\bsl G(\R))\, d_\sigma+\sum_{C\neq 1}\lim_{T\to\inf} Z_C^T(\fs h)
\]
where $C$ moves over all geometric unipotent conjugacy classes in $\cU_G$ containing a $\Q$-rational point and $T\to\inf$ means that $\min_{\alpha\in\Delta_0}\alpha(T)\to\inf$.
\end{prop}
\begin{proof}
This proposition follows from Propositions \ref{p21} and \ref{cz1}.
\end{proof}
We call $\lim_{T\to\inf}Z_C^T(\fs h)$ the contribution of $C(\Q)$ to $m(\sigma,\Gamma)$.

\section{Convergence of unipotent weighted orbital integrals}\label{secu}

Throughout this section, we set $F=\Q$ or $\R$, and $G$ denotes a connected reductive linear algebraic group over $F$.
We fix a minimal Levi subgroup $M_0$ over $F$ and a maximal compact subgroup $\bK_\inf$ of $G(\R)$ which is admissible relative to $M_0$ (cf. \cite[Section 1]{Arthur2}).
We will define unipotent weighted orbital integrals $J_M^G(u,f)$ over $G(\R)$ by the integral \eqref{euw}.
Furthermore, we give a sufficient condition for functions $f$ on $G(\R)$ such that $J_M^G(u,f)$ absolutely converges.
This is necessary to use the results of Section \ref{s2} (see Condition \ref{cuni}).

\subsection{General definition}\label{secu1}

Fix a standard Levi subgroup $M$ of $G$ over $F$.
According to Section \ref{notation}, the notations $\cP(M)$, $A_M$, and $\fa_M$ are defined over $F$. 
We choose a unipotent element $u$ in $M(\R)$.
Let $\fm$ denote the Lie algebra of $M$ over $\R$.
A nilpotent element $X$ in $\fm(\R)$ is defined by $X=\log(u)$.
By the Jacobson-Morozov theorem, there exist a semisimple element $H$ in $\fm(\R)$ and a nilpotent element $Y$ in $\fm(\R)$ such that $[H,X]=2X$, $[H,Y]=-2Y$, $[X,Y]=H$ (cf. \cite[Chapter VIII, Section 11]{Bourbaki}).
Set
\[
\fm_j=\{Z\in\fm\mid [H,Z]=jZ\},\quad \fu=\bigoplus_{j>0}\fm_j,\quad \fu_{> k}=\bigoplus_{j> k}\fm_j.
\]
The inclusion $[\fm_j,\fm_k]\subset \fm_{j+k}$ clearly holds.
Hence, $\fu$ and $\fu_{k>0}$ are Lie subalgebras in $\fm$.
Let $L$ denote the centralizer of $H$ in $G$.
Then, $\fm_0$ is the Lie algebra of $L$.
Let $B(\, , \, )$ denote a non-degenerate $G(\R)$-invariant bilinear form on $\fm(\R)$ which is the Killing form on the derived algebra of $\fm(\R)$.
Choose a basis $\{Z_l\}_{l=1}^d$ of $\fm_1(\R)$ and a basis $\{Z_m'\}_{m=1}^d$ of $\fm_{-1}(\R)$ such that $B(Z_l,Z_m')=\delta_{lm}$, where $d=\dim \fm_1(\R)=\dim \fm_{-1}(\R)$.
We define polynomials $c_{lm}(x)$ $(x\in\fm_2(\R))$ by $[x,Z_l']=\sum c_{ml}(x)Z_m$.
Furthermore, a polynomial $\varphi$ on $\fm_2(\R)$ is defined by $\varphi(x)=\det(c_{lm}(x))$ $(x\in\fm_2(\R))$.
Especially, we have $\varphi(X)\neq 0$.
By \cite[Lemma 2]{Rao} the equality $\varphi(\mathrm{Ad}(g)x)=\det(\mathrm{Ad}(g)|_{\fm_1})^2 \, \varphi(x)$ holds for any $g$ in $L(\R)$.

A unipotent subgroup $U_{>k}$ over $\R$ is defined by $U_{>k}=\exp(\fu_{>k})$.
We choose Haar measures $\d X_k$ on $\fm_k$.
Let $\d y$ denote a Haar measure on $U_{>2}(\R)$ induced from the Haar measure $\prod_{k>2}\d X_k$ on $\fu_{>2}(\R)$ via the exponential map.
The adjoint orbit $O_u(\R)=\mathrm{Ad}(L(\R))X$ in $\fm_2(\R)$ is an open subset.
If we set $\d x=\d X_2$, then $|\varphi(x)|^{1/2}\d x$ is a measure on $O_u(\R)$.
Choose a parabolic subgroup $P_1$ in $\cP(M)$ and set
\[
\mu=\exp(x)y \qquad (x\in O_u(\R), \;\; y\in U_{>2}(\R)).
\]
By the equality
\begin{equation}\label{relation}
n^{-1}a\mu n=a\mu\nu \qquad (n,\nu\in N_{P_1}(\R), \;\; a\in A_M(\R)),
\end{equation}
$n$ is regarded as a function of $\mu$, $\nu$, and $a$.
For each $P\in\cP(M)$, we put
\[
v_P(\lambda,g)=e^{-\lambda(H_P(g))}\quad (\lambda\in\fa_{M,\C}^*,\;\; g\in G(\R)),
\]
\[
w_P(\lambda,a,\mu\nu)=v_P(\lambda,n)\,\prod_\beta r_\beta(\lambda,u,a)
\]
where $\beta$ runs over all the reduced roots of $(N_P/(N_P\cap N_{P_1}),A_M)$ and we set
\[
r_\beta(\lambda,u,a)=|(a^\beta-1)(1-a^{-\beta})|^{\lambda(\beta_u^\vee)/2}
\]
($\beta_u^\vee\in\fa_M$ is a coroot $\beta^\vee$ multiplied by a non-negative real number such that  the limits $w_P(\lambda,1,\mu\nu)$ exist and do not vanish for generic $\lambda$ and $\nu$, cf. \cite[p.238]{Arthur3}).
Then, a weight factor $w_M(1,\mu\nu)$ is defined by
\begin{equation}\label{ewf}
w_M(1,\mu\nu)=\lim_{\lambda\to 0}\sum_{P\in\cP(M)}\frac{w_P(\lambda,1,\mu\nu)}{\theta_P(\lambda)}
\end{equation}
where $\theta_P(\lambda)=\vol(\fa_M/\Z(\Delta_P^\vee))^{-1}\prod_{\alpha\in\Delta_P}\lambda(\alpha^\vee)$ (the set $\Delta_P^\vee=\{\alpha^\vee\mid \alpha\in\Delta_P\}$ denotes the basis of $\fa_P^G$ dual to $\widehat\Delta_P$).
For a test function $f$ on $G(\R)$, a unipotent weighted orbital integral $J_M^G(u,f)$ is defined by
\begin{multline}\label{euw}
J_M^G(u,f)=\int_{\bK_\inf}\d k \, \int_{N_{P_1}(\R)}\d \nu \, \int_{U_{>2}(\R)}\d y \, \int_{O_u(\R)}\d x \\ f(k^{-1}\exp(x)y\nu  k)\, |\varphi(x)|^{1/2}\, w_M(1, \exp(x)y\nu ) .
\end{multline}
This definition falls into line with \cite[Theorem 1 and Lemma 2]{Rao}, \cite[Sections 5 and 6]{Arthur3} and \cite[(2.3) in Section 2.4]{HW}.
If $f$ belongs to $C_c^\inf(G(\R))$, then the integral $J_M^G(u,f)$ of \eqref{euw} is compatible with the usual definition \cite[(3*) in p.224]{Arthur3} (also see \cite[(18.12)]{Arthur5}) by the arguments in \cite[Sections 5 and 6]{Arthur3}.
It is known that $J_M^G(u,f)$ absolutely converges for any $f\in C_c^\inf(G(\R))$ (cf. \cite{Arthur3}).

\subsection{Estimates for polynomials and integrations}

We will study some estimates given in \cite[Section 7]{Arthur3} for polynomials and integrations.
In particular, we give a refinement of \cite[Lemma 7.1]{Arthur3} over $\R$ to prove the convergence of $J_M^G(u)$ for non-compactly supported functions.
The argument is almost the same as in \cite[Proof of Lemma 7.1]{Arthur3}.

For $n\in\Z_{>0}$ and $d\in \Z_{\geq 0}$, we set
\[
\mathbf{F}(n,d)=\{p\in\C[x_1,x_2,\dots,x_n] \mid \deg_{x_j}p\leq d \;\; (1\leq \forall j\leq n) \}
\]
where $\deg_{x_j}p$ means the degree of $p$ with respect to $x_j$.
For each $p \in \mathbf{F}(n,d)$, we denote by $\|p\|$ the maximum of the absolute values of the coefficients of $p$.
For a positive real number $\delta$, we set
\[
\mathbf{F}(n,d,\delta)=\{ p\in \mathbf{F}(n,d) \mid   \|p\|\geq \delta \}.
\]
For each $n$-tuple $k=(k_1,k_2,\dots,k_n)\in(\Z_{>0})^{\oplus n}$, a closed subset $\mathbf{U}_k$ in $\R^{\oplus n}$ is defined by
\[
\mathbf{U}_k=U_{k_1}\times U_{k_2} \times \cdots \times U_{k_n}
\]
where $U_j$ denotes the closed subset $[-\frac{1}{2}j,-\frac{1}{2}(j-1)] \sqcup [\frac{1}{2}(j-1),\frac{1}{2}j]$ in $\R$.
For $k\in(\Z_{>0})^{\oplus n}$, $p\in\mathbf{F}(n,d)$, and $\varepsilon\in\R_{>0}$, a domain $\mathbf{U}_k(p,\varepsilon)$ in $\R^{\oplus n}$ is defined by
\[
\mathbf{U}_k(p,\varepsilon)=\{ x\in \mathbf{U}_k \mid |p(x)|<\varepsilon\}.
\]

\begin{lem}\label{lesti1}
Let $n=1$ and fix a non-negative integer $d$.
There exist $m_d\in\Z_{>0}$, $C_{d,1}$, $C_{d,2},\dots,C_{d,m_d}>0$, and $0<t_{d,1},t_{d,2},\dots,t_{d,m_d}\leq 1$  such that, for any $\varepsilon>0$, any $\delta>0$, any $p\in \mathbf{F}(1,d,\delta)$, and any $k_1\in \Z_{>0}$, we have the inequality
\[
\vol(\mathbf{U}_{k_1}(p,\varepsilon))<k_1^d \sum_{l=1}^{m_d} C_{d,l}\, (\delta^{-1}\varepsilon)^{t_{d,l}} .
\]
\end{lem}
\begin{proof}
We will prove this lemma using an induction with respect to $d$.

First, we consider the case $d=0$.
Since $p$ is in $\mathbf{F}(1,0,\delta)$, $p(x)$ is now a constant greater than or equal to $\delta$.
Hence, it is obvious that
\[
\vol(\mathbf{U}_{k_1}(p,\varepsilon))\leq \begin{cases} 1 & \text{if $\varepsilon\geq\delta$,} \\ 0 & \text{otherwise,} \end{cases}.
\]
This shows that $\vol(\mathbf{U}_{k_1}(p,\varepsilon))\leq \delta^{-1}\varepsilon$.
Hence, it is sufficient to choose $C_{0,1}=2$ and $t_{0,1}=1$ $(m_0=1)$.

Next we treat the case $d>0$.
There exists a constant $r$ in $\C$ such that $p(x)=(x-r)q(x)$ where $\deg q<d$.
Since $\|p_1p_2\|\leq (\deg p_1+1)\|p_1\|\|p_2\|$ $(p_1,p_2\in \C[x_1])$, we have
\[
\|q\|\geq \min(\frac{\delta}{2},\; \frac{\delta}{2|r|} ).
\]
This shows that $q$ belongs to $\mathbf{F}(1,d-1,\min(\frac{\delta}{2},\; \frac{\delta}{2|r|}))$.

\noindent
(i) We consider the case $|r|>\frac{k_1}{2}+1$.
Then, we have $q\in\mathbf{F}(1,d-1,\frac{\delta}{2|r|})$.
The set $\mathbf{U}_{k_1}(q,\frac{\varepsilon}{|r|-\frac{k_1}{2}})$ includes $\mathbf{U}_{k_1}(p,\varepsilon)$, because
\[
\Big(|r|-\frac{k_1}{2}\Big)|q(x)|\leq |x-r||q(x)|=|p(x)| \qquad (x\in \mathbf{U}_{k_1}).
\]
By the assumption of the induction, there exist constant $C_{d-1,1},\dots,C_{d-1,m_{d-1}}>0$ and $0<t_{d-1,1},\dots,t_{d-1,m_{d-1}}\leq 1$ such that
\[
\vol(\mathbf{U}_{k_1}(q,\epsilon)) < k_1^{d-1} \sum_{l=1}^{m_{d-1}}  C_{d-1,l}\, (2|r|\delta^{-1}\epsilon)^{t_{d-1,l}} \qquad (\epsilon>0).
\]
Therefore, using $|r|>\frac{k_1}{2}+1$ we have
\begin{multline*}
\vol(\mathbf{U}_{k_1}(p,\varepsilon))\leq \vol(\mathbf{U}_{k_1}(q,\frac{\varepsilon}{|r|-\frac{k_1}{2}})) < k_1^{d-1} \sum_{l=1}^{m_{d-1}}  C_{d-1,l}\, \Big\{ \frac{|r|}{|r|-\frac{k_1}{2}}\times 2\delta^{-1}\varepsilon \Big\}^{t_{d-1,l}} \\
< k_1^{d-1} \sum_{l=1}^{m_{d-1}}  C_{d-1,l}\, \Big\{ (k_1+2)\delta^{-1}\varepsilon \Big\}^{t_{d-1,l}} < k_1^{d-1} \sum_{l=1}^{m_{d-1}}  C_{d-1,l}\, (k_1+2) (\delta^{-1}\varepsilon )^{t_{d-1,l}} \\
< k_1^d \sum_{l=1}^{m_{d-1}}  (3C_{d-1,l})\, (\delta^{-1}\varepsilon )^{t_{d-1,l}}.
\end{multline*}
Note that the inequality $\frac{x}{x-y}<y+1$ holds if $x-1>y>0$.

\noindent
(ii) We shall consider the case $|r|\leq \frac{k_1}{2}+1$.
It is obvious that
\[
\vol(\mathbf{U}_{k_1}(x_1-r,(\delta^{-1}\varepsilon)^{1/2}))< 2(\delta^{-1}\varepsilon)^{1/2} .
\]
Since $q\in \mathbf{F}(1,d-1,\min(\frac{\delta}{2},\; \frac{\delta}{2|r|}))$, by the assumption of the induction one gets
\[
\vol(\mathbf{U}_{k_1}(q,(\delta\varepsilon)^{1/2}))< k_1^{d-1}\sum_{l=1}^{m_{d-1}} C_{d-1,l}\, \, \left\{ \max(2,2|r|)\delta^{-1}(\delta\varepsilon)^{1/2} \right\}^{t_{d-1,l}}.
\]
For $x\in\mathbf{U}_{k_1}(p,\varepsilon)$, $|x-r|<(\delta^{-1}\varepsilon)^{1/2}$ or $|q(x)|<(\delta\varepsilon)^{1/2}$ holds.
Hence, it follows from the above two inequalities that
\begin{multline*}
\vol(\mathbf{U}_{k_1}(p,\varepsilon))\leq \vol(\mathbf{U}_{k_1}(x_1-r,(\delta^{-1}\varepsilon)^{1/2}))+ \vol(\mathbf{U}_{k_1}(q,(\delta\varepsilon)^{1/2}))\\
< 2(\delta^{-1}\varepsilon)^{1/2}+ k_1^{d-1}\sum_{l=1}^{m_{d-1}} C_{d-1,l}\, (k_1+2) \, (\delta^{-1}\varepsilon)^{t_{d-1,l}/2}  \\
< 2(\delta^{-1}\varepsilon)^{1/2}+ k_1^d\sum_{l=1}^{m_{d-1}} (3C_{d-1,l}) \, (\delta^{-1}\varepsilon)^{t_{d-1,l}/2}  .
\end{multline*}
Note that $\min\{1/2,1/2|r|\}^{-1}=\max\{2,2|r|\}\leq k_1+2$.

From (i) and (ii), we find that the inequality holds if we set $m_d=2m_{d-1}+1$, $C_{d,l}=3C_{d-1,l}$ $(1\leq l\leq m_{d-1})$, $C_{d,j+m_{d-1}}=3C_{d-1,j}$ $(1\leq j\leq m_{d-1})$, $C_{d,m_d}=2$, $t_{d,l}=t_{d-1,l}$ $(1\leq l\leq m_{d-1})$, $t_{d,j+m_{d-1}}=t_{d-1,j}/2$ $(1\leq j\leq m_{d-1})$, $t_{d,m_d}=1/2$.
Thus, we have proved this lemma by the induction.
\end{proof}

\begin{lem}\label{lesti}
Fix a positive integer $n$ and a non-negative integer $d$.
There exist $m_{n,d}\in\Z_{>0}$, $C_{n,d,1},C_{n,d,2},\dots,C_{n,d,m_{n,d}}>0$ and $0<t_{n,d,1},t_{n,d,2},\dots,t_{n,d,m_{n,d}}\leq 1$ such that, for any $\varepsilon>0$, any $\delta>0$, any $p\in \mathbf{F}(n,d,\delta)$, and any $k\in(\Z_{>0})^{\oplus n}$, we have
\[
\vol(\mathbf{U}_k(p,\varepsilon))< \big( k_1^d+k_2^d+\cdots+k_n^d \big) \sum_{l=1}^{m_{n,d}} C_{n,d,l} \,  (\delta^{-1}\varepsilon)^{t_{n,d,l}}.
\]
\end{lem}
\begin{proof}
We will show this lemma by an induction for $n$.
By Lemma \ref{lesti1} we have already proved the case $n=1$.
Hence, we assume $n>1$.

Let $\check x=(x_2,\dots,x_n)$ and there exist polynomials $p_j(\check x)$ such that\[
p(x)=\sum_{j=0}^d p_j(\check x)\, x_1^j.
\]
Since $\|p\|>\delta$, there exists $l$ $(0\leq l\leq d)$ such that $\|p_l\|>\delta$.
We fix such an integer $l$.
Namely, $p_l$ belongs to $\mathbf{F}(n-1,d,\delta)$.

When $|p_l(\check x)|>(\delta\varepsilon)^{1/2}$, the polynomial $p$ is regarded as a polynomial of $x_1$ in $\mathbf{F}(1,d,(\delta\varepsilon)^{1/2})$.
Therefore, if we set $\check k=(k_2,\dots,k_n)$ and $p_{\check x}(x_1)=p(x)$, then $\mathbf{U}_k(p,\varepsilon)$ is contained in the union
\[
(\mathbf{U}_{k_1}\times \mathbf{U}_{\check k}(p_l,(\delta\varepsilon)^{1/2})) \cup \{(x_1,\check x)\in \mathbf{U}_{k_1} \times \mathbf{U}_{\check k}  \mid  x_1\in \mathbf{U}_{k_1}(p_{\check x},\varepsilon) , \;\;  p_{\check x}\in\mathbf{F}(1,d,(\delta\varepsilon)^{1/2})  \}.
\]
By this inclusion, we find
\begin{multline*}
\vol(\mathbf{U}_k(p,\varepsilon))\\
<\{ k_2^d+\cdots + k_n^d \} \sum_{l=1}^{m_{n-1,d}} C_{n-1,d,l}  \{\delta^{-1}(\delta\varepsilon)^{1/2}\}^{t_{n-1,d,l}}+ k_1^d \sum_{j=1}^{m_{1,d}} C_{1,d,j} \{(\delta\varepsilon)^{-1/2} \varepsilon \}^{t_{1,d,j}} \\
<\{ k_1^d+ k_2^d+\cdots + k_n^d \}\left[ \sum_{l=1}^{m_{n-1,d}} C_{n-1,d,l}  (\delta^{-1}\varepsilon)^{t_{n-1,d,l}/2} +  \sum_{j=1}^{m_{1,d}} C_{1,d,j} (\delta^{-1}\varepsilon)^{t_{1,d,j}/2}  \right] .
\end{multline*}
If we set $m_{n,d}=m_{n-1,d}+m_{1,d}$, $C_{n,d,l}=C_{n-1,d,l}$ $(1\leq l\leq m_{n-1,d})$, $t_{n,d,l}=t_{n-1,d,l}/2$ $(1\leq l\leq m_{n-1,d})$, $C_{n,d,m_{n-1,d}+j}=C_{1,d,j}$ $(1\leq j\leq m_{1,d})$, $t_{n,d,m_{n-1,d}+j}=t_{1,d,j}/2$ $(1\leq j\leq m_{1,d})$, then the proof is completed.
\end{proof}

For $x=(x_1,\dots,x_n)\in\R^n$, we set $\|x\|=(x_1^2+\cdots+x_n^2)^{1/2}$.
Let $\d x$ denote the Lebesgue measure on $\R^{\oplus n}$.
\begin{prop}\label{esti}
Fix $n$, $r\in\Z_{>0}$, $d\in\Z_{\geq 0}$ and take polynomials $p_1,\dots,p_r$ in $\mathbf{F}(n,d)$.
If $M>1+rd+n$, then the integral
\begin{equation}\label{e41}
\int_{\R^{\oplus n}} (1+\|x\|)^{-M} \prod_{j=1}^r\big|\log|p_j(x)|\big| \, \d x
\end{equation}
is convergent.
\end{prop}
\begin{proof}
Let $\Xi_n$ denote the set of subsets of $\{1,2,\dots,n\}$.
Fix a positive integer $N$.
For each $\xi\in\Xi_n$, we set
\[
\mathbf{D}(\xi)=\sum_{k\in L(\xi)} \int_{\mathbf{U}_k}(1+\|x\|)^{-M} \prod_{j=1}^r\big|\log|p_j(x)|\big| \, \d x
\]
where 
\[
L(\xi)=\{ (k_1,\dots,k_n)\in (\Z_{>0})^{\oplus n} \mid k_j\geq N \;\; \text{if $j\in \xi$}, \quad 0<k_l<N \;\; \text{if $l\not\in \xi$} \}.
\]
Dividing the domain $\R^{\oplus n}$, we have $\eqref{e41}=\sum_{\xi\in\Xi_n}\mathbf{D}(\xi)$.

Assume that $\xi$ belongs to $\Xi_n$.
There exists a constant $\delta>0$ such that $p_1,\dots,p_r$ belong to $\mathbf{F}(n,d,\delta)$.
We choose a small real number $\varepsilon>0$ such that $0<\delta^{-1}\varepsilon <1$.
For each $\tilde \xi\in \Xi_r$, we set
\[
\mathbf{D}(\xi,\tilde \xi,\varepsilon)=\sum_{k\in L(\xi)} \int_{\mathbf{U}_k(\tilde \xi,\varepsilon)}(1+\|x\|)^{-M} \prod_{j=1}^r\big|\log|p_j(x)|\big| \, \d x
\]
where
\[
\mathbf{U}_k(\tilde \xi , \varepsilon)=\{x\in\mathbf{U}_k \mid  |p_j(x)|<\varepsilon \;\;\text{if $j\in \tilde \xi$}, \quad |p_j(x)|\geq \varepsilon  \;\;\text{if $j\not\in \tilde \xi$} \}.
\]
Since it is clear that
\[
\eqref{e41}=\sum_{\xi\in\Xi_n}\sum_{\tilde \xi\in \Xi_r} \mathbf{D}(\xi,\tilde \xi,\varepsilon),
\]
we have only to prove the convergence of $\mathbf{D}(\xi,\tilde \xi,\varepsilon)$ for $M>1+rd+n$.
We set $m=|\xi|(\geq 0)$ and $u=|\tilde \xi|(\geq 0)$ and we may assume that
\[
\xi=\{1,\dots,m\} \quad \text{and} \quad  \tilde \xi=\{1,\dots,u\}
\]
without loss of generality.
There exist a constant $0<\epsilon<1$ and a constant $C_1>0$ such that
\[
\big| \log|p_j(x)|\big| <C_1\times (1+\|x\|)^{\epsilon/r}
\]
for every $u<j\leq r$, every $k\in L(\xi)$, and every $x\in \mathbf{U}_k(\tilde \xi,\varepsilon)$, because $\log|p_j(x)|$ has the lower bound $\log\varepsilon$ and a upper bound given by the product of $\log\|x\|$ with a positive constant.
We can choose a large natural number $N$ and a constant $C_2>0$ such that
\[
1+\|x\|>C_2\times ( k_1+\cdots+k_m )
\]
 holds for any $k\in L(\xi)$ and any $x\in\mathbf{U}_k$ (the constants $N$ and $C_2$ do not depend on them).
Hence, we have
\[
\mathbf{D}(\xi,\tilde \xi,\varepsilon)<\mathrm{(constant)}\times\sum_{k\in L(\xi)} (k_1+\cdots+k_m)^{-M+\epsilon} \int_{\mathbf{U}_k(\tilde \xi,\varepsilon)}\prod_{j=1}^u\big|\log|p_j(x)|\big| \, \d x.
\]
We consider the decomposition
\[
\int_{\mathbf{U}_k(\tilde \xi,\varepsilon)}\prod_{j=1}^u\big|\log|p_j(x)|\big| \, \d x = \sum_{l=0}^\inf\int_{\mathbf{U}_k(\tilde \xi,2^{-l}\varepsilon)\setminus \mathbf{U}_k(\tilde \xi,2^{-l-1}\varepsilon)}\prod_{j=1}^u\big|\log|p_j(x)|\big| \, \d x
\]
Since $\mathbf{U}_k(\tilde \xi,\varepsilon)$ is included in $\cap_{j=1}^u\mathbf{U}_k(p_j,\varepsilon)$, we have
\[
\prod_{j=1}^u\big|\log|p_j(x)|\big|<|\log(2^{-l-1}\varepsilon)|^u<\mathrm{(constant)}\times (1+l)^u
\]
for every $x\in\mathbf{U}_k(\tilde \xi,2^{-l}\varepsilon)\setminus \mathbf{U}_k(\tilde \xi,2^{-l-1}\varepsilon)$.
By Lemma \ref{lesti} and $0<\delta^{-1}\varepsilon<1$, there exist a constant $0<t\leq 1$ and a constant $C_3>0$ such that
\[
\vol(\mathbf{U}_k(\tilde \xi,2^{-l}\varepsilon)\setminus \mathbf{U}_k(\tilde \xi,2^{-l-1}\varepsilon)) < C_3^u\times \{k_1^d+\cdots+k_n^d\}^u\times (\delta^{-1}\varepsilon)^{t u}\times 2^{-ltu}.
\]
Therefore,
\begin{align*}
\mathbf{D}(\xi,\tilde \xi,\varepsilon)<&\mathrm{(constant)}\times\sum_{k\in L(\xi)} (k_1+\cdots+k_m)^{-M+\epsilon+rd} \sum_{l=0}^\inf  (1+l)^u \times 2^{-ltu} \\
<&\mathrm{(constant)}\times\sum_{k_1=N}^\inf\cdots\sum_{k_m=N}^\inf (k_1+\cdots+k_m)^{-M+\epsilon+rd} \\
<&\mathrm{(constant)}\times\sum_{\tilde k=1}^\inf \tilde k^{-M+\epsilon+rd+m-1} .
\end{align*}
Hence, if $M>\epsilon+rd+n$, then $\mathbf{D}(\xi,\tilde \xi,\varepsilon)$ is convergent.
Since we have chosen $0<\epsilon<1$, the proof is completed.
\end{proof}

\subsection{Absolute convergence of $J_M^L(u)$}\label{acJ}

We may assume that $G$ is a closed subgroup of $\GL(N)$ over $F$ without loss of generality.
Let $\cL$ denote the set of standard Levi subgroups over $F$.
For $x=\prod_v(x_{ij})\in\GL(N,\R)$, a height function $\|\;\; \|_\inf$ is defined by
\[
\|x\|_\inf=\Big(\sum_{i,j}|x_{ij,\inf}|^2 + \sum_{i,j}|x_{ij,\inf}'|^2 \Big)^{1/2}   
\]
where we set ${}^t\!x^{-1}=(x_{ij}')$.
A height function $\|\;\;\|_\inf$ on $G(\R)$ is obtained from that of $\GL(N,\R)$.
For any $x$, $y$ in $G(\R)$, one clearly finds
\[
\|x\|_\inf\geq 1,\quad \|xy\|_\inf\leq \|x\|_\inf \|y\|_\inf, \quad  \|x^{-1}\|_\inf\leq C_0\|x\|_\inf^{N_0}
\]
where $C_0$ and $N_0$ are certain positive constants (cf. \cite[p.919]{Arthur12} and \cite[p.70]{Arthur5}).

\begin{prop}\label{lac}
Let $M$, $L\in\cL$, $M\subset L$ and let $Q$ be a parabolic subgroup in $\cP(L)$.
There exists a positive constant $\mathfrak{M}$ such that, if $k>\mathfrak{M}$ and a continuous function $f$ on $G(\R)$ satisfies
\[
|f(g)|<\mathrm{(constant)}\times \|g\|_\inf^{-k} \qquad (g\in G(\R)),
\]
then $J_M^L(u,f_Q)$ absolutely converges for any unipotent orbits $u$ in $M(\R)$ with respect to \eqref{euw}.
\end{prop}
\begin{proof}
As in Section \ref{secu1}, for a unipotent element $u$ in $L(\R)$, we choose a parabolic subgroup $P_1\in \cP^L(M)$ and we have the subgroup $U_{>2}(\R)$ in $M(\R)$ and the open orbit $O_u(\R)$ in $\fm_2(\R)$.
We have only to show that the integral
\begin{multline}\label{absu}
\int_{N_Q(\R)}\d n_Q \int_{N_{P_1}(\R)} \d \nu \int_{U_{>2}(\R)} \d y \int_{O_u(\R)}\d x \\
\|\exp(x) y\nu n_Q\|_\inf^{-k} \, |\varphi(x)|^{1/2}\, |w_M(\exp(x)y \nu)|
\end{multline}
is convergent.
Note that there exist positive constants $C_1$, $C_2$ such that
\[
C_1\|g\|_\inf<\| k_1gk_2\|_\inf<C_2\|g \|_\inf
\]
holds for any $g\in G(\R)$ and any $k_1,k_2\in\bK_\inf$.

We may assume that $U$, $N_{P_1}$, and $N_Q$ are contained in the unipotent radical $N_{P_0}$ of a minimal parabolic subgroup $P_0$ in $\cP$.
Let $\fn_0$ denote the Lie algebra of $N_{P_0}$.
Choose a Euclidean norm $\|\; \|$ on the $\R$-vector space $\fn_0(\R)$.
\begin{lem}\label{llac}
Assume that $\fg_1$ is a subspace of $\fn_0$, $\fg_2$ is a subalgebra of $\fn_0$ and they satisfy $[\fg_1,\fg_2]\subset \fg_2$ and $\fg_1\cap\fg_2=0$.
There exist a natural number $N_3$ and a positive constant $C_3$ such that
\[
\|\exp(X)\exp(Y)\|_\inf^{N_3}>C_3(1+\|X\|)\, \|\exp(Y)\|_\inf \qquad (X\in\fg_1(\R),\;\; Y\in\fg_2(\R)).
\]
\end{lem}
\begin{proof}
It is obvious that there exists positive constants $N_4$, $N_5$, $C_4$, $C_5$ such that
\[
\|\exp(Z)\|_\inf^{N_4}>C_4(1+\|Z\|) \quad \text{and} \quad (1+\|Z\|)^{N_5}>C_5\|\exp(Z)\|_\inf
\]
for any $Z\in \fn_0(\R)$, because $\exp(Z)$ (resp. $\log(x)$) is a polynomial of $Z\in\fn_0(\R)$ (resp. $x\in N_0(\R)$).
Since $\|\exp(X)\exp(Y)\|_\inf>C_0^{-1}\|\exp(X)\|_\inf^{-N_0}\|\exp(Y)\|_\inf$, we get
\[
\|\exp(X)\exp(Y)\|_\inf>C_0^{-1}C_5^{N_0}(1+\|X\|)^{-N_0N_5}\|\exp(Y)\|_\inf.
\]
Using the Campbell-Hausdorff formula, we derive
\[
\|\exp(X)\exp(Y)\|_\inf^{N_4}>C_4(1+\|X\|).
\]
Hence, the lemma follows from them.
\end{proof}

Let $\fn_Q$ (resp. $\fn_1$) denote the Lie algebra of $N_Q$ (resp. $N_{P_1}$).
Then, we see that $\fu_{>1}\cap \fn_1=\fu_{>1}\cap \fn_Q=\fn_1\cap \fn_Q=0$, $[\fn_1,\fn_Q]\subset \fn_Q$, and $[\fu_{>1},\fn_1+\fn_Q]\subset \fn_1+\fn_Q$.
Furthermore, there exist positive constants $N_6$ and $C_6$ such that $|\varphi(x)|^{1/2}$ is bounded by $C_6\times (1+\|x\|)^{N_6}$.
Hence, by Lemma \ref{llac}, the integral \eqref{absu} is bounded by a constant multiple of the integral
\begin{multline*}
\int_{\fn_Q(\R)}\d X_Q \int_{\fn_1(\R)} \d X_1 \int_{\fu_{>1}(\R)} \d X_u  \\
(1+\|X_u\|+ \|X_1\|+\|X_Q\|)^{-\frac{k}{N_3^2N_4}+N_6} \, |w_M(\exp(X_u)\exp(X_1))|
\end{multline*}
where $\d X_Q$, $\d X_1$, $\d X_u$ are Lebesgue measures.
By \cite[Corollary 4.3 and Lemma 5.4]{Arthur3} the factor $w_M(\exp(X_u)\exp(X_1))$ is a linear combination of functions of the form $\prod_{j=1}^m\log|p_j(\exp(X_u)\exp(X_1))|$ for certain polynomials $p_1,\dots,p_m$ on $U_{>1}(\R)N_{P_1}(\R)$.
Hence, by Proposition \ref{esti}, the above integral is absolutely convergent for any sufficiently large integer $k$.
\end{proof}

There exists a sufficiently small $\varepsilon>0$ such that, for any $M$, $L\in\cL$ $(M\subset L)$, any $Q\in\cP(L)$, any unipotent orbit $u\in M(\R)$, and any $0<p<\varepsilon$, the linear functional $\cC^p(G(\R))\ni f\to J_M^L(u,f)\in \C$ is continuous in the topology of $\cC^p(G(\R))$.
This fact follows from Proposition \ref{lac}, because one can show that the function $f$ in $\cC^p(G(\R))$ is bounded by $\mathrm{(constant)}\times\|\;\|_\inf^{-k}$ using a well-known upper bound of spherical functions.
The continuity for $\cC^p(G(\R))\ni f\to J_M^L(u,f)\in \C$ is a local analogue of Hoffmann's geometric estimates \cite{Hoffmann}.

\section{Dimension formula for spaces of holomorphic Siegel cusp forms}\label{s3}

Let $n$ be an integer greater than $0$.
We set $G=\Sp(n)$ throughout this section (see \eqref{eqsp} for the definition of $\Sp(n)$).
Especially, $G$ satisfies all the assumptions in Section~\ref{s2}.
Let $\fg$ denote the Lie algebra of $G$ over $\Q$.
We write $M_0$ for the minimal Levi subgroup which consists of diagonal matrices in $G$, and a positive root system $\Phi_0$ for $\fa_0=\mathrm{Lie}(M_0)$ in $\fg$ is defined by
\[
\Phi_0=\{ \chi \in X(M_0)_\Q \mid \text{$\chi=\chi_i-\chi_j$ $(1\leq i<j\leq n)$ or $\chi=\chi_i+\chi_j$ $(1\leq i,$ $j\leq n)$}    \}
\]
where $\chi_i(a)=a_i$ for $a=\diag(a_1,\dots,a_n,a_1^{-1},\dots,a_n^{-1})$.
By the adjoint action $\mathrm{Ad}(a)x=axa^{-1}$, we have the minimal parabolic subgroup $P_0$ corresponding to $\Phi_0$.
Its unipotent radical is denoted by $N_0$.
We set $\alpha_j=\chi_j-\chi_{j+1}$ $(1\leq j\leq n-1)$ and $\alpha_n=2\chi_n$. Then, the set $\Delta_0$ of simple roots is given by
\[
\Delta_0=\{\alpha_1,\alpha_2,\dots,\alpha_{n-1},\alpha_n\}.
\]
A maximal compact subgroup $\bK_\inf$ of $G(\R)$ is defined as
\[
\bK_\inf=\{\begin{pmatrix}A&-B \\ B&A \end{pmatrix}\in G(\R)\}.
\]
The group $\bK_\inf$ is identified with the compact unitary group
\[
\U(n,\C/\R)=\{g\in\GL(n,\C)\mid g\, \overline{{}^t\!g}=I_n\}
\]
via the isomorphism
\[
\bK_\inf\ni \begin{pmatrix}A&-B \\ B&A \end{pmatrix}\mapsto A+iB\in \U(n,\C/\R).
\]
We set $\bK_v=G(\Z_v)$ for each finite place $v$ of $\Q$.

Now, $G$ is a closed subgroup of $\GL(2n)$ over $\Q$.
For each $x=(x_{ij})_{1\leq i,j\leq 2n}\in G(\R)$, a height function $\|\;\; \|_\inf$ is defined by
\[
 \|x\|_\inf=\Big(\sum_{1\leq i,j\leq 2n}|x_{ij,\inf}|^2\Big)^{1/2}.
\]
This height function $\|\;\; \|_\inf$ is essentially the same as that of Section \ref{acJ}.
Particularly, $\|\;\; \|_\inf$ satisfies the same properties as in Section \ref{acJ}.

\subsection{Holomorphic discrete series of $\Sp(n,\R)$}\label{s31}

The Lie algebra $\fg(\R)$ has a compact Cartan subalgebra $\fh_\R$ over $\R$, which is isomorphic to $(i\R)^{\oplus n}$.
Set $\fh_\C=\fh_\R\otimes_\R\C$ and $\fh_\C^*=\mathrm{Hom}(\fh_\C,\C)$.
Let $\mathbf{e}_j$ denote the element in $\fh_\C^*$ defined by
\[
\mathbf{e}_j(a_1,a_2,\dots,a_n)=a_j\qquad ((a_1,\dots,a_n)\in\fh_\C).
\]
Then, we may assume that the root system $\varPhi$ of $(\fg(\C),\fh_\C)$ is given by
\[
\varPhi=\varPhi_{K,+}\sqcup(-\varPhi_{K,+})\sqcup\varPhi_{n,+}\sqcup(-\varPhi_{n,+})
\]
where
\[
\varPhi_{K,+}=\{\mathbf{e}_j-\mathbf{e}_k\mid 1\leq j< k\leq n\}\quad \text{and} \quad \varPhi_{n,+}=\{\mathbf{e}_j+\mathbf{e}_k\mid 1\leq j\leq k\leq n\}
\]
($\varPhi_{K,+}$ (resp. $\varPhi_{n,+}$) is the set of positive compact (resp. non-compact) roots).
Let $P(\varPhi)$ denote the set of analytically integral forms on $\fh^*_\C$ (cf. \cite[p.84]{Knapp}).
Then, we have
\[
P(\varPhi)=\Z \mathbf{e}_1\oplus \Z \mathbf{e}_2\oplus \cdots\oplus \Z \mathbf{e}_n
\]
and $P(\varPhi)$ is identified with $\Z^{\oplus n}$ by $c_1\mathbf{e}_1+\cdots +c_n\mathbf{e}_n\mapsto (c_1,\dots,c_n)$.
We denote by $\fk_\C$ the complexification of the Lie algebra of $\bK_\inf$.
We may assume that $\fk_\C$ contains $\fh_\C$ and $\varPhi_{K,+}\sqcup(-\varPhi_{K,+})$ is the root system of $(\fk_\C,\fh_\C)$.

The Weyl group of $(\fk_\C,\fh_\C)$ is isomorphic to the symmetric group $S_n$ of degree $n$.
Furthermore, the Weyl group of $(\fg(\C),\fh_\C)$ is isomorphic to the semi-direct product $\{\pm 1\}^n\rtimes S_n$.
The group $S_n$ acts on $\fh_\C^*$ as $w\cdot \mathbf{e}_j=\mathbf{e}_{w^{-1}(j)}$ and $\{\pm 1\}^n$ acts on $\fh_\C^*$ as $\epsilon\cdot \mathbf{e}_j=\epsilon_j \mathbf{e}_j$ $(\epsilon=(\epsilon_1,\dots,\epsilon_n)\in\{\pm 1\}^n)$.
Therefore, from \cite[Theorem 9.20 in p.310]{Knapp} we find that equivalence classes of discrete series of $G(\R)$ are parameterized by the set
\[
\mathrm{HCP}=\left\{(l_1,l_2,\dots,l_n)\in P(\varPhi) \mid  \begin{array}{c} l_1>l_2>\cdots>l_n , \\ l_k+l_j\neq0 \;\; (1\leq k\leq j\leq n)\end{array} \right\}
\]
(this is called the Harish-Chandra parameter).
If $(l_1,l_2,\dots,l_n)\in\mathrm{HCP}$ satisfies $l_n>0$, then the corresponding discrete series $\sigma$ is said to be holomorphic and $\sigma$ is characterized by its minimal $\bK_\inf$-type with the highest weight
\[
(l_1,l_2,\dots,l_n)+(1,2,\dots,n).
\]
Hence, there is the one-to-one correspondence between the holomorphic discrete series representations of $G(\R)$ and the $n$-tuples $(k_1,k_2,\dots,k_n)$ of integers with $k_1\geq k_2\geq\cdots\geq k_n>n$ (this is called the Blattner parameter).

\subsection{Spaces of Siegel cusp forms}\label{s3.2}

For a symmetric matrix $X$ in $M(n,\R)$, we write $X>0$ if $X$ is positive definite.
We denote by $\fH_n$ the Siegel upper half space of degree $n$, i.e.,
\[
\fH_n=\{Z\in M(n,\C) \mid  {}^t\!Z=Z , \;\;   \mathrm{Im}(Z)>0  \}.
\]
The group $G(\R)$ acts on $\fH_n$ as
\[
g\cdot Z=(AZ+B)(CZ+D)^{-1} , \qquad g=\begin{pmatrix}A&B\\ C&D \end{pmatrix}\in G(\R), \quad Z\in\fH_n
\]
where $A$, $B$, $C$, $D\in M(n,\R)$.
Let $\rho$ be a finite-dimensional irreducible rational representation of $\GL(n,\C)$.
We write $H_\rho$ for a representation space of $\rho$ over $\C$.
We set
\[
J_\rho(g,Z)=\rho(CZ+D),\qquad g=\begin{pmatrix}A&B\\ C&D \end{pmatrix}\in G(\R), \quad Z\in\fH_n.
\]
The automorphic factor $J_\rho$ satisfies
\[
J_\rho(g_1g_2,Z)=J_\rho(g_1,g_2\cdot Z)\, J_\rho(g_2,Z), \qquad g_1,g_2\in G(\R),\quad Z\in\fH_n
\]
and
\[
J_\rho(k,iI_n)=\rho(k),\qquad k\in \bK_\inf.
\]

Let $\Gamma$ be an arithmetic subgroup of $G(\Q)$.
For each positive definite symmetric matrix $Y$ in $M(n,\R)$, there exists an orthogonal matrix $k$ such that $kY{}^t\!k=a$ is a diagonal matrix and we set $Y^{1/2}=ka^{1/2}\,{}^t\!k$ where $a^{1/2}$ is defined by $a^{1/2}=\diag(a_1^{1/2},\dots,a_n^{1/2})$ for $a=\diag(a_1,\dots,a_n)$.
A vector space $S_\rho(\Gamma)$ over $\C$ is defined by
\[
S_\rho(\Gamma)=\left\{ \phi:\fH_n\to H_\rho \mid \begin{array}{c} \text{$\phi$ is holomorphic,} \\ \phi(\gamma\cdot Z)=J_\rho(\gamma,Z)\,\phi(Z)\quad (\gamma\in\Gamma,\;\; Z\in\fH_n), \\ \sup_{Z\in\fH_n}\|\rho(\mathrm{Im}(Z)^{1/2})\phi(Z)\|_{H_\rho} <\inf \end{array}        \right\}.
\]
The space $S_\rho(\Gamma)$ is called the space of Siegel cusp forms of weight $\rho$ with respect to $\Gamma$.
If $\rho=\det^k$, then $S_\rho(\Gamma)$ is simply denoted by $S_k(\Gamma)$.
The functions in $S_k(\Gamma)$ are called scalar valued Siegel cusp forms.
If $\dim H_\rho$ is greater than one, then the functions in $S_\rho(\Gamma)$ are said to be vector valued Siegel cusp forms.

\subsection{Godement's spherical function}\label{s3.1}

Let $\sigma$ be a holomorphic discrete series representation of $G(\R)$ corresponding to the $n$-tuple $\bk=(k_1,k_2,\dots,k_n)$ $(k_1\geq k_2\geq\cdots\geq k_n>n)$.
By Weyl's unitary trick, a finite dimensional irreducible polynomial representation $(\rho,H_\rho)$ of $\GL(n,\C)$ is obtained from an extension of the minimal $\bK_\inf$-type $(\tau,H_\sigma(\tau))$ of $\sigma$.
Namely $H_\rho$ is identified with $H_\sigma(\tau)$ and $\rho(k)= \tau(k)$ holds for any $k\in\bK_\inf$.
We identify $H_\rho$ with the vector space $\C^{d_\tau}$ of column vectors of degree $d_\tau(=\dim H_\rho)$.
Then, $\rho$ becomes a homomorphism from $\GL(n,\C)$ to $\GL(d_\tau,\C)$.
Furthermore, we have
\[
\rho\big(\, \overline{{}^t\!g}\, \big)=\overline{{}^t\!\rho(g)} \quad \text{in $\GL(d_\tau,\C)$ for $g\in\GL(n,\C)$}
\]
(cf. \cite[Expos\'e 5]{Godement}) and $\rho$ corresponds to $\bk$ in the usual sense.
Let $B$ denote the Borel subgroup of $\GL(n)$ which consists of upper triangle matrices.
The $n$-tuple $\bk$ is identified with the rational character on $M_B(\C)$ such that
\[
\bk(\diag(a_1,a_2,\dots,a_n))=a_1^{k_1}a_2^{k_2}\cdots a_n^{k_n}.
\]
We are assuming that the highest weight of $\rho$ is $\bk$, that is, there exists a highest weight vector $v_0$ in $H_\rho$ such that
\[
\rho(an)v_0=\bk(a)v_0 \qquad (a\in M_B(\C) , \;\; n\in N_B(\C)).
\]

For $\C^{d_\tau}$-valued functions $\varphi_1$ and $\varphi_2$ on $\fH_n$, we set
\[
\langle \varphi_1,\varphi_2\rangle=\int_{\fH_n}\overline{{}^t\!\varphi_1(Z)}\, \rho(\mathrm{Im}(Z))\,\varphi_2(Z)\, \d Z
\]
where $\d Z$ is a $G(\R)$-invariant measure on $\fH_n$ induced from the Haar measures on $G(\R)$ and $\bK_\inf$.
A representation space $H_\sigma$ of $\sigma$ is realized by
\[
H_\sigma=\left\{ \varphi\in \fH_n\to H_\rho  \mid  \text{$\varphi$ is holomorphic,}  \quad \langle \varphi,\varphi\rangle  <\inf   \right\}
\]
and $\sigma(g)$ is acting on $H_\sigma$ as
\[
(\sigma(g)\varphi)(Z)=J_\rho(g^{-1},Z)^{-1} \varphi(g^{-1}\cdot Z) \qquad (\varphi\in H_\sigma,\;\; g\in G(\R),\;\; Z\in\fH_n).
\]
In particular, $\sigma$ is unitary for the inner product $\langle \, , \, \rangle$.
Note that $H_\sigma$ is identified with a Hilbert space whose elements are $\C^{d_\tau}$-valued functions on $G(\R)$ by the lifting $\varphi\mapsto \tilde\varphi$, $\tilde\varphi(g)=J_\rho(g^{-1},iI_n)^{-1}\varphi(g^{-1}\cdot iI_n)$.
For details of this realization, we refer to \cite{Harish-Chandra}, \cite[p.3--4]{Langlands1}, and \cite[Expos\'e 6]{Godement}.

We recall the notations $\mathrm{pr}_\tau$, $\psi_{\sigma,\tau}$, $d_\sigma$ and $\fs$ given in Section \ref{s25}.
A $\tau$-spherical function $\Psi_{\sigma,\tau}$ of $\sigma$ is defined by
\[
\Psi_{\sigma,\tau}(g)=\mathrm{pr}_\tau\circ \sigma(g)\circ \mathrm{pr}_\tau \qquad (g\in G(\R)).
\]
Clearly, one can see that the $\tau$-spherical function $\Psi_{\sigma,\tau}$ satisfies
\begin{equation}\label{stpp}
\Psi_{\sigma,\tau}(kgk')=\rho(k)\, \Psi_{\sigma,\tau}(g) \, \rho(k') \qquad (k,k'\in\bK_\inf,\;\; g\in G(\R))
\end{equation}
and
\[
\fs(g)=d_\tau^{-1}d_\sigma\overline{\psi_{\sigma,\tau}(g)}=d_\tau^{-1}d_\sigma\overline{\Tr(\Psi_{\sigma,\tau}(g))} \qquad (g\in G(\R)).
\]
\begin{lem}\label{lgo}
The spherical function $\Psi_{\sigma,\tau}$ satisfies
\[
\Psi_{\sigma,\tau}(g)=J_\rho(g^{-1},iI_n)^{-1}\, \rho\left(\frac{g^{-1}\cdot iI_n+iI_n}{2i}\right)^{-1}  \qquad (g\in G(\R)).
\]
Note that $\overline{{}^t\!\Psi_{\sigma,\tau}(g)} =\Psi_{\sigma,\tau}(g^{-1})$ holds.
\end{lem}
\begin{proof}
In \cite[Expos\'e 6, Theorem 5]{Godement}, a reproducing kernel $K_{\sigma,\tau}(Z_1,Z_2)$ of $H_\sigma$ is given by
\[
K_{\sigma,\tau}(Z_1,Z_2)=d_\tau^{-1}d_\sigma\, \rho\left(\frac{Z_1-\overline{Z_2}}{2i} \right)^{-1} \qquad (Z_1,Z_2\in \fH_n),
\]
i.e.,
\[
f(Z_1)=\int_{\fH_n}K_{\sigma,\tau}(Z_1,Z_2)\, \rho(\mathrm{Im}(Z_2))\, f(Z_2)\, \d Z_2 \qquad (f\in H_\sigma)
\]
holds.
By the same argument as in \cite{WW}, we can show that
\[
d_\tau^{-1}d_\sigma J_\rho(x^{-1},iI_n)\, \Psi_{\sigma,\tau}(xy^{-1})\, \overline{{}^t\!J_\rho(y^{-1},iI_n)}\quad (Z_1=x^{-1}\cdot iI_n,\;\; Z_2=y^{-1}\cdot iI_n)
\]
also becomes a reproducing kernel of $H_\sigma$.
It is well-known that a reproducing kernel of $H_\sigma$ uniquely exists by the Riesz representation theorem (see, e.g., \cite[Proposition IX.2.1]{FK}).
Thus, this lemma is proved.
\end{proof}
By Lemma \ref{lgo}, one sees that
\[
\Psi_{\sigma,\tau}(g)=\rho\left(\frac{({}^t\!A+{}^t\!D)\, i-{}^t\!B+{}^t\!C}{2i}\right)^{-1} \quad \text{if} \quad g=\begin{pmatrix}A&B\\ C&D\end{pmatrix}\in G(\R).
\]
Note that $\Psi_{\sigma,\tau}(I_{2n})=I_n$ and
\begin{equation}\label{kft2}
\rho\left(\frac{g\cdot Z_1-\overline{g\cdot Z_2}}{2i}\right)^{-1}=J_\rho(g,Z_1) \, \rho\left(\frac{Z_1-\overline{Z_2}}{2i}\right)^{-1}\overline{{}^t\!J_\rho(g,Z_2)}
\end{equation}
hold $(Z_1,Z_2\in\fH_n$, $g\in G(\R))$ (cf. \cite[6-18]{Godement}).
\begin{lem}\label{l31}
Assume that $\rho$ corresponds to $\bk=(k_1,k_2,\dots,k_n)$ as above.

\vspace{1mm}
\noindent
(i) If $\sigma_0$ and $\tau_0$ are corresponding to $(k_n,\cdots,k_n)$, then there exists a positive constant $C_1$ such that
\[
|\psi_{\sigma,\tau}(g)|<C_1\times |\psi_{\sigma_0,\tau_0}(g)| \qquad (g\in G(\R)).
\]

\vspace{1mm}
\noindent
(ii) There exists a positive constant $C_2$ such that
\[
|\psi_{\sigma,\tau}(g)|<C_2\times  \| g\|_\inf^{-k_n} \qquad (g\in G(\R)).
\]
%where $\|\; \|_\inf$ is the hight function on $G(\R)$ defined by \eqref{height} and the inclusion $\Sp(n)\subset \GL(2n)$.
\end{lem}
\begin{proof}
By the property \eqref{stpp} and the Cartan decomposition $g=kak'$ $(k,k'\in \bK_\inf$, $a=\diag(a_1,\dots,a_n,a_1^{-1},\dots,a_n^{-1})$, we get
\[
\psi_{\sigma,\tau}(g)=\Tr(\rho(k'k)\, \rho(\diag(\frac{a_1+a_1^{-1}}{2},\dots,\frac{a_n+a_n^{-1}}{2}))^{-1}).
\]
We note that $|\psi_{\sigma_0,\tau_0}(k_1gk_2)|=|\psi_{\sigma_0,\tau_0}(g)|$ and $\| k_1gk_2\|_\inf=\|g \|_\inf$ hold for any $g\in G(\R)$ and any $k_1,k_2\in\bK_\inf$.
Hence, the inequalities follow from this.
\end{proof}
Note that an inequality similar to Lemma \ref{l31} (ii) is deduced from \cite[Theorem and Lemma 1]{Milicic} for any discrete series.

\subsection{A property of $\dim S_\rho(\Gamma)$}\label{secap}

Let $K_0$ be an open compact subgroup of $G(\A_\fin)$ and let $\Gamma=G(\Q)\cap(G(\R)K_0)$.
When $n$ is greater than $1$, it is known that any arithmetic subgroup of $G(\Q)$ contains a principal congruence subgroup (cf. \cite[Theorem 9.14]{PR}).
Hence, if $n>1$, for each arithmetic subgroup $\Gamma'$ of $G(\Q)$, we can choose an open compact subgroup $K_0'$ of $G(\A_\fin)$ such that $\Gamma'=G(\Q)\cap(G(\R)K_0')$.

We assume that $\rho$ and $\sigma$ satisfies the same conditions as in Section~\ref{s3.1}.
It is well-known that
\[
\dim S_\rho(\Gamma)=m(\sigma,\Gamma)
\]
holds (cf. \cite{Wallach}).
Hence, we can study $\dim S_\rho(\Gamma)$ using results given in Section~\ref{s2}.
Note that the multiplicity of an anti-holomorphic discrete series equals that of holomorphic one obtained by its complex conjugation.

\begin{prop}\label{s4po}
Let $\sigma$ be a holomorphic discrete series of $G(\R)$ corresponding to the $n$-tuple $\bk=(k_1,k_2,\dots,k_n)$.
Assume that $\Gamma$ satisfies Condition \ref{c21} and a finite dimensional polynomial representation $\rho$ also corresponds to $\bk$.
If $k_n>n+1$, then $\dim S_\rho(\Gamma)$ is a polynomial of $k_1$, $k_2,\dots,k_n$ with constant coefficients.
To be more exact, there exists a polynomial $g_\Gamma(X_1,X_2,\dots,X_n)$ in $\C[X_1,X_2,\dots,X_n]$ such that $\dim S_\rho(\Gamma)=g_\Gamma(k_1,k_2,\dots,k_n)$ holds if $k_n>n+1$.
\end{prop}
\begin{proof}
By the result of \cite{Hiraga}, $\sigma$ satisfies Condition \ref{c2} if $k_n>n+1$.
Hence, applying Proposition \ref{p1}, we get the equality
\[
\dim S_\rho(\Gamma)=m(\sigma,\Gamma)=\sum_{M\in\cL}\frac{|W^M_0|}{|W^G_0|}\sum_{u\in(\cU_M(\Q))_{M,S}} a^M(S,u) \, I_M^G(u,\tfs) \, J_M^M(u,h_Q).
\]
Using \eqref{lim} and Lemmas \ref{unr1} and \ref{unr2} (see Section \ref{lbw}), one can find that $I_M^G(u,\tfs)$ becomes a polynomial of $(k_1,k_2,\dots,k_n)=(l_1+1,l_2+2,\dots,l_n+n)$ with constant coefficients.
Since $a^M(S,u)$ and $J_M^M(u,h_Q)$ do not depend on $\bk$, the proof is completed.
\end{proof}
From the proof one can also see that the total degree of the polynomial $g_\Gamma$ is less than or equal to $n^2$.
We will use this proposition to extend a weight range in the proof of the dimension formula (Theorem \ref{t37}) for $S_\rho(\Gamma)$.

\subsection{Behaviors of invariant weighted orbital integrals}\label{lbw}

We use an argument similar to \cite[the proof of Lemma 5.3]{Arthur4} to study $I_M^G(u,\tfs)$ for each $u$ in $\cU_M(\R)$.
By the argument in \cite[p.277]{Arthur4} (see also the proof of Lemma \ref{p1l}), we have
\[
I_M^G(u,\tfs)=\lim_{a\to 1} I_M^G(au,\tfs)
\]
where $a$ ranges over small generic points on $A_M(\R)$.
From now, we consider $I_M^G(au,\tfs)$ in the limit and we fix such an element $a$ in $A_M(\R)$.
It follows from \cite[(2.3)]{Arthur8} that there exists a small invariant neighborhood $U$ of $1$ in $M(\R)$ such that $I_M^G(a\mu,\tfs)\overset{(M,a)}{\sim} 0$ $(\mu\in U)$.
This means that there exists a function $\phi$ in $C_c^\inf(M(\R))$ such that
\[
I_M^G(a\mu,\tfs)=I_M^M(a\mu,\phi) \quad (\mu\in U).
\]
Set $\phi_a(x)=\phi(ax)$.
Then, $\phi_a$ belong to $C_c^\inf(M(\A))$ and we have
\[
I_M^G(a\mu,\tfs)=I_M^M(\mu,\phi_a) \quad (\mu\in U).
\]
For each maximal torus $T$ of $M$ over $\R$, we set $\ft=\mathrm{Lie}(T)$, $\fm=\mathrm{Lie}(M)$, and $\Delta_T^M(H)=\prod_{\alpha>0}(e^{\frac{1}{2}\alpha(H)}-e^{-\frac{1}{2}\alpha(H)})$ $(H\in\ft(\R))$ where $\alpha$ runs over the positive roots of $(\fm,\ft)$ with respect to some ordering.
Applying Harish-Chandra's limit formula \cite[(A.3)]{Arthur4}, we find that there exists a finite set of triples $(T_j,C_j,h_j)$ $(1\leq j\leq t)$, where $T_j$ is a maximal torus of $M$ over $\R$, $C_j$ is a chamber in $\ft_j(\R)$, $h_j$ is a harmonic polynomial on $\ft_j(\C)$ ($\deg h_j\leq n^2$), such that
\[
I_M^M(u,\phi_a)=\sum_{j=1}^t \lim_{\text{$H\to 0$ in $C_j$}} \partial(h_j) \, \Delta_{T_j}^M(H)\int_{T_j(\R)\bsl M(\R)}\phi_a(x^{-1}\exp(H)x)\, \d x
\]
where $H$ approaches $0$ through the regular points in $C_j$.
Note that $I_M^G(au,\tfs)=I_M^M(u,\phi_a)$ holds for any unipotent conjugacy class $u$ in $M(\R)$.
If $\exp(H)$ are in $U$, then we find
\[
I_M^G(a\exp(H),\tfs)=I_M^M(\exp(H),\phi_a)=c_j\times \Delta_{T_j}^M(H)\int_{T_j(\R)\bsl M(\R)}\phi_a(x^{-1}\exp(H)x)\, \d x
\]
where $c_j$ is a constant depending only on the measures and the chamber $C_j$ $(H\in C_j)$.
Therefore, since we may replace $c_j^{-1}h_j$ by $h_j$, we have
\begin{equation}\label{lim}
I_M^G(au,\tfs)=\sum_{j=1}^t \lim_{\text{$H\to 0$ in $C_j$, $\exp(H)$ in $U$}} \partial(h_j) \, I_M^G(a\exp(H),\tfs).
\end{equation}
\begin{lem}\label{unr1}
Assume that $M$ is not cuspidal over $\R$.
Then, we have $I_M^G(u,\tfs)=0$ for any $u$ in $\cU_M(\R)$.
\end{lem}
\begin{proof}
For any such element $a\in A_M(\R)$, the element $a\exp(H)$ is not $\R$-elliptic in $M$.
Hence, we have $I_M^G(a\exp(H),\tfs)=0$ by Lemma \ref{p1l}.
This implies that $I_M^G(u,\tfs)$ vanishes by \eqref{lim}.
\end{proof}

Next we assume that $M$ is cuspidal over $\R$.
Then, we have a maximal torus $B$ over $\R$ such that $B(\R)$ is contained in $(\bK_\inf\cap M(\R))A_M(\R)^0$.
Note that $I_M^G(a\exp(H),\tfs)$ vanishes unless $T_j$ is $G(\R)$-conjugate to $B(\R)$, since $\tfs$ is cuspidal.
For each regular element $a\exp(H)$ in $B(\R)$, it follows from Arthur's works (see, e.g., \cite{Arthur7}) that $I_M^G(a\exp(H),\tfs)$ equals $|D^G(a\exp(H))|^{1/2}\overline{\Theta_\sigma(a\exp(H))}$ up to a constant multiple, where $\Theta_\sigma$ is the character of $\sigma$.
Hence, by the character formulas of \cite{Hecht} and \cite{Martens} for holomorphic discrete series, we get the following lemma.
\begin{lem}\label{unr2}
Assume that $M$ is cuspidal over $\R$.
We may assume
\[
M\cong \GL(1)^{\lambda_1}\times \GL(2)^{\lambda_2}\times \Sp(m) \quad (\lambda_1+2\lambda_2+m=n),
\]
and set
\begin{multline*}
a=\diag(a_1,\dots,a_{\lambda_1},b_1,b_1,\dots,b_{\lambda_2},b_{\lambda_2},\overbrace{1,\dots,1}^m, \\
a_1^{-1},\dots,a_{\lambda_1}^{-1},b_1^{-1},b_1^{-1},\dots,b_{\lambda_2}^{-1},b_{\lambda_2}^{-1},1,\dots,1)\in A_M(\R)^0 \quad (a_j>0,\;\; b_j>0),
\end{multline*}
\begin{multline*}
H=\diag(\overbrace{0,\dots,0}^{\lambda_1},-it_1,it_1,\dots,-it_{\lambda_2},it_{\lambda_2},-i\theta_1,\dots,-i\theta_m, \\
0,\dots,0,it_1,-it_1,\dots,it_{\lambda_2},-it_{\lambda_2},i\theta_1,\dots,i\theta_m)\in \fg(\C).
\end{multline*}
There exists a maximal torus $T$ in $M$ over $\R$ such that $T(\R)/A_M(\R)$ is compact and $\exp(H)$ is turned into an element of $T(\R)$ by a Cayley transform.
For the Harish-Chandra parameter $(l_1,\dots,l_n)$ of $\sigma$, $1>a_1>a_2>\cdots>a_{\lambda_1}>0$, $1>b_1>\cdots>b_{\lambda_2}>0$, and a regular element $H$, there exists a constant $c_M$ such that
\begin{align*}
&I_M^G(a\exp(H),\tfs)=c_M\times \sum_{w\in S_n} \sgn(w) a_1^{l_{w(1)}} \cdots a_{\lambda_1}^{l_{w(\lambda_1)}} \\
& \qquad \qquad \times (b_1e^{it_1})^{l_{w(\lambda_1+1)}}(b_1e^{-it_1})^{l_{w(\lambda_1+2)}}\cdots (b_{\lambda_2}e^{it_{\lambda_2}})^{l_{w(\lambda_1+2\lambda_2-1)}}(b_{\lambda_2}e^{-it_{\lambda_2}})^{l_{w(\lambda_1+2\lambda_2)}} \\
& \qquad \qquad \times e^{i\theta_1 l_{w(\lambda_1+2\lambda_2+1)}}\cdots e^{i\theta_m l_{w(\lambda_1+2\lambda_2+m)}}.
\end{align*}
\end{lem}

\begin{lem}\label{nr1}
Let $M$ be a Levi subgroup in $\cL$ and let $u$ be a unipotent element in $M(\R)$.
Then, we have $I_M^G(u,\tfs)=0$ unless $M\cong G$, $\GL(1)\times \Sp(n-1)$, or $\GL(2)\times \Sp(n-2)$.
In addition, we get $I_M^G(u,\tfs)=0$ if $M\cong \GL(2)\times \Sp(n-2)$ and the $\GL(2)$-part of $u$ is not trivial.
\end{lem}
\begin{proof}
If $M$ is not cuspidal over $\R$, it is clear that $I_M^G(u,\tfs)=0$ by Lemma \ref{unr1}.
Hence, we assume that $M$ is cuspidal over $\R$, that is, $M\cong \GL(1)^{\lambda_1}\times \GL(2)^{\lambda_2}\times \Sp(m)$.
By \eqref{lim} and Lemma \ref{unr2}, for each unipotent element
\[
u=(1,\dots,1,u_1,\dots,u_{\lambda_2},u')\in\GL(1,\R)^{\lambda_1}\times\GL(2,\R)^{\lambda_2}\times\Sp(m,\R),
\]
we get
\begin{align*}
 I_M^G(u,\tfs)=&\lim_{a\to 1}I_M^G(au,\tfs) \\
=& \text{(constant)}\times \lim_{a_1,\dots,a_{\lambda_1},b_1,\dots,b_{\lambda_2}\to 1,\; 0<a_*<1,\; 0<b_*<1} \\
& \times \sum_{k=1}^{\lambda_2}\lim_{t_k\to 0} \partial(h_{u_k}(t_k))\times \sum_j \lim_{C_{j,u'}\text{chamber},\;(\theta_1,\dots,\theta_m)\to 0}\partial(h_{j,u'}(\theta_1,\dots,\theta_m)) \\
& \sum_{w\in S_n}\sgn(w) a_1^{l_{w(1)}} \cdots a_{\lambda_1}^{l_{w(\lambda_1)}} \\
& \times (b_1e^{it_1})^{l_{w(\lambda_1+1)}}(b_1e^{-it_1})^{l_{w(\lambda_1+2)}}\cdots (b_{\lambda_2}e^{it_{\lambda_2}})^{l_{w(\lambda_1+2\lambda_2-1)}}(b_{\lambda_2}e^{-it_{\lambda_2}})^{l_{w(\lambda_1+2\lambda_2)}} \\
& \times e^{i\theta_1 l_{w(\lambda_1+2\lambda_2+1)}}\cdots e^{i\theta_m l_{w(\lambda_1+2\lambda_2+m)}}.
\end{align*}
One obviously finds that this limit is convergent.
The sums of $k$, $j$, $w$ and the limits of $a_*$, $b_*$, $t_*$, $\theta_*$ can be exchanged freely.

First, we consider the case $\lambda_1>1$.
For any element $w_0$ in $S_n$ and $w_1=(w_0(1),w_0(2))$, we have
\[
w_1w_0=w_0(1,2).
\]
Hence
\begin{align*}
& \sum_{w\in S_n}\sgn(w) a_1^{l_{w(1)}} \cdots a_{\lambda_1}^{l_{w(\lambda_1)}} \\
& \times (b_1e^{it_1})^{l_{w(\lambda_1+1)}}(b_1e^{-it_1})^{l_{w(\lambda_1+2)}}\cdots (b_{\lambda_2}e^{it_{\lambda_2}})^{l_{w(\lambda_1+2\lambda_2-1)}}(b_{\lambda_2}e^{-it_{\lambda_2}})^{l_{w(\lambda_1+2\lambda_2)}} \\
& \times e^{i\theta_1 l_{w(\lambda_1+2\lambda_2+1)}}\cdots e^{i\theta_m l_{w(\lambda_1+2\lambda_2+m)}}\\
&= \sum_{w_0\in S_n/\langle (1,2)\rangle}  \sgn(w_0)  \sum_{s\in\langle w_1 \rangle}\sgn(s)  a_1^{l_{sw_0(1)}} a_2^{l_{sw_0(2)}} a_3^{l_{w_0(3)}} \cdots a_{\lambda_1}^{l_{w_0(\lambda_1)}} \\
&\quad \times (b_1e^{it_1})^{l_{w_0(\lambda_1+1)}}(b_1e^{-it_1})^{l_{w_0(\lambda_1+2)}}\cdots (b_{\lambda_2}e^{it_{\lambda_2}})^{l_{w_0(\lambda_1+2\lambda_2-1)}}(b_{\lambda_2}e^{-it_{\lambda_2}})^{l_{w_0(\lambda_1+2\lambda_2)}} \\
&\quad \times e^{i\theta_1 l_{w_0(\lambda_1+2\lambda_2+1)}}\cdots e^{i\theta_m l_{w_0(\lambda_1+2\lambda_2+m)}} .
\end{align*}
Thus, we have $I_M^G(u,\tfs)=0$ since
\begin{align*}
& \lim_{a_1,a_2\to 1}(\sgn(w_0)a_1^{l_{w_0(1)}}a_2^{l_{w_0(2)}}+\sgn(w_1w_0)a_1^{l_{w_1w_0(1)}}a_2^{l_{w_1w_0(2)}})\\
& =\lim_{a_1,a_2\to 1}(a_1^{l_{w_0(1)}}a_2^{l_{w_0(2)}}-a_1^{l_{w_0(2)}}a_2^{l_{w_0(1)}})=1-1=0.
\end{align*}

Next, we assume that the factor $u_1$ in $\GL(2)$ is not trivial.
We know the fact $\partial(h_{1,u_1}(t_1))=1$ up to a constant multiple (cf. \cite[Section 3 in Chapter~XI]{Knapp}).
Hence, we get $I_M^G(u,\tfs)=0$ by using the same argument as above for $w=w_0$ and $w_1=(w_0(\lambda_1+1),w_0(\lambda_1+2))$, $(w_1w_0=w_0(\lambda_1+1,\lambda_1+2))$, and
\[
\lim_{t_1\to 0}(b_1e^{it_1})^{l_{w(\lambda_1+1)}}(b_1e^{-it_1})^{l_{w(\lambda_1+2)}}=b_1^{l_{w(\lambda_1+1)}+l_{w(\lambda_1+2)}}.
\]

Third, we assume that all the factors of $u$ in $\GL(2)$ is trivial and $\lambda_2>1$.
From \cite[Section 3 in Chapter~XI]{Knapp}, one finds $\partial(h_{1,u_1}(t_1))=\partial/\partial t_1$ up to a constant multiple.
Therefore, we get $I_M^G(u,\tfs)=0$ using
\[
\lim_{t_1\to 0}\frac{\partial}{\partial t_1}(b_1e^{it_1})^{l_{w(\lambda_1+1)}}(b_1e^{-it_1})^{l_{w(\lambda_1+2)}}=(l_{w(\lambda_1+1)}-l_{w(\lambda_1+2)})\times i\times b_1^2
\]
and
\[
\sum_{w\in S_4}\sgn(w)\, (l_{w(1)}-l_{w(2)})(l_{w(3)}-l_{w(4)})= \sum_{i=1,2, \; j=3,4}(- 1)^{i+j}\sum_{w\in S_4}\sgn(w)l_{w(i)}l_{w(j)}=0.
\]

Finally, we consider the case $\lambda_1=\lambda_2=1$ and $u_1=I_2$. In this case one can show $I_M^G(u,\tfs)=0$ by using the fact $\sum_{w\in S_3}\sgn(w)\, \{ l_{w(2)}-l_{w(3)}\}=0$.
\end{proof}

\subsection{Vanishing of unipotent weighted orbital integrals}

It is known that geometric unipotent conjugacy classes in $G$ are in one-to-one correspondence with the set of partitions of $2n$ in which odd parts occur with even multiplicity (cf. \cite[Theorem 5.1.3]{CM}).
By the inclusion $\fg=\fsp(n)\subset \sl(2n)$, each nilpotent orbit in $\fg$ is regarded as a nilpotent orbit in $\sl(2n)$ and its partition means the block sizes in the Jordan normal form.
\begin{lem}\label{lu2}
Let $C$ be a geometric conjugacy class in $\cU_G$.
Then, there exist a non-negative integer $l_0$ and two partitions $l=(l_1,l_2,\dots,l_r)$, $m=(m_1,m_2,\dots,m_s)$ such that $l_0\geq l_1$, $l_0+2|l|+|m|=n$, and the partition
\[
((2r+1)^{2l_r},(2r-1)^{2l_{r-1}-2l_r},\dots,3^{2l_1-2l_2},1^{2l_0-2l_1},(2s)^{m_s},(2s-2)^{m_{s-1}-m_s},\dots,2^{m_1-m_2})
\]
corresponds to $C$, where $r$ (resp. $s$) denotes the length of $l$ (resp. $m$).
Furthermore, there exists a unipotent element $u$ in $C(\Q)$ such that the semisimple element
\[
H=\begin{pmatrix}A&O_n \\ O_n&-A\end{pmatrix}\in\fg,
\]
\[
A=\diag(\cdots,\overbrace{2j,\cdots,2j}^{2l_j},\overbrace{2j-1,\cdots,2j-1}^{m_j},\cdots,\overbrace{2,\cdots,2}^{2l_1},\overbrace{1,\cdots,1}^{m_1},\overbrace{0,\cdots,0}^{l_0}),
\]
determines the canonical parabolic subgroup of $u$.
\end{lem}
\begin{proof}
The above mentioned results for $H$, $L$, and $\fg_2$ are deduced from explicit standard triples (cf. \cite[Proposition 5.2.3]{CM}).
\end{proof}

Let $C$ be a geometric unipotent conjugacy class of $G$.
By Lemma \ref{lu2}, $C$ always has a $\Q$-rational point and we choose an element $u$ in $C(\Q)$.
We recall the notations $\fg_j$, $U_{>2}$, and so on, given in Section \ref{s26} for the element $u$.
\begin{lem}\label{lu3}
The following two conditions are equivalent.
\begin{itemize}
\item[(1)] There exists an integer $j$ such that $C$ corresponds to $(2^j,1^{2n-2j})$.
\item[(2)] $U_{>2}$ is the trivial group.
\end{itemize}
\end{lem}
\begin{proof}
One can easily show this lemma by Lemma \ref{lu2}.
\end{proof}

\begin{lem}\label{vl1}
Set $n_{11}(t)=\begin{pmatrix}I_n&B(t) \\ O_n&I_n \end{pmatrix}$, $B(t)=(b_{ij})_{1\leq i,j\leq n}$, $b_{11}=-t$, and $b_{ij}=0$ unless $i=j=1$.
For any $x$, $y\in G(\R)$, we have
\[
\int_\R \fs(x^{-1}n_{11}(t)y)\, \d t=0.
\]
\end{lem}
\begin{proof}
Using \eqref{kft2} we get
\begin{multline*}
d_\tau d_\sigma^{-1}\overline{\fs(x^{-1}n_{11}(t)y)}\\
=\Tr\left\{ \overline{{}^t\!J_\rho(y,iI_n)^{-1}} J_\rho(x,iI_n)^{-1}\rho\left(2i\, (-B(t)+x\cdot iI_n - \overline{y\cdot iI_n})^{-1}\right)\right\}.
\end{multline*}
There exist rational functions $h(x,y)$, $g_{ij,1}(x,y)$, $g_{ij,2}(x,y)$ for $x$, $y$ such that
\[
(-B(t)+x\cdot iI_n - \overline{y\cdot iI_n})^{-1} = \left( \frac{g_{ij,1}(x,y)}{t+h(x,y)}+g_{ij,2}(x,y)\right)_{1\leq i,j\leq n} .
\]
Since there exists a polynomial representation $\rho_0$ such that $\rho=\det^{n+1}\otimes \rho_0$, there exist rational functions $g_m(x,y)$ for $x$ and a natural number $N$ such that
\[
\fs(x^{-1}n_{11}(t)y)=\sum_{m=n+1}^N \frac{g_m(x,y)}{(t+\overline{h(x,y)})^m}.
\]
Hence, we have
\[
\int_\R\fs(x^{-1}n_{11}(t)y)\, \d t=\sum_{m=n+1}^N \int_\R \frac{g_m(x,y)}{(t+\overline{h(x,y)})^m}\d t=0.
\]
Thus, the proof is completed.
\end{proof}

\begin{lem}\label{vl2}
Fix a unipotent conjugacy class $u$ in $G(\R)$.
Assume that the integral $J_G(u,\fs)$ is absolutely convergent.
Then, we have $J_G(u,\fs)=0$ if $u$ does not correspond to any partition of the form $(2^j,1^{2n-2j})$.
\end{lem}
\begin{proof}
By Lemmas \ref{lu2}, \ref{lu3} and \ref{vl1}, one easily finds that the vanishing $J_G(u,\fs)=0$ follows from the integration over $U_{>2}(\R)$ in \eqref{euw}.
\end{proof}

\begin{lem}\label{nr2}
Let $M$ be a Levi subgroup in $\cL$ such that $M$ is isomorphic to $\GL(1)\times \Sp(n-1)$ or $\GL(2)\times \Sp(n-2)$.
Let $u$ be a unipotent conjugacy class in $M(\R)$.
Assume that the integral $J_M^G(u,\fs)$ absolutely converges.
If $M\cong\GL(1)\times \Sp(n-1)$, then we have $J_M^G(u,\fs)=0$ unless $u=1$.
If $n>2$, $M\cong\GL(2)\times \Sp(n-2)$ and the $\GL(2)$-part of $u$ is trivial, then $J_M^G(u,\fs)=0$ for any $u$.
\end{lem}
\begin{proof}

We choose a standard parabolic subgroup $P_1$ in $\cL(M)$ according to Section \ref{secu1}.
We will consider the weight factor $w_M(1,\mu \nu)$ in \eqref{euw} and the vanishing will follow from the integration over $N_{P_1}(\R)$.

An element $e_j$ in $\fa_0$ is defined by
\[
e_j=\diag(\overbrace{0,\dots,0}^{j-1},1,\overbrace{0,\dots,0}^{n-j},\overbrace{0,\dots,0}^{j-1},-1,\overbrace{0,\dots,0}^{n-j}).
\]
It is clear that $\{e_j\mid 1\leq j\leq n\}$ is a basis of $\fa_0$.
A bilinear form $\langle \; , \; \rangle$ on $\fa_0$ is defined by $\langle e_{j_1},e_{j_2}\rangle=\delta_{j_1,j_2}$.
By this bilinear form $\langle \; , \; \rangle$, we may identify $e_j$ with the element $[\sum_{l=1}^n a_le_l\mapsto a_j]$ $(a_l\in\R)$ in $\fa_0^*$.
So, we also denote it by $e_j$.

Let $P$ be a parabolic subgroup in $\cL(M)$.
The set $\widehat\Delta_P$ consists of the single weight $\varpi_P$.
If $M\cong\GL(1)\times \Sp(n-1)$, then 
\[
\varpi_P=e_1\quad \text{or} \quad -e_{n+1}. 
\]
If $M\cong\GL(2)\times \Sp(n-2)$, then 
\[
\varpi_P=e_1+e_2 \quad \text{or} \quad -e_{n+1}-e_{n+2}. 
\]
By \cite[p.237]{Arthur3}, for $x$ in $G(\R)$, we find
\begin{equation}\label{eq3}
v_P(x,\varpi_P)(=\exp(-\varpi_P(H_P(x))))=\|x^{-1}e_1\| \quad \text{or} \quad \|x^{-1}e_{n+1}\|  
\end{equation}
if $M\cong\GL(1)\times \Sp(n-1)$, and
\begin{equation}\label{eq4}
v_P(x,\varpi_P)=\|(x^{-1}e_1)\wedge (x^{-1}e_2)\| \quad \text{or} \quad \|(x^{-1}e_{n+1})\wedge (x^{-1}e_{n+2})\| 
\end{equation}
if $M\cong\GL(2)\times \Sp(n-2)$.

First, we will compute some matrices for $M\cong\GL(m)\times \Sp(n-m)$.
For the geometric conjugacy class $C$ containing $u$ in $M$, we have a grading $\fm=\oplus_j \fm_j$ for $\fm=\mathrm{Lie}(M)$ in the usual manner.
We may set
\[
\mathbf{n}=\begin{pmatrix}I_m&X&Z&Y \\ 0&I_{n-m}& {}^t\!Y & 0 \\ 0&0& I_m& 0 \\ 0&0&  -{}^t\!X& I_{n-m} \end{pmatrix} \in N_{P_1}(\R)
\]
($Z-X{}^t\!Y$ is a symmetric matrix),
\[
\mu=\begin{pmatrix} I_m & & & \\ & A & & B \\ & & I_m & \\ & C & & D \end{pmatrix}\in M(\R)\cap N_0(\R) ,\quad g=\begin{pmatrix}A&B \\ C&D \end{pmatrix}\in \cU_{\Sp(n-m)}(\R),
\]
\[
g=\exp(\tilde X),\;\; \tilde X=\sum_{j>1}\tilde X_j\in \bigoplus_{j>1}\fm_j(\R),\;\; \tilde X_j\in\fm_j(\R),\;\; \exp(\tilde X)=I_{2n-2m}+D,
\]
\[
\tilde a=\diag(a^{-1},\dots,a^{-1},1,\dots,1,a,\dots,a,1,\dots,1) \in A_M(\R)^0 .\]
In the above, we may assume the variable $\mu$ is included in $N_0(\R)$ by Lemma \ref{lu2}.
By direct calculation, we get
\[
\mathbf{n}^{-1}=\begin{pmatrix}I_m&-X&-Z+X{}^t\!Y-Y{}^t\!X & -Y \\ 0&I_{n-m}& -{}^t\!Y & 0 \\ 0&0& I_m& 0 \\ 0&0&  {}^t\!X& I_{n-m} \end{pmatrix},
\]
\[
\mu^{-1}\tilde a^{-1}\mathbf{n}^{-1}\tilde a \mu=\begin{pmatrix}I_m&-aX'&a^2Z'&-aY' \\ 0&I_{n-m}& -a{}^t\!Y' & 0 \\ 0&0& I_m& 0 \\ 0&0&  a{}^t\!X'& I_{n-m} \end{pmatrix}, 
\]
\[
Z'=-Z+X{}^t\!Y-Y{}^t\!X , \quad (X',Y')=(X,Y)g,
\]
\[
\mu^{-1}\tilde a^{-1}\mathbf{n}^{-1}\tilde a \mu \mathbf{n}=\begin{pmatrix}I_m&X-aX'&a^2Z''&Y-aY' \\ 0&I_{n-m}& {}^t\!Y-a{}^t\!Y' & 0 \\ 0&0& I_m& 0 \\ 0&0&  -{}^t\!X+a{}^t\!X'& I_{n-m} \end{pmatrix}, 
\]
\[
Z'' = (1-a^2)Z + a^2X{}^t\!Y - a^2Y{}^t\!X - aX'{}^t\!Y + aY'{}^t\!X.
\]
Set
\[
\mu^{-1}\tilde a^{-1}\mathbf{n}^{-1}\tilde a \mu \mathbf{n}=\nu=\begin{pmatrix}I_m&U&W&V \\ 0&I_{n-m} & {}^t\!V & 0 \\ 0&0& I_m & 0 \\ 0&0&  -{}^t\!U  & I_{n-m} \end{pmatrix}.
\]
We want to express $\mathbf{n}^{-1}$ as a rational polynomial matrix of variable $\mu$ and $\nu$, because $v_P(\mathbf{n},\varpi_P)$ is the form of \eqref{eq3} or \eqref{eq4}.
By
\[
\begin{pmatrix}U&V \end{pmatrix}=\begin{pmatrix}X&Y \end{pmatrix}-a\begin{pmatrix}X&Y \end{pmatrix}g,
\]
we have
\begin{align}\label{eq1}
\begin{pmatrix} X&Y \end{pmatrix} =& \begin{pmatrix} U&V \end{pmatrix}(1-ag)^{-1}=(1-a)^{-1}\begin{pmatrix} U&V \end{pmatrix}(1-\frac{a}{a-1}D)^{-1} \\
=&(1-a)^{-1}\begin{pmatrix}U&V \end{pmatrix}(1+\sum_{k\geq 1}(\frac{a}{a-1})^k D^k) \nonumber \\
=&(1-a)^{-1}\begin{pmatrix}U&V \end{pmatrix}(1+\sum_{k\geq 1}f_k(\frac{a}{a-1})\tilde X^k) \nonumber
\end{align}
for polynomials $f_k$, where $\deg f_k= k$ $(k\geq 1)$.
By direct calculation,
\[
W=Z'' = (1-a^2)Z - a^2 \begin{pmatrix}X&Y \end{pmatrix}J\begin{pmatrix}{}^t\!X \\{}^t\!Y \end{pmatrix} + \begin{pmatrix}X&Y \end{pmatrix}J\begin{pmatrix}{}^t\!X \\{}^t\!Y \end{pmatrix} -\begin{pmatrix}U&V \end{pmatrix}J\begin{pmatrix}{}^t\!X \\{}^t\!Y \end{pmatrix}
\]
where $J=\begin{pmatrix}O_{n-m}&-I_{n-m} \\ I_{n-m}&O_{n-m}\end{pmatrix}$.
Since $\tilde X_j J= -J{}^t\!\tilde X_j$ and $-Z+X{}^t\!Y-Y{}^t\!X=-Z-\begin{pmatrix}X&Y \end{pmatrix}J\begin{pmatrix} {}^t\!X \\ {}^t\!Y\end{pmatrix}$, we have
\begin{align}\label{eq2}
&-Z+X{}^t\!Y-Y{}^t\!X\\
&=-(1-a^2)^{-1}W-(1-a^2)^{-1}\begin{pmatrix}U&V \end{pmatrix}J\begin{pmatrix} {}^t\!X \\ {}^t\!Y\end{pmatrix} \nonumber \\
& =-(1-a^2)^{-1}W \nonumber  \\
& \quad -(1-a)^{-1}(1-a^2)^{-1}\begin{pmatrix}U&V \end{pmatrix}(1+\sum_{k\geq 1}(-1)^kf_k(\frac{a}{a-1})\tilde X^k)J\begin{pmatrix}{}^t\!U \\{}^t\!V \end{pmatrix} \nonumber \\
&= - (1-a^2)^{-1}W -(1-a)^{-1}(1-a^2)^{-1}\begin{pmatrix}U&V \end{pmatrix}J\begin{pmatrix}{}^t\!U \\{}^t\!V \end{pmatrix} \nonumber \\
&\quad - (1-a)^{-1}(1-a^2)^{-1}\begin{pmatrix}U&V \end{pmatrix}(\sum_{k\geq 1}(-1)^kf_k(\frac{a}{a-1})\tilde X^k)J\begin{pmatrix}{}^t\!U \\{}^t\!V \end{pmatrix}. \nonumber 
\end{align}
If we set
\[
D^k=(\sum_{u\geq 1}\frac{1}{u!}\tilde X^u)^k=\sum_{q\geq 1}b_{k,q} \tilde X^q\quad (b_{k,q}>0),
\]
then, by \eqref{eq1} we get
\[
f_k(\frac{a}{a-1})=\sum_{m\geq 1} b_{m,k}(\frac{a}{a-1})^m=\frac{1}{(a-1)^k}\sum_{m=1}^k b_{m,k}a^m(a-1)^{k-m}.
\]
Since the numerator of $f_k(\frac{a}{a-1})$ does not vanish at $a=1$, the order of the pole of $f_k(\frac{a}{a-1})$ at $a=1$ is just $k$.

We will prove that the variable $W$ does not appear in $\lim_{a\to 1}w_P(\lambda,a,\mu\nu)$.
Then, the weight factor $w_M(1,\mu\nu)$ is not related to $W$.
As a coordinate of $\nu$, we can choose $(\tilde W, U,V)$ such that $W=\tilde W+U{}^t\!V$ and $\tilde W$ is a symmetric matrix.
Therefore, we obtain $J_M^G(u,\fs)=0$ by Lemma \ref{vl1} and the integration for $\tilde W$.

We consider the case $m=1$, i.e., $M=\GL(1)\times \Sp(n-1)$.
Assume that $n>1$ and $\tilde X$ is not trivial.
By \eqref{eq3}, \eqref{eq1} and \eqref{eq2}, it is clear that the variable $W$ does not appear in $\lim_{a\to 1}w_P(\lambda,a,\mu\nu)$.
Therefore, $J_M^G(u,\fs)$ vanishes by Lemma \ref{vl1}.

Next, we treat the case $m=2$, i.e., $M=\GL(2)\times\Sp(n-2)$.
It is enough to see $(\mathbf{n}^{-1}e_{n+1})\wedge(\mathbf{n}^{-1}e_{n+2})$ by \eqref{eq4}.
Hence, we consider the determinant of partial $2\times 2$ matrices in the $n+1$-column and the $n+2$-column of $\mathbf{n}^{-1}$.
Let $t$ denote the maximal number among $\{k\in\Z_{\geq 0} \mid \tilde X^k\neq 0\}$.
When $n\geq 4$, the order of poles at $a=1$ for the determinant of partial matrices in $\begin{pmatrix} X&Y \end{pmatrix}$ is $(t+1)+(t+1-1)=2t+1$ at least.
Besides, the order of poles at $a=1$ for the terms involving $Z$ is equal to $1+(2+t)=t+3$.
Namely, if $t\geq 3$, then $W$ does not appear in $\lim_{a\to 1}w_P(\lambda,a,\mu\nu)$.
We write $m_C$ for the maximal number among $\{j\in\Z\mid \fm_j\neq 0\}$.
If $t< 3$, then we have $m_C\leq 4$.
When $n= 3$, all nilpotent orbits in $M(\R)$ satisfy $m_C\leq 4$.
Therefore, it is sufficient to show the assertion for the case $m_C\leq 4$.

We assume $m_C=4$ $(t=2)$.
Since $\tilde X_j J$ is a symmetric matrix, we may set
\[
-f_1(\frac{a}{a-1})\tilde XJ=\begin{pmatrix}A&B\\ {}^t\!B&O \end{pmatrix}.
\]
By $m_C=4$, we have $A\neq O$ and $B\neq O$.
Furthermore,
\[
\begin{pmatrix}U&V \end{pmatrix}\begin{pmatrix}A&B\\ {}^t\!B&O \end{pmatrix}\begin{pmatrix}{}^t\!U \\{}^t\!V \end{pmatrix}=UA{}^t\!U+UB{}^t\!V+V{}^t\!B{}^t\!U,
\]
\[
UB{}^t\!V+V{}^t\!B{}^t\!U=\begin{pmatrix}u_1B{}^t\!v_1+v_1{}^t\!B{}^t\!u_1 & u_1B{}^t\!v_2+v_1{}^t\!B{}^t\!u_2 \\ u_2B{}^t\!v_1+v_2{}^t\!B{}^t\!u_1 & u_2B{}^t\!v_2+v_2{}^t\!B{}^t\!u_2 \end{pmatrix}
\]
if $U=\begin{pmatrix}u_1\\u_2 \end{pmatrix}$, $V=\begin{pmatrix} v_1 \\ v_2\end{pmatrix}$.
We also set
\[
W=\begin{pmatrix}w_1&w_2\\ w_3&w_4 \end{pmatrix}.
\]
The order of poles at $a=1$ for the terms involving $w_1$, $w_2$, $w_3$, $w_4$ is less than $6$, but that of the products for constitutes in $u_1A{}^t\!u_1$ and $u_2B{}^t\!v_2+v_2{}^t\!B{}^t\!u_2$ is $6$.
Hence, the variable $W$ does not appear in $\lim_{a\to 1}w_P(\lambda,a,\mu\nu)$.
We assume $m_C=2$ $(t=1)$.
In the above computation, we have $B=O$, but $A$ is not $O$.
Hence, the assertion is obvious if one sees the product of constitutes in $u_1A{}^t\!u_2$.
Finally, we consider the case $m_C=0$ $(t=0)$.
Since the matrix $\begin{pmatrix}U&V \end{pmatrix}J\begin{pmatrix}{}^t\!U \\{}^t\!V \end{pmatrix}$ does not vanish in \eqref{eq2} by $n>2$, the variable $W$ does not appear in $\lim_{a\to 1}w_P(\lambda,a,\mu\nu)$.
Thus, we get all the assertions in this lemma.
\end{proof}

\begin{prop}\label{p33}
Assume that Condition \ref{c21} holds.
There exists a sufficiently large real number $\mathfrak{M}>0$ such that, if $k_n>\mathfrak{M}$, then
\[
\dim S_\rho(\Gamma)=\sum_C \lim_{T\to\inf}Z_C^T(\fs h)
\]
where $C$ runs over all the geometric unipotent conjugacy classes corresponding to $(2^j,1^{2n-2j})$ $(0\leq j\leq n)$.
\end{prop}
\begin{proof}
Assume that $k_n$ is greater than $\mathfrak{M}$.
Hence, $\sigma$ is integrable (see \cite{HS}).
By Proposition \ref{lac} and Lemma \ref{l31} (ii), Condition \ref{cuni} holds.
Hence, it follows from Proposition \ref{p2} that $I_M^G(u,\tfs)=J_M^G(u,\fs)$ for any $u$ in $\cU_M(\R)$.
Moreover, by Lemmas \ref{nr1}, \ref{vl2}, and \ref{nr2}, we have
\[
I_M^G(u,\tfs)=J_M^G(u,\fs)=0
\]
unless the following three cases
\begin{itemize}
\item[(i)] $M\cong\GL(1)\times \Sp(n-1)$ and $u=1$,
\item[(ii)] $n=2$, $M\cong\GL(2)$, and $u=1$, 
\item[(iii)] $M=G$ and $u$ corresponds to a partition of the form $(2^j,1^{2n-2j})$.
\end{itemize}
For $M\cong \GL(1)\times \Sp(n-1)$ (resp. $M\cong \GL(2)$) and $u=1$, the unipotent conjugacy class $C_u^G$ corresponds to $(2,1,\dots,1)$ (resp. $(2,2)$) by \cite[Corollary 7.3.4]{CM}.
Therefore, for each geometric conjugacy class $C$, by Lemma \ref{l2} and Proposition \ref{cz1}, the contribution $\lim_{T\to \inf}Z_C^T(\fs h)$ of $C(\Q)$ vanishes unless $C$ corresponds to $(2^j,1^{2n-2j})$.
Hence, the proof is completed by Proposition \ref{t1}.
\end{proof}
By this proposition, one sees that a lot of unipotent contributions vanish.
Note that our proof depends on the classification of nilpotent orbits.

\subsection{Analysis on spaces of symmetric matrices}\label{s3.4}

From now on, we choose normalizations of some measures.
For all places $v$ of $\Q$, non-trivial additive characters $\tilde\psi_v$ on $\Q_v$ are defined by
\[
\tilde\psi_\inf(x)=\exp(2\pi i x) \;\; (x\in\R) \quad\text{and} \quad \tilde\psi_v(x)=\exp(-2\pi i \, [x]_p) \;\; (v=p,\;\; x\in\Q_p)
\]
where $p$ is a prime number and we set
\[
[x]_p=a_{-N}p^{-N}+a_{-N+1}p^{-N+1}+\cdots +a_{-1}p^{-1}
\]
for the $p$-adic expansion $\sum_{k=-N}^\inf a_kp^k$ of $x$ $(0\leq a_l\leq p-1)$.
We set
\[
\tilde\psi(x)=\prod_v \tilde\psi_v(x_v) \qquad (x=(x_v)\in\A).
\]
Then, $\tilde\psi$ is a non-trivial additive character on $\Q\bsl \A$.
Let $\d x_\inf$ denote the Lebesgue measure on $\R$ and let $\d x_v$ denote the Haar measure on $\Q_v$ normalized by $\int_{\Z_v}\d x_v=1$ for $v<\inf$.
Then, for each place $v$ of $\Q$, $\d x_v$ is the self-dual measure on $\Q_v$ with respect to $\tilde\psi_v$.
Furthermore, the Haar measure $\d x=\prod_v\d x_v$ on $\A$ satisfies $\int_{\Q\bsl\A}\d x=1$.
We set
\[
V_r=\{ X\in M(r) \mid {}^t\!x=x \} .
\]
For each place $v$, a Haar measure $\d x_v$ on $V_r(\Q_v)$ is defined by
\[
\d x_v=\prod_{1\leq i\leq j\leq r} \d x_{ij,v} \quad (x_v=(x_{ij,v})\in V(\Q_v))
\]
where $\d x_{ij,v}$ is the same measure on $\Q_v$ as above.
A Haar measure $\d x$ on $V_r(\A)$ is also defined by $\d x=\prod_v \d x_v$.
We set
\[
\psi_v(x_v)=\tilde\psi_v(\Tr(x_v))\;\; (x_v\in V_r(\Q_v)) \quad \text{and} \quad \psi(x)=\tilde\psi(\Tr(x))\;\; (x\in V_r(\A)).
\]
If we set $r=n$, then an invariant measure $\d z$ on $\fH_n$ is defined by 
\[
\d z=\frac{\d x_\inf \, \d y_\inf}{\det(y_\inf)^{n+1}} \qquad (z=x_\inf+iy_\inf\in\fH_r).
\]
We have already normalized the Haar measure $\d k_\inf$ on $\bK_\inf$ by $\int_{\bK_\inf}\d k_\inf=1$.
Hence, a Haar measure $\d g_\inf$ on $G(\R)$ is determined by $\d z$, $\d k_\inf$ and the isomorphism $G(\R)/\bK_\inf\cong \fH_n$.

We give some notations for partitions according to \cite[Chapter I]{Macdonald}.
For each partition $(\lambda_1,\lambda_2,\dots)$, we denote by $\lambda'=(\lambda_1',\lambda_2',\dots)$ the conjugate of $\lambda$ and we set
\[
|\lambda|=\lambda_1+\lambda_2+\cdots \quad \text{and} \quad l(\lambda)=\lambda_1'.
\]
Any partition $\lambda$ is identified with its diagram $\{(i,j)\in \Z\oplus\Z \mid 1\leq j\leq \lambda_i\}$.

Let $n$ be a natural number and assume $1\leq r\leq n$.
Let $\rho_{\bk,n}$ denote a finite dimensional irreducible polynomial representation of $\GL(n,\C)$ corresponding to a partition $\bk=(k_1,k_2,\dots)$ with $l(\bk)\leq n$.
A function $\phi^*_{\bk,n,r}$ on $V_r(\R)$ is defined by
\[
\phi^*_{\bk,n,r}(x)=\Tr\rho_{\bk,n}(I_n-i\begin{pmatrix}x&0 \\ 0&0 \end{pmatrix})^{-1} \qquad (x\in V_r(\R)).
\]
It is clear that
\[
\phi^*_{\bk,n,r}(x)=\Tr\rho_{\bk,n}(\begin{pmatrix}(I_r-ix)^{-1}&0 \\ 0& I_{n-r} \end{pmatrix}) = \sum_{l(\bb)\leq r , \; \beta_r\geq k_n } m_{\bk,\bb}^{n,r} \, \phi^*_{\bb,r,r}(x)
\]
where $\bb=(\beta_1,\beta_2,\dots)$ and $m_{\bk,\bb}^{n,r}$ denotes the multiplicity of the irreducible representation $\rho_{\bb,r}$ in the restriction of $\rho_{\bk,n}$ to $\GL(r,\C)=\{\begin{pmatrix}g&0 \\ 0& I_{n-r} \end{pmatrix}\mid g\in \GL(r,\C)\}$.
Since the branching rule is well-known (cf. \cite[Chapter 8]{GW}), the multiplicity  $m_{\bk,\bb}^{n,r}$ is computable if one fixes the parameters.

For a partition $\bm$ with $l(\bm)\leq r$, a zonal polynomial $\bz_{\bm,r}$ is defined by
\[
\bz_{\bm,r}(x)=c_{\bm,r}\cdot \int_{\SO(r,\R)}\, \Delta_\bm({}^t\!kxk) \, \d k
\]
where $\d k$ is the Haar measure on $\SO(r,\R)$ normalized by $\int_{\SO(r,\R)}\d k=1$, $\Delta_k(x)$ denotes the determinant of $(x_{ij})_{1\leq i,j\leq k}$ for $x=(x_{ij})_{1\leq i,j\leq r}$, and we set
\[
c_{\bm,r}=\prod_{(i,j)\in\bm}(r-i+2j-1)    ,
\]
\[
\Delta_{\bm,r}(x)=\Delta_1(x)^{m_1-m_2}\Delta_2(x)^{m_2-m_3}\cdots\Delta_r(x)^{m_r}
\]
(cf. \cite[p.409]{Macdonald}).
Note that the degree of $\bz_{\bm,r}$ equals $|\bm|$ and $c_{\bm,r}^{-1}\bz_{\bm,r}$ is called a spherical polynomial in the book \cite[Chapter XI]{FK}.
Since $\deg\Tr(\rho_{\bb,r}(x))=|\bb|$, the Schur polynomial $\Tr(\rho_{\bb,r}(x))$ can be expressed as a linear combination
\[
\Tr(\rho_{\bb,r}(x))=\sum_{|\bm|=|\bb|,\, l(\bm)\leq r, \, m_r\geq \beta_r} n_{\bb,\bm,r} \, \frac{\bz_{\bm,r}(x)}{c_{\bm,r}}\qquad (n_{\bb,\bm,r}\in\C)
\]
by \cite[Proposition XI.3.1]{FK} or \cite[p.405]{Macdonald}, where $\bm=(m_1,m_2,\dots)$.
If we fix parameters $\bb$ and $r$, then it is possible to calculate explicitly $n_{\bb,\bm,r}$ by using equalities in \cite[p.114]{Macdonald} and \cite[(2.16) in p.406]{Macdonald}.

Let $\Omega_r$ denote the set of all positive definite symmetric matrices of degree $r$, i.e.,
\[
\Omega_r=\{x\in V_r(\R) \mid x>0\}.
\]
The subset $\Omega_r$ is a symmetric cone in $V_r(\R)$ (cf. \cite[p.8--10]{FK}).
As in \cite[Chapter VII]{FK}, the Siegel integral over $\Omega_r$ is defined by
\[
\Gamma_{\Omega_r}(\bm)=\int_{\Omega_r}\Delta_\bm(x)\, \det(x)^{-\frac{r+1}{2}}\,\exp(-\Tr(x)) \, \d x.
\]
It follows from \cite[Theorem VII.1.1]{FK} that
\[
\Gamma_{\Omega_r}(\bm)=\pi^{\frac{r(r-1)}{4}}\prod_{j=1}^r\Gamma\left(m_j-\frac{j-1}{2}\right)
\]
where $\Gamma(s)$ denotes the Gamma function.
Note that our normalization is different from theirs in the book \cite{FK} and the factor $2^{\frac{r(r-1)}{4}}$ appeared in their description.
A function $\phi_{\bk,n,r}(x)$ on $V_r(\R)$ is defined by
\begin{multline*}
\phi_{\bk,n,r}(x)=\\
\sum_{l(\bb)\leq r , \; \beta_r\geq k_n } m_{\bk,\bb}^{n,r}\sum_{|\bm|=|\bb|,\, l(\bm)\leq r, \, m_r\geq \beta_r} \frac{(2\pi)^{|\bb|}n_{\bb,\bm,r}}{\Gamma_{\Omega_r}(\bm)\,c_{\bm,r} }\bz_{\bm,r}(x) \, \det(x)^{-\frac{r+1}{2}}\, \exp(-2\pi \Tr(x))
\end{multline*}
if $x\in \Omega_r$, and $\phi_{\bk,n,r}(x)=0$ if $x\not\in\Omega_r$.

\begin{lem}\label{l34}
Assume that $k_n>0$.

\vspace{2mm}
\noindent
(i) If $-1<\Re(s)<k_n-r$, then the integral
\[
\int_{V_r(\R)}\phi^*_{\bk,n,r}(x)\,|\det(x)|^s \d x \qquad (s\in\C)
\]
absolutely converges.

\vspace{2mm}
\noindent
(ii) If $\mathrm{Re}(s)+k_n>\frac{r-1}{2}$, then the integral
\[
\int_{V_r(\R)} \phi_{\bk,n,r}(x) \; |\det(x)|^s\,  \d x
\]
is absolute convergent.
Furthermore, if $k_n>\frac{r-1}{2}$, then we have
\[
\int_{V_r(\R)}\phi_{\bk,n,r}(x)\,\psi_\inf(xy)\, \d x=\phi_{\bk,n,r}^*(y).
\]
\end{lem}
\begin{proof}
The assertion (i) follows from \cite[Lemma 19 (i)]{Shintani} and Lemma \ref{l31} (i).
The absolute convergence of (ii) can be proved by the argument as in \cite[Proof of Theorem VII.1.1]{FK}.
Using \cite[Lemma XI.2.3]{FK} or \cite[p.436]{Macdonald} we obtain
\[
\int_{V_r(\R)} \frac{\bz_{\bm,r}(x)}{c_{\bm,r}}  \, \exp(-\Tr(xy)) \, \det(x)^{-\frac{r+1}{2}}\d x = \Gamma_{\Omega_r}(\bm)\, \frac{\bz_{\bm,r}(y^{-1})}{c_{\bm,r}} 
\]
for $k_n>\frac{r-1}{2}$.
From this we have the equalities in (ii).
\end{proof}

Let $\cS(V_r(\A_\fin))$ denote the space of Schwartz-Bruhat functions on $V_r(\A_\fin)$.
Take a test function $\phi_r^*$ in $\cS(V_r(\A_\fin))$ and set
\[
\Phi_{\bk,n,r}^*(x)=\phi^*_{\bk,n,r}(x_\inf)\, \phi^*_r(x_\fin) \qquad (x=(x_\inf,x_\fin)\in V_r(\A)).
\]
A test function $\phi_r$ in $\cS(V_r(\A_\fin))$ is determined by
\begin{equation}\label{e32}
\phi_r(y)=2^{-\frac{r(r-1)}{2}}\int_{V_r(\A_\fin)}\phi_r^*(x)\, \psi_\fin(-xy)\, \d x \qquad (y\in V_r(\A_\fin)) 
\end{equation}
where we set $\psi_\fin=\prod_{v<\inf}\psi_v$.
Then, we have
\[
\phi_r^*(y)=\int_{V_r(\A_\fin)}\phi_r(x)\, \psi_\fin(xy)\, \d x
\]
by the Fourier inversion formula.
Note that the factor $2^{-\frac{r(r-1)}{2}}$ comes from the fact that our measure is not self-dual for $\psi_\fin$.
Besides, the measure $\d x$ on $V_r(\A)$ is self-dual for $\psi$.
We also set
\[
\Phi_{\bk,n,r}(x)=\phi_{\bk,n,r}(x_\inf)\, \phi_r(x_\fin) \quad (x=(x_\inf,x_\fin)\in V_r(\A)).
\]
\begin{prop}\label{p35}
If $k_n>2n$, then the Poisson summation formula
\[
\sum_{x\in V_r(\Q)}\Phi^*_{\bk,n,r}(x)=\sum_{y\in V_r(\Q)\cap \Omega_r}\Phi_{\bk,n,r}(y)
\]
holds.
\end{prop}
\begin{proof}
In \cite[Corollaire in 10-17]{Godement}, he proved a Poisson summation formula for $\phi^*_{\bb,r,r}$.
Hence, this formula follows from Lemma \ref{l34} (ii) and his Poisson summation formula, since we have
\[
\int_{V_r(\A)}\Phi_{\bk,n,r}(x)\,\psi(xy) \d x=\Phi^*_{\bk,n,r}(y)
\]
by Lemma \ref{l34} (ii) and \eqref{e32}.
\end{proof}

Assume that a partition $\bk=(k_1,k_2,\dots)$ satisfies $l(\bk)= n$ and $k_n>2n$.
Let $\sigma$ denote the integrable holomorphic discrete series of $G(\R)$ corresponding to $\bk$.
Under the above mentioned normalization of $\d g_\inf$, Harish-Chandra's formula for the formal degree $d_\sigma$ gives
\begin{equation}\label{e33}
d_\sigma=2^{-n} \, (4\pi)^{-n(n+1)/2} \, d_\tau \, \prod_{1\leq t\leq u\leq n}(k_t+k_u-t-u)
\end{equation}
(see, e.g., \cite[6--21]{Godement}, \cite[Theorem 10.2.4.1]{Warner}), where $d_\tau$ denotes the degree of its minimal $K_\inf$-type $\tau$, that is,
\begin{equation}\label{e34}
d_\tau=\Tr(\rho_{\bk,n}(I_n))=\frac{\prod_{1\leq t<u\leq n}(k_t-k_u+u-t) }{\prod_{1\leq t'<u'\leq n}(u'-t')}.
\end{equation}
For $1\leq r\leq n$, we set
\begin{equation}\label{e35}
\bC(\bk,n,r)=2^{(n-r)^2}\pi^{\frac{(n-r)(n-r+1)}{2}}d_\tau^{-1}\, d_\sigma\,\int_{V_r(\R)}\phi_{\bk,n,r}(x)\, \det(x)^{\frac{r-1}{2}-n}\d x.
\end{equation}
Using \cite[Proposition XI.3.1]{FK} or \cite[p.405]{Macdonald}, we have
\begin{multline}\label{e36}
\bC(\bk,n,r)= d_\tau^{-1}\, d_\sigma\,2^{n^2-nr+\frac{r(r+1)}{2}}\pi^{\frac{n(n+1)}{2}}  \\
 \sum_{l(\bb)= r , \; \beta_r\geq k_n } m_{\bk,\bb}^{n,r} \sum_{|\bm|=|\bb|,\, l(\bm)= r, \, m_r\geq \beta_r} \frac{n_{\bb,\bm,r}\Gamma_{\Omega_r}(\bm+\frac{r-1}{2}-n)}{\Gamma_{\Omega_r}(\bm)} .
\end{multline}
where $\bm+\frac{r-1}{2}-n$ means $(m_1+\frac{r-1}{2}-n,\cdots,m_r+\frac{r-1}{2}-n)$.
Hence, $\bC(\bk,n,r)$ is computable if we fix the parameters $\bk$, $n$, and $r$.
For the case $r=0$, we also set
\begin{equation}\label{e37}
\bC(\bk,n,0)=2^{n^2}\pi^{\frac{n(n+1)}{2}} d_\sigma=2^{-2n} \, d_\tau \, \prod_{1\leq t\leq u\leq n}(k_t+k_u-t-u).
\end{equation}
For the special case $r=n$, the constant $\bC(\bk,n,n)$ has a simple expression as
\[
\bC(\bk,n,n)= 2^{-n}\,\Tr(\rho_{\bk,n}(I_n)),
\]
which can be proved by \cite[Expose 6, Th\'eor\`em 6]{Godement} and the same argument as in \cite[Proof of Theorem 5.7]{Wakatsuki}.
Fix a natural number $k$ which is larger than $2n$.
We shall consider the other special case $\rho_{\bk,n}\cong\det^k$, i.e., $k_1=\cdots=k_n=k$.
In this case, it is easy to see
\begin{equation}\label{e38}
\bC(\bk,n,r)= 2^{-2n+r}\prod_{t=1}^{n-r} \prod_{u=t+r}^n(2k-t-u)
\end{equation}
for $0\leq r\leq n$, because computations of $m_{\bk,\bb}^{n,r}$ and $n_{\bb,\bm,r}$ are trivial issues.
Note that $\bC(\bk,n,n)= 2^{-n}$ in \eqref{e38} for $r=n$.

\subsection{Non-vanishing terms of unipotent elements}\label{s3ze}

Let $1\leq r\leq n$.
We choose a Haar measure $\d h_\inf$ on $\GL(r,\R)$ by $\d h_\inf=|\det(h)|^{-r}\prod_{1\leq i,j\leq r}\d h_{ij}$ for $h=(h_{ij})$.
For every $v<\inf$, a Haar measure $\d h_v$ on $\GL(r,\Q_v)$ is normalized by $\vol(\GL(r,\Z_v))=1$.
A Haar measure $\d h$ on $\GL(r,\A)$ is defined by $\d h=\prod_v \d h_v$.
The group $\GL(r)$ acts on $V_r$ via
\[
v\cdot h={}^t\!hxh \qquad (h\in \GL(r),\;\; x\in V_r).
\]
We recall the functions $\Phi_{\bk,n,r}^*$ and $\Phi_{\bk,n,r}$ on $V_r(\A)$ given in Section \ref{s3.4}.
Let $\Phi=\Phi_{\bk,n,r}^*$ or $\Phi_{\bk,n,r}$.
A zeta integral $Z_r(\Phi,s)$ is defined by
\[
Z_r(\Phi,s)=\int_{\GL(r,\Q)\bsl \GL(r,\A)}|\det(h)|^{2s} \sum_{x\in V_r^0(\Q)}\Phi(x\cdot h)\, \d h \qquad (s\in\C)
\]
where $V_r^0=\{x\in V_r\mid \det(x)\neq 0\}$.
\begin{prop}\label{p36}
(i) If $\mathrm{Re}(s)>\frac{r+1}{2}$ and $\mathrm{Re}(k_n+s)>r$, then $Z_r(\Phi_{\bk,n,r},s)$ is absolutely convergent.
The zeta integral $Z_r(\Phi_{\bk,n,r},s)$ is meromorphically continued to the whole $s$-place.

\vspace{2mm}
\noindent
(ii) If $\mathrm{Re}(s)>\frac{r-1}{2}$, $k_n>2r$ and
\[
\begin{cases}
\mathrm{Re}(s)<k_n & \text{if $r=1$}, \\
2\mathrm{Re}(s)<k_n & \text{if $r=2$}, \\
\mathrm{Re}(s)<k_n -\frac{r-1}{2}& \text{if $r\geq 3$},
\end{cases}
\]
then $Z_r(\Phi_{\bk,n,r}^*,s)$ is absolutely convergent.
The zeta integral $Z_r(\Phi_{\bk,n,r}^*,s)$ is a meromorphic function of $s$ on $\C$.

\vspace{2mm}
\noindent
(iii) The functional equation
\[
Z_r(\Phi_{\bk,n,r}^*,s)=Z_r(\Phi_{\bk,n,r},\frac{r+1}{2}-s)
\]
holds if $k_n>2r$.
\end{prop}
\begin{proof}
In \cite[Proof of Lemma 21]{Shintani}, Shintani proved this proposition for the case $\rho_{\bk,n}=\det^k$.
For only the case $r=2$, his estimation is not accurate, but it was improved in \cite[Proposition 5.5]{Wakatsuki} and generalized to any representation $\rho_{\bk,2}$.
For the other cases, it is easy to generalize his argument to all representations $\rho_{\bk,n}$ if one uses Lemmas \ref{l31} (i) and \ref{l34} and Proposition \ref{p35}.
\end{proof}

For the case $1\leq r\leq n$, the canonical parabolic subgroups of non-trivial unipotent elements corresponding to $(2^r,1^{2n-2r})$ are maximal in $G$ (cf. Lemma \ref{lu2}).
We write $Q_r$ for the standard maximal parabolic subgroup defined by
\[
\Delta_0^{Q_r}=\{ \alpha_1,\dots,\alpha_{r-1},\alpha_{r+1},\dots,\alpha_n\}.
\]
We denote by $C_r$ the geometric unipotent conjugacy class of $G$ containing a $\Q$-rational point whose canonical parabolic subgroup is $Q_r$.
Now the Levi subgroup $M_{Q_r}$ (resp. the vector space $\fg_2$) is isomorphic to $\GL(r)\times \Sp(n-r)$ (resp.  $V_r$) over $\Q$.
Hence, we identify $M_{Q_r}$ (resp. $\fg_2$) with $\GL(r)\times \Sp(n-r)$ (resp.  $V_r$).
The adjoint action of the subgroup $\GL(r)$ on $V_r$ is the same as above and the subgroup $\Sp(n-r)$ is trivially acting on $V_r$.

For the case $r=0$, we set $C_0=\{I_{2n}\}$ and
\[
Z_r(\Phi_{\bk,n,r}^*,s)=Z_r(\Phi_{\bk,n,r},s)=d_\tau.
\]

Let $K_0$ be an open compact subgroup of $G(\A_\fin)$ and let $h_0$ denote the characteristic function of $K_0$.
Set $h=\vol(K_0)^{-1}h_0$.
Furthermore, a function $\phi^*_{0,r}$ on $V_r(\A_\fin)$ is defined by
\begin{equation}\label{s4tf}
\phi_{0,r}^*(x)=\int_{\bK_\fin}h_0(k^{-1}\exp(-2x)k)\, \d k \qquad (x\in \fg_2(\A_\fin)=V_r(\A_\fin))
\end{equation}
(this definition is the same as in \eqref{sme1}).
We also set
\[
\Phi_{0,\bk,n,r}^*(x_\inf x_\fin)=\phi_{\bk,n,r}^*(x_\inf)\phi_{0,r}^*(x_\fin),
\]
which is a special case of $\Phi_{\bk,n,r}^*$ for the substitution $\phi_r^*=\phi_{0,r}^*$ (cf. Section \ref{s3.4}).
\begin{lem}\label{l47}
Let $0\leq r\leq n$.
If $k_n>2n$, then we have
\begin{align*}
\lim_{T\to\inf} Z_{C_r}^T(\fs h)=&\vol(K_0)^{-1} \vol(\Sp(n-r)) \, d_\tau^{-1} d_\sigma \, Z_r(\Phi^*_{0,\bk,n,r},n-\frac{r-1}{2}) \\
=&\vol(K_0)^{-1} \vol(\Sp(n-r)) \, d_\tau^{-1} d_\sigma \, Z_r(\Phi_{0,\bk,n,r},r-n)
\end{align*}
where we set
\[
\vol(\Sp(n-r))=\vol(\Sp(n-r,\Q)\bsl \Sp(n-r,\A)).
\]
\end{lem}
\begin{proof}
By $k_n>2n$, $\sigma$ is now integrable (cf. \cite{HS}).
Hence, it follows from Proposition \ref{cz1} that $Z_{C_r}^T(\fs h)$ absolutely converges.
By $\rho=\rho_{\bk,n}$, one sees
\[
\fs(k^{-1}\exp(-2x)k)=d_\tau^{-1}d_\sigma \, \phi^*_{\bk,n,r}(x)
\]
for $x\in V_r(\R)$ and $k\in \bK_\inf$.
Hence, we get
\begin{multline*}
Z_{C_r}^T(\fs h)=\int_{L(\Q)\bsl L(\A)}\Big( \int_{U(\Q)\bsl U(\A)}F^G(ul,T)\, \d u \Big) \\
\sum_{\mu\in\fg_2^\reg(\Q)}(\fs h)(\exp(l^{-1}\mu l)) \, e^{-\rho_U(H_L(l))}\, d l
\end{multline*}
where $\rho_U(H_L(l))=\log|\det(\mathrm{Ad}(l)|_\fu)|)$.
Since
\[
\Big|\int_{U(\Q)\bsl U(\A)}F^G(ul,T)\, \d u \Big|\leq 1 \quad \text{and} \quad \int_{U(\Q)\bsl U(\A)}F^G(ul,T)\, \d u\to 1 \;\; (T\to\inf)
\]
hold for any $l$ in $L(\Q)\bsl L(\A)$, we finds
\[
\lim_{T\to\inf} Z_{C_r}^T(\fs h)=\vol(K_0)^{-1} \vol(\Sp(n-r)) \, d_\tau^{-1} d_\sigma \, Z_r(\Phi^*_{0,\bk,n,r},n-\frac{r-1}{2})
\]
by Lebesgue's convergence theorem and Proposition \ref{p36} (ii).
Therefore, this lemma follows from Proposition \ref{p36}.
\end{proof}

\subsection{Main theorem}\label{s3.5}

We shall introduce Shintani zeta functions for $\GL(r)$ and $V_r$.
Let $1\leq r\leq n$.
Fix a test function $\phi_r$ in $\cS(V_r(\A_\fin))$.
We assume that $\phi_r$ is $\prod_{v<\inf}\GL(r,\Z_v)$-invariant, i.e. $\phi_r(x\cdot k)=\phi_r(x)$ $(k\in \prod_{v<\inf}\GL(r,\Z_v))$.
Let $U_r$ denote the support of $\phi_r$.
Note that $U_r$ is an open compact subset in $V_r(\A_\fin)$.
We set
\[
L_r=V_r(\Q)\cap(V_r(\R)U_r),\quad \tilde\Gamma_r=\GL(r,\Z) \quad \text{and} \quad \Gamma_r=\SL(r,\Z).
\]
Since $L_r$ is $\tilde\Gamma_r$-invariant, we may write $L_r\cap\Omega_r/\tilde\Gamma_r$ (resp. $L_r\cap\Omega_r/\Gamma_r$) for a complete set of representative elements of $\tilde\Gamma_r$-orbits (resp. $\Gamma_r$-orbits) in $L_r\cap\Omega_r$.
Let
\[
\tilde\varepsilon_r(x)=|\{\gamma\in\tilde\Gamma_r\mid x\cdot \gamma=x\}| \quad \text{and} \quad \varepsilon_r(x)=|\{\gamma\in\Gamma_r\mid x\cdot \gamma=x\}|.
\]
A Shintani zeta function $\zeta_\Shin(\phi_r,s)$ is defined by
\[
\zeta_\Shin(\phi_r,s)=\sum_{x\in L_r\cap\Omega_r/\tilde\Gamma_r}\frac{\phi_r(x)}{ \tilde\varepsilon_r(x)\, \det(x)^s} \qquad (s\in\C).
\]
In \cite[p.63]{Shintani}, Shintani defined the zeta function $\zeta^*_r(s)$ for $L_r\cap\Omega_r/\Gamma_r$ and $\varepsilon_r(x)$.
However, it is obvious that $\zeta_\Shin(\phi_r,s)$ is compatible with his definition, because it satisfies
\[
\zeta_\Shin(\phi_r,s)=\frac{1}{2}\sum_{x\in L_r\cap\Omega_r/\Gamma_r}\frac{\phi_r(x)}{ \varepsilon_r(x)\, \det(x)^s} .
\]
It is known that $\zeta_\Shin(\phi_r,s)$ absolutely converges for $\mathrm{Re}(s)>\frac{r+1}{2}$ and $\zeta_\Shin(\phi_r,s)$ is meromorphically continued to the whole $s$-place (cf. \cite{Shintani,Yukie}).
Especially, $\zeta_\Shin(\phi_r,s)$ is holomorphic except for possible simple poles at $s=1,\frac{3}{2},\dots,\frac{r}{2},\frac{r+1}{2}$ (cf. \cite[Theorem 5]{Shintani}).
For the case $r=0$, we set $\zeta_\Shin(\phi_r,s)=1$.

In the previous section, we have chosen the test function $\phi_{0,r}^*$ by \eqref{s4tf}.
A test function $\phi_{0,r}$ on $V_r(\A)$ is also fixed by \eqref{e32} (or \eqref{sme2}).
It is clear that $\phi_{0,r}$ is $\prod_{v<\inf}\GL(r,\Z_v)$-invariant and this zeta function $\zeta_\Shin(\phi_{0,r},s)$ is the same as that of \eqref{sme3}.
The following is our main theorem.
\begin{thm}\label{t37}
Let $\Gamma=G(\Q)\cap(G(\R)K_0)$ and let $\rho$ denote a finite dimensional irreducible polynomial representation of $\GL(n,\C)$ corresponding to $\bk=(k_1,k_2,\dots,k_n)$.
If we assume that $\Gamma$ satisfies Condition \ref{c21} and $k_n>n+1$, then we have
\begin{multline*}
\dim S_\rho(\Gamma)=\\
\frac{[\bK_\fin:K_0\cap\bK_\fin]}{[K_0:K_0\cap\bK_\fin]}\sum_{r=0}^n \zeta_\Shin(\phi_{0,r},r-n) \, \bC(\bk,n,r) \prod_{j=1}^{n-r}\frac{(-1)^j(j-1)!}{(2j-1)!}\zeta(1-2j).
\end{multline*}
Here, the factor $\bC(\bk,n,r)$ was defined by \eqref{e35} and \eqref{e37} and it has the properties \eqref{e36} and \eqref{e38}.
\end{thm}
\begin{proof}
Let $\sigma$ denote the holomorphic discrete series of $G(\R)$ corresponding to $\bk$.
If $k_n>2n$, $\sigma$ is integrable (cf. \cite{HS}).
By Proposition \ref{lac} and Lemma \ref{l31} (ii), we can take a sufficiently large natural number $\mathfrak{N}>2n$ such that $\fs$ satisfies Condition \ref{cuni} if $k_n>\mathfrak{N}$.

Assume that $k_n$ is greater than $\mathfrak{N}$ and $\Gamma$ satisfies Condition \ref{c21}.
Then, it follows from Proposition \ref{p33} and Lemma \ref{l47} that
\begin{multline}\label{te1}
 \dim S_\rho(\Gamma)=m(\sigma,\Gamma)=\sum_{r=0}^n\lim_{T\to 0}Z_{C_r}^T(\fs h) \\
 =\vol(K_0)^{-1} \sum_{r=0}^n \vol(\Sp(n-r)) \, d_\tau^{-1} d_\sigma \, Z_r(\Phi_{0,\bk,n,r},r-n) 
\end{multline}
Since the zeta integral $Z_r(\Phi_{0,\bk,n,r},s)$ is decomposed into the product of the zeta function $\zeta_\Shin(\phi_{0,r},s)$ and the local integral of $\phi_{\bk,n,r}(x)$ over $\R$ in the usual manner (see, e.g., \cite[p.600--601]{Saito2}), we find
\begin{multline*}
\eqref{te1}=\vol(K_0)^{-1}\sum_{r=0}^n  d_\tau^{-1}\, d_\sigma\,\int_{V_r(\R)}\phi_{\bk,n,r}(x)\, \det(x)^{\frac{r-1}{2}-n}\d x  \\
\times \zeta_\Shin(\phi_{0,r},r-n)\times \prod_{j=1}^{n-r}\frac{\zeta(2j)\, \Gamma(j)}{\pi^j}
\end{multline*}
by change of variable as in \cite[p.62--63]{Shintani} and the normalizations of measures given in Sections \ref{s3.4} and \ref{s3ze}.
By the normalization $\vol(\bK_\fin)=1$, it is obvious that
\[
\vol(K_0)=\frac{[K_0:K_0\cap\bK_\fin]}{[\bK_\fin:K_0\cap\bK_\fin]}.
\]
Using the functional equation $\Gamma(\frac{s}{2})\zeta(s)=\pi^{s-\frac{1}{2}}\Gamma(\frac{1-s}{2})\zeta(1-s)$, we get
\[
\frac{\zeta(2j)\, \Gamma(j)}{\pi^j}=2^{2j-1}\pi^j \times \frac{(-1)^j(j-1)!}{(2j-1)!}\zeta(1-2j).
\]
Therefore, the dimension formula is proved for $k_n>\mathfrak{N}$.

We will extend the range for $\bk$ using Proposition \ref{s4po}.
Set $k_j'=k_j-k_n$ $(1\leq j\leq n)$.
We have the parameter $\bk'=(k_1',k_2',\dots,k_n')\in\Z^{\oplus n}$ $(k_1'\geq k_2'\geq \cdots \geq k_{n-1}'\geq k_n'=0)$.
From now on, we fix the parameter $\bk'$.
We have $\bk=\bk'+(k_n,k_n,\dots,k_n)$ and we will move only $k_n$.
By Proposition \ref{s4po}, $\dim S_\rho(\Gamma)$ is regarded as a polynomial of $k_n$ with constant coefficients for the range $k_n>n+1$, i.e., there exists a polynomial $g_{\Gamma,\bk'}(X)$ such that $\dim S_\rho(\Gamma)=g_{\Gamma,\bk'}(k_n)$ for $k_n>n+1$.
Assume that the above dimension formula is a rational polynomial of $k_n$ with constant coefficients, i.e., there exist polynomials $f_{\Gamma,\bk',1}(X)$ and $f_{\Gamma,\bk',2}(X)$ such that $\dim S_\rho(\Gamma)=f_{\Gamma,\bk',1}(k_n)/f_{\Gamma,\bk',2}(k_n)$ for $k_n>\mathfrak{N}$.
Then, we obtain $f_{\Gamma,\bk',1}(X)/f_{\Gamma,\bk',2}(X)=g_{\Gamma,\bk'}(X)$ as polynomial  because $f_{\Gamma,\bk',1}(k_n)=g_{\Gamma,\bk'}(k_n)f_{\Gamma,\bk',2}(k_n)$ hold for any $k_n>\mathfrak{N}$.
Therefore, the dimension formula is valid for any $k_n>n+1$.
Since we may choose arbitrary $\bk'$, it is sufficient to prove $C(\bk,n,r)$ is a rational polynomial of $k_n$.
In \eqref{e36}, we can set $\bm=(m_1',m_2',\dots,m_r')+(k_n,k_n,\dots,k_n)$.
Then, we have
\[
\frac{\Gamma_{\Omega_r}(\bm+\frac{r-1}{2}-n)}{\Gamma_{\Omega_r}(\bm)}=\prod_{j=1}^r\frac{ \Gamma(k_n+m_{r-j+1}'-n+\frac{j-1}{2})}{\Gamma(k_n+m_j'-\frac{j-1}{2})},
\]
\[
\frac{\Gamma(k_n+m_{r-j+1}'-n+\frac{j-1}{2})}{\Gamma(k_n+m_j'-\frac{j-1}{2})}= \prod_{t=1}^{m_j'-m_{r-j+1}'+n-j+1} (k_n+m_j'-\frac{j-1}{2}-t)^{-1}
\]
if $m_j'-m_{r-j+1}'+n-j+1\geq 0$, and
\[
\frac{\Gamma(k_n+m_{r-j+1}'-n+\frac{j-1}{2})}{\Gamma(k_n+m_j'-\frac{j-1}{2})}=\prod_{t=1}^{m_{r-j+1}'-m_j'-n+j-1} (k_n+m_{r-j+1}'-n+\frac{j-1}{2}-t)
\]
if $m_j'-m_{r-j+1}'+n-j+1\leq 0$.
This means that $C(\bk,n,r)$ is a rational polynomial of $k_n$ by \eqref{e36}.
Thus, the proof is completed, because the factors $m_{\bk,\bb}^{n,r}$ and $n_{\bb,\bm,r}$ are independent of $k_n$.
\end{proof}

For a natural number $N$, we have defined the open compact subgroup $K(N)$ of $G(\A_\fin)$ in Section \ref{smt}.
The principal congruence subgroup $\Gamma_n(N)$ was defined by $K(N)$ and \eqref{sme4}.
It is well-known that $\Gamma_n(N)$ is neat if $N>2$ (see, e.g., \cite[p.118]{Borel1}, \cite[Lemma 2]{Morita}).
As a special case of $\zeta_\Shin(\phi_{0,r},s)$, the zeta function $\zeta_\Shin(L_r^*,s)$ was defined by \eqref{sme5}.
The following is a corollary of Theorem \ref{t37}.
It is derived by putting $K_0=K(N)$.
\begin{cor}\label{corf}
Let $\rho$ denote a finite dimensional irreducible polynomial representation of $\GL(n,\C)$ corresponding to $\bk=(k_1,k_2,\dots,k_n)$.
If $N>2$ and $k_n>n+1$, then we have
\begin{align*}
\dim S_\rho(\Gamma_n(N))=&[\Gamma_n(1):\Gamma_n(N)]\sum_{r=0}^n \zeta_\Shin(L_r^*,r-n)\,  2^{r-r^2+rn} N^{\frac{r(r-1)}{2}-rn}\\
&\times \bC(\bk,n,r) \times \prod_{l=1}^{n-r}\frac{(-1)^l (l-1)!}{(2l-1)!}\zeta(1-2l),
\end{align*}
where $\bC(\bk,n,r)$ was defined by \eqref{e35} and \eqref{e37} and it satisfies the properties \eqref{e36} and \eqref{e38}.
In particular, $\bC(\bk,n,r)$ is computable if one fixes the parameters $n$, $r$, and $(k_1-k_n, k_2-k_n,\dots,k_{n-1}-k_n,0)$.
\end{cor}

In Section \ref{sef}, we have already given an explicit dimension formula for general degree $n$ and the scalar valued case $(k_1,k_2,\dots,k_n)=(k,k,\dots,k)$.
As an example of the vector valued case, we give an explicit dimension formula for degree three and the standard representation.
\begin{exa}\normalfont
Let $n=3$, $k>4$ and $N>2$.
Assume that $\rho$ corresponds to $(k+1,k,k)$.
From Theorems \ref{t13} and \ref{t14} and Corollary \ref{corf} we deduce
\begin{multline*}
\dim S_\rho(\Gamma_3(N))=[\Gamma_3(1):\Gamma_3(N)] \times \\
\left\{ \frac{3\cdot(2k)(2k-2)(2k-3)(2k-4)(2k-5)(2k-6)}{2^{16}\cdot 3^6\cdot 5^2\cdot 7}  -\frac{6k-10}{2^{10}\cdot 3^2\cdot 5\cdot N^5}+\frac{3}{2^8\cdot 3^3\cdot N^6}  \right\}.
\end{multline*}

\vspace{1mm}
\noindent
Numerical examples of $\dim S_\rho(\Gamma_3(N))$ where $\rho$ corresponds to $(k+1,k,k)$. \\
\begin{tabular}{|c@{\, \vrule width1.5pt \,\,\,}cccc|} \hline
$N$ \begin{picture}(4,4)(0,0) \put(4,-1){ \line(-3,2){6} } \end{picture} $k$ &   5 & 6 & 7 & 8    \\ \noalign{\hrule height 1.5pt}
3 & 210210 & 1178268 & 4357626 & 12622974  \\ \hline
4 & 72432640 & 395006976 & 1451584512 & 4196369408  \\ \hline
5 & 10968753250 & 59435649000 & 218097857250 & 630209284250  \\ \hline
\end{tabular}
\end{exa}

\vspace{5mm}
\noindent
{\bf Acknowledgement}.
The author thanks Prof. Tobias Finis, Prof. Kaoru Hiraga, Prof. Werner Hoffmann, Prof. Tomoyoshi Ibukiyama, Prof. Takashi Sugano, and Prof. Masao Tsuzuki for kind advice and helpful discussions.
In particular, Prof. Finis advised the author to change a pseudo-coefficient into a spherical trace function via the trace Paley-Wiener theorem.
This idea is a key point of the proof.
The author is partially supported by JSPS Grant-in-Aid for Scientific Research No.~26800006,  No.~25247001, and No.~15K04795.

\end{document}